\documentclass[10pt,reqno]{amsart}
\usepackage{bbm}
\usepackage{mathrsfs}
\usepackage{amsfonts} %%% i.e. use 12pt type
\usepackage[dvipsnames,usenames]{color}
\usepackage{graphicx,latexsym,bm,amsmath,amssymb,verbatim,multicol,lscape}
\usepackage{enumerate}
\usepackage{cite}
\usepackage[left=2cm,right=2cm,top=2.5cm,bottom=2.5cm]{geometry}
\newtheorem{thm}{Theorem} [section]
\newtheorem{lem}{Lemma}[section]

\newtheorem{cor}{Corollary}[section]

\theoremstyle{definition}

\theoremstyle{remark}
\newtheorem{rem}{Remark}[section]
\newtheorem{cnj}{Conjecture}[section]
\numberwithin{equation}{section}

\allowdisplaybreaks
\begin{document}

\title{The 3-adic valuations of Stirling numbers of the first kind }

\author[M. Qiu]{Min Qiu}
\address{School of Science, Xihua University, Chengdu 610039, P.R. China}
\email{minqiu@mail.xhu.edu.cn}
\author[Z.B. Lin]{Zongbing Lin$^*$}
\address{School of Mathematics and Computer Science, 
Panzhihua University, Panzhihua 617000, P.R. China}
\email{zongbinglin@sohu.com}
\author[L. Chen]{Long Chen}
\address{School of Mathematics and Computer Science, 
Panzhihua University, Panzhihua 617000, P.R. China}
\email{chenlongscumath@126.com}

\thanks{$^*$ Corresponding author.}
% corresponding author
%\corraddr{zongbinglin@sohu.com}

\begin{abstract}
Let $v_3$ denote the usual $3$-adic valuation, 
and let $s(n, k)$ be the unsigned Stirling 
number of the first kind. In this paper, for $a\in\{1,2\}$, 
we determine the values of $v_3(s(a3^n, k))$ for all $1\le k\le a3^n$. 
More precisely, for each admissible pair $(m, k)$, 
we obtain an explicit formula for $v_3(s(a3^n, a3^m-k))$. 
The proof combines properties of the 
$m$-th Stirling numbers of the first kind with a detailed analysis 
of the relevant $3$-adic orders. As a consequence,  
we prove the case $p=3$ of a conjecture of Hong and Qiu proposed in 2020.
We also derive formulas near the diagonal, comparison results 
for the adjacent orders $a3^n$ and $a3^n+1$, sharp upper bounds 
for the families $v_3(s(3^n, k))$ and $v_3(s(2\cdot3^n, k))$, 
and partial confirmations of conjectures of Lengyel 
and of Leonetti and Sanna.
\end{abstract}

\keywords{
3-adic valuation, 3-adic analysis, Stirling number of the
first kind, the $m$-th Stirling number of the first kind,
elementary symmetric function.
\newline
\textbf{Mathematics Subject Classification:} 11B73, 11A07}

\maketitle

\section{Introduction}

\subsection{Background and motivation}

For integers $n\ge 1$ and $0\le k\le n$, we write $s(n, k)$ for the 
unsigned Stirling number of the first kind and $S(n, k)$ for 
the Stirling number of the second kind. These two families 
are characterized by  
\[
(x)_n=\sum_{k=0}^n (-1)^{n-k}s(n, k)x^k
\quad \text{and}\quad 
x^n=\sum_{k=0}^n S(n, k)(x)_k,
\]
respectively, where $(x)_n=x(x-1)(x-2)\cdots(x-n+1)$ 
denotes the falling factorial. Equivalently, $s(n, k)$ 
counts permutations of an $n$-element set 
by their number of cycles, while $S(n, k)$ 
counts partitions of such a set into $k$ nonempty blocks.

Throughout this paper, $p$ denotes a prime. 
The arithmetic properties of Stirling numbers have been studied extensively. 
For Stirling numbers of the second kind, much of the 
literature concerns congruences, periodicity, and $p$-adic valuations; 
see, for example, \cite{[Am],[CF],[Da],[HZ1], 
[TL],[TL2],[TL3],[Lu],[Mi],[Wan],[HZ2],[HZ3]}. 
For Stirling numbers of the first kind, the corresponding 
valuation problem is less completely understood. 
Stirling numbers of the first kind are closely related to 
generalized harmonic sums. For $0\le k\le n$, define   
\[
H(n, 0):= 1, \quad
H(n, k) := \sum_{1 \le i_1 < \cdots < i_k \le n} \frac{ 1 } { i_1 \cdots i_k }. 
\]
Then $H(n, k)$ is the $k$-th elementary symmetric function of 
$1, 1/2, \ldots, 1/n$. The identity 
\begin{equation}
s(n+1, k+1)=n!H(n, k)\label{1.1}
\end{equation}
links Stirling numbers of the first kind with 
these reciprocal elementary symmetric functions; see \cite{[LS]}. 
Related $p$-adic properties of harmonic numbers and 
reciprocal elementary symmetric functions have been investigated in  
\cite{[Alt], [Bo], [CCG], [CDFG], [CT], [EN], [Es], [HW], [Ka], [LHQW], [N], [Sa], [T], [WH], [WC]}. 

We write $v_p$ for the usual $p$-adic valuation. 
Thus, for a nonzero integer $n$, 
$v_p(n)$ is the largest integer $r$ such that $p^r\mid n$, and we set 
$v_p(0)=\infty$. For a rational number $\frac{n_1}{n_2}$, we define 
$v_p(\frac{n_1}{n_2})=v_p(n_1)-v_p(n_2)$,  
where $n_1, n_2\in \mathbb{Z}$ and $n_2\ne 0$. 
By Legendre's formula, 
$v_p(n!)=\frac{n-d_p(n)}{p-1},$
where $d_p(n)$ is the sum of the base-$p$ digits of $n$.  
Hence (\ref{1.1}) shows that the study of $v_p(H(n,k))$ is 
equivalent to that of $v_p(s(n+1,k+1))$.

\subsection{Known results and conjectures}

We next place the present work in the context of several earlier 
results and conjectures. Lengyel proved in \cite{[L]} that, for 
every fixed positive integer $k$, the sequence $v_p(s(n, k))$ 
tends to infinity as $n\to \infty$. More precisely, for each pair 
$(k, p)$ there exist constants $c=c(k, p)> 0$ and $n_0=n_0(k, p)$  
such that $v_p(s(n, k))\ge cn$ whenever $n\ge n_0$. 
Together with (\ref{1.1}), this gives corresponding lower estimates 
for $v_p(H(n, k))$. Leonetti and Sanna proposed the following 
conjecture and confirmed it for some special cases.   
\begin{cnj}[{\cite[Conjecture 1.1]{[LS]}}]\label{cnj1}
For any prime number $p$ and any integer $k \ge 1$, there exists a constant
$c = c(p, k) > 0$ such that
\[
v_p(H(n,k)) < -c \log n
\]
for all sufficiently large integers $n$.
\end{cnj}

Several more explicit formulas are known in special situations. 
Komatsu and Young \cite{[KY]} considered integers of the form 
$n=kp^r+m$, where $0\le m<p^r$, and obtained 
$v_p(s(n+1,k+1))=v_p(n!)-v_p(k!)-kr$. 
For powers of $2$, Qiu and Hong \cite{[QH]} determined 
$v_2(s(2^n, k))$ for the full range $1\le k\le 2^n$. 
In a subsequent work, Hong and Qiu \cite{[HQ]} proposed 
a general conjectural formula for 
$v_p(s(ap^n, ap^m-k))$. To state it, let $\langle k\rangle$ 
be the integer satisfying $0\le \langle k\rangle\le p-2$
and $k\equiv \langle k\rangle \pmod {p-1}$. Put 
$\epsilon_k :=0$ if $k$ is even, and $\epsilon_k :=1$ if $k$ is odd.  
For a real number $x$, let $\lfloor x\rfloor$ denote the greatest
integer not exceeding $x$, and let $B_n$ be the $n$-th Bernoulli number. 
Hong and Qiu proposed the following conjecture. 
\begin{cnj}[{\cite[Conjecture 5.1]{[HQ]}}]\label{cnj2}
Let $p$ be an odd prime. Let $a, n, m, k$ be positive integers such that
$1\le a\leq p-1$, $1\le m\le n$ and $2\le k\le ap^n-2$.
Then each of the following is true:

\textup{(i)} If $2\le k \le a(p-1)p^{m-1}+1<ap^m$, then
\[
v_p(s(ap^n, ap^m-k))=\frac{a}{p-1}(p^n-p^m)-(n-m)(ap^m-k)
+m+(m+v_p(k))\epsilon_k+T_k,
\]
where
\[
T_k:=\begin{cases}
 -1- v_p(\lfloor\frac{k}{2}\rfloor),
 & \text{if }\ k\equiv \epsilon_k\pmod{p-1}; \\
v_p(B_{2\lfloor\frac{\langle k\rangle}{2}\rfloor}),
& \text{if }\ k\not\equiv \epsilon_k\pmod{p-1}.
\end{cases}
\]

\textup{(ii)} If $a\ge 4$ and $a(p-1)+2\le k\le ap-2$, then
\[
v_p(s(ap^n, ap-k))\ge\frac{a}{p-1}(p^n-p)-(n-1)(ap-k)+a+k-ap.
\]

\end{cnj}

\noindent Hong and Qiu also presented results on 
$v_p(s(ap, k))$ and $v_p(s(ap^n, ap^n-k))$ when $p\ge 5$ in \cite{[HQ]}. 

The results above indicate that the $p$-adic valuations of Stirling numbers 
of the first kind remain far from fully understood, 
especially in explicit form. In this paper, we focus on 
the case $p=3$. Some isolated $3$-adic formulas were 
previously known. For example,  
Lengyel \cite{[L]} proved that $v_3(s(3^n, 2))=\frac{1}{2}(3^n+3)-2n$,
$v_3(s(3^n, 3))=\frac{1}{2}(3^n+3)-3n$ and $v_3(s(2\cdot3^n, 2))=3^n-2n-1$.
Komatsu and Young \cite{[KY]}, using a $p$-adic Newton polygon argument, 
proved that $v_3(s(3^n, 3^m))=\frac{1}{2}(3^n-3^m)-3^m(n-m)$ and
$v_3(s(2\cdot3^n,2\cdot3^m))=3^n-3^m-2\cdot3^m(n-m)$. 
Our result gives a uniform formula that covers 
these cases as special instances. 
For $a\in\{1,2\}$, our aim is to determine 
$v_3(s(a3^n, i))$ for every $1\le i\le a3^n$. 
The two boundary values follow from the standard identities for Stirling
numbers of the first kind; see, for example, \cite{[LC]}. Indeed, 
\[
v_3(s(a3^n, a3^n))=v_3(1)=0\quad \text{and} \quad
v_3(s(a3^n, a3^n-1))=v_3\Big(\binom{a3^n}{2}\Big)
=v_3\Big(\frac{a3^n(a3^n-1)}{2}\Big)=n.
\]
Thus it remains to understand the remaining indices. 
If $1\le i\le a3^n-2$, then $i$ can be written in the form 
$i=a3^m-k$ with $(m,k)\in T_{a,n}$, where 
\[
T_{a,n}:=\{(m,k) : 1\le m\le n,\  2\le k\le 2a3^{m-1}+1,
\  k<a3^m, m, k\in \mathbb{Z}\}.
\]

\subsection{Main results}

We now state the main results of this paper. The first result gives an 
explicit formula for $v_3(s(a3^n, a3^m-k))$ with 
$a\in\{1,2\}$ and $(m, k)\in T_{a,n}$. 

\begin{thm}\label{thm1}
Let $a\in\{1,2\}$, and let $n, m, k$ be positive integers such that 
$1\le m\le n$, $2\le k\le 2a3^{m-1}+1$ and $k<a3^m$. Then 
\begin{equation}\label{1.4}
v_3(s(a3^n,a3^m-k))
=\frac{a}{2}(3^n-3^m)-(n-m)(a3^m-k)
+m-1-v_3\Big(\Big\lfloor\frac{k}{2}\Big\rfloor\Big)+(m+v_3(k))\epsilon_k,
\end{equation}
where $\epsilon_k =0$ if $k$ is even and $\epsilon_k =1$ if $k$ is odd.
\end{thm}

Since $p-1=2$ when $p=3$, every integer $k$ satisfies 
$k\equiv \epsilon_k \pmod 2$. Hence Theorem \ref{thm1} 
confirms the first part of Conjecture \ref{cnj2} in the case $p=3$; 
the second part of that conjecture is then vacuous. 
Moreover, Lengyel proved in \cite{[L]} that, for any prime $p$, any integer
$a\ge 1$ with $(a,p)=1$, and any even $k\ge 2$ with the condition
\begin{equation}\label{1.5}
\text{there exists } n_1> 3\log_p{k}+\log_p{a}\ \text{such that }
v_p(s(ap^{n_1}, ap^{n_1}-k))< n_1
\end{equation}
one has
$v_p(s(ap^{n+1},ap^{n+1}-k))=v_p(s(ap^n,ap^n-k))+1$
for all $n\ge n_1$. The same conclusion also holds 
for $k=1$ with $n_1=1$. Lengyel further expected that 
(\ref{1.5}) holds for all even integers $k\ge 2$. 
For odd integers $k\ge 3$, he proposed the following conjecture.

\begin{cnj}[{\cite[Conjecture 3.1]{[L]}}]\label{cnj3}
Let $p$ be a prime, let $a\ge 1$ be an integer with $(a, p) = 1$, 
and let $k\ge 3$ be odd. Then there exists $n_1=n_1(p, k)$ 
such that, for all $n\ge n_1$, 
\[
v_p(s(ap^{n+1}, ap^{n+1}-k)) = v_p(s(ap^n, ap^n-k))+2.
\]
\end{cnj}

By taking $m=n$ in Theorem \ref{thm1}, 
we obtain the following explicit formula near the diagonal.

\begin{cor}\label{cor1}
Let $a\in\{1,2\}$ and let $n, k\in \mathbb{Z}^+$ with 
$2\le k\le 2a3^{n-1}+1$ and $k<a3^n$. Then
\[
v_3(s(a3^n, a3^n-k)) =\begin{cases}n-1-v_3(k),
& \ \text{if }\ 2\mid k, \\
2n-1+v_3(k)-v_3(k-1), & \ \text{if }\ 2\nmid k. \end{cases}
\]
\end{cor}

In particular, Corollary \ref{cor1} verifies Lengyel's expected 
condition (\ref{1.5}) for $p=3$, $a\in \{1,2\}$ and all even integers $k\ge 2$, 
once $n$ is sufficiently large. It also gives the predicted increment 
in Lengyel's conjecture for odd $k$ in the same $3$-adic setting. 

The main formula also yields a comparison between adjacent 
orders. Similar phenomena are known in the $2$-adic case: 
Hong, Zhao and Zhao \cite{[HZ1]} proved an adjacent-order 
identity for Stirling numbers of the second kind, and Qiu and 
Hong \cite{[QH]} obtained the corresponding identity 
$v_2(s(2^n+1, k+1))=v_2(s(2^n, k))$ for Stirling numbers of 
the first kind. The next theorem gives the corresponding 
$3$-adic analogue, with an additional parity distinction.

\begin{thm}\label{thm2}
Let $a\in\{1,2\}$ and let $n, k\in \mathbb{Z}^+$ with $k\le a3^n$. Then
\begin{equation}\label{1.7}
\begin{aligned}
v_3(s(a3^n+1, k+1)) &=v_3(s(a3^n, k)), \quad && \text{if }\ 2\mid (k-a), \\
v_3(s(a3^n+1, k+1)) &\ge v_3(s(a3^n, k+1))+n, \quad && \text{if }\ 2\nmid (k-a).
\end{aligned}
\end{equation}
\end{thm}

Another consequence concerns the maximal possible 
valuation in each family. While the $2$-adic situation has 
a particularly simple upper bound, the $3$-adic case separates 
according to whether the order is $3^n$ or $2\cdot 3^n$. 

\begin{thm}\label{thm3}
Let $n, k\in \mathbb{Z}^+$ with $k\le 3^n$. Then
\begin{equation}\label{1.8}
v_3(s(3^n,k))
\le \begin{cases}
v_3(s(3^n,1))=\dfrac{1}{2}(3^n-2n-1), & \text{if }\ n\ge 3,\\
v_3(s(3^2, 6))=4, & \text{if }\ n=2,\\
v_3(s(3,2))=1, & \text{if }\ n=1.\end{cases}
\end{equation}
\end{thm}

\begin{thm}\label{thm4}
Let $n, k\in \mathbb{Z}^+$ with $k\le 2\cdot 3^n$. Then
\begin{equation}\label{1.9}
v_3(s(2\cdot3^n, k))
 \le \begin{cases}
 v_3(s(2\cdot3^n,1))=3^n-n-1, & \text{if }\ n\ge 2,\\
v_3(s(2\cdot3, 3))=2, & \text{if }\ n=1.\end{cases}
\end{equation}
\end{thm}

Finally, by translating the preceding estimates through (\ref{1.1}), 
we obtain the following consequence for the elementary symmetric 
functions $H(n, k)$, which partially confirms Conjecture \ref{cnj1}.
\begin{cor}\label{cor2}
Let $a\in \{1,2\}$. Let $n$ and $k$ be positive integers 
such that $n\ge 3$, $k\le a3^n$ and $2\mid (k-a)$. Then 
$v_3(H(a3^n,k)) \leq -n.$
\end{cor}

Thus our results provide evidence for conjectures of Hong 
and Qiu, Lengyel, and Leonetti and Sanna in the $3$-adic setting.

\subsection{Proof strategy and organization}

We briefly describe the main strategy of the proof of Theorem \ref{thm1}. 
The main difference from the $2$-adic case lies in the structure of the
parameter range. For powers of $2$, one deals with the single family
$s(2^n,k)$. In the present $3$-adic setting, the two families
$s(3^n,k)$ and $s(2\cdot 3^n,k)$ have to be treated simultaneously.
In addition, the valuation formula contains an essential parity correction
through $\epsilon_k$ and $v_3(\lfloor k/2\rfloor)$, which is responsible
for the finer case decomposition in the proof. 

The proof proceeds by induction on $n$.  
We first reduce the odd values of $k$ to the even values. 
Thus it remains to establish the formula for even $k$. 
Under the induction hypothesis for the family $s(3^n, 3^m-k)$, 
we first derive the corresponding formula for $v_3(s(2\cdot 3^n, 2\cdot 3^m-k))$. 
This is the first key step. We then use this formula, 
together with the convolution identity for Stirling numbers 
of the first kind and estimates for the $m$-th Stirling numbers,  
to obtain the formula for $v_3(s(3^{n+1}, 3^m-k))$. This 
completes the induction.  

The remainder of this paper is organized as follows. Section 2 collects the
auxiliary results needed in the sequel, including the reduction from odd
values of $k$ to even values of $k$ and comparison estimates for the
$m$-th Stirling numbers of the first kind. In Section 3, we establish the
key inductive step from the order $3^n$ to the order $2\cdot3^n$.
Section 4 completes the induction by passing from the order
$2\cdot3^n$ to the order $3^{n+1}$, thereby proving Theorem \ref{thm1}.
Section 5 applies the main formula to prove Theorems \ref{thm2}--\ref{thm4}
and Corollary \ref{cor2}. Section 6 concludes the paper with a brief
discussion of the restriction to the prime $3$ and the obstruction to a
direct extension of the present method to arbitrary odd primes.

\section{Preliminaries and reductions}

This section collects the auxiliary results used in the proof of Theorem \ref{thm1}.
We first recall a standard identity for Stirling numbers of the first kind.
We then prove a reduction showing that the odd case follows from the even case.
Finally, we introduce the $m$-th Stirling numbers of the first kind and derive
the comparison estimates needed in the induction.

\subsection{Basic identities}

We begin with the following standard identity for Stirling numbers of the first kind.

\begin{lem}[\cite{[VA]}]\label{lem2.1}
Let $n$ and $k$ be positive integers. If $n+k$ is odd, then
\[
s(n,k)=\frac{1}{2}\sum_{i=k+1}^{n}(-1)^{n-i}s(n,i)\binom{i-1}{i-k}n^{i-k}.
\]
\end{lem}

This identity will be used to relate neighboring values of the 
parameter $k$ and to reduce the odd case to the even case. 

\subsection{Reduction from odd $k$ to even $k$}

The aim of this subsection is to show that the odd case of formula (\ref{1.4}) 
is a consequence of the even case. We first establish a valuation 
relation between the adjacent Stirling numbers $s(a3^n, a3^n-2t)$ and 
$s(a3^n, a3^n-2t-1)$, under the assumption that (\ref{1.4}) holds for 
the relevant even indices. We then apply this relation to obtain formula 
(\ref{1.4}) for odd $k$. 

%*******************************Lemma 2.2*******************************

\begin{lem}\label{lem2.2}
Let $a\in\{1,2\}$, and let $n$ be a positive integer. Assume that 
(\ref{1.4}) holds for all integers $m$ and even $k$ with $(m,k)\in T_{a,n}$. 
Then, for every integer $t$ satisfying $0\le t\le \frac{a3^n-2}{2}$, we have
\begin{align}\label{2.1}
v_3(s(a3^n,a3^n-2t-1))=v_3(s(a3^n,a3^n-2t))+v_3(2t+1)+n.
\end{align}
\end{lem}

\begin{proof}
We prove (\ref{2.1}) by induction on $t$. The case $t=0$ is immediate, 
since
\[
v_3(s(a3^n,a3^n-1))=n=v_3(s(a3^n,a3^n))+v_3(1)+n.
\]
Now let $1\le t\le \frac{a3^n-2}{2}$.
Assume that (\ref{2.1}) holds for any nonnegative integer $e$ with $e\le t-1$.
We show that (\ref{2.1}) is still true for $t$. 

The idea is to isolate the principal contribution to $s(a3^{n},a3^n-2t-1)$. 
By Lemma \ref{lem2.1}, we have 
\begin{align}
s(a3^{n},a3^n-2t-1)
&=\frac{1}{2}\sum_{i=a3^n-2t}^{a3^{n}}(-1)^{a3^{n}-i}s(a3^{n},i)
\binom{i-1}{i-a3^n+2t+1}(a3^{n})^{i-a3^n+2t+1}\notag\\
&=\frac{a3^{n}}{2}s(a3^{n},a3^n-2t)(a3^n-2t-1) + \frac{a3^{n}}{2}L,\label{2.2}
\end{align}
where
\begin{align*}
L=\sum_{i=a3^n-2t+1}^{a3^{n}}(-1)^{a3^{n}-i}\binom{i-1}{i-a3^n+2t+1}L_i
\end{align*}
with $L_i:=s(a3^n,i)(a3^{n})^{i-a3^n+2t}$. 

Thus the first term on the right-hand side of (\ref{2.2}) is the desired 
principal term. It remains to show that the term involving $L$ has strictly 
larger 3-adic valuation. We shall prove the following estimate: 
\begin{align}
v_3(L_i)\ge v_3(s(a3^n,a3^n-2t))+v_3(2t+1)+1 \label{2.3}
\end{align}
for every integer $i$ with $a3^n-2t+1\le i\le a3^n$. If (\ref{2.3}) holds, 
then the isosceles triangle principle (see, for example, \cite{[K]}) gives 
\begin{align}
v_3(L)\ge v_3(s(a3^n,a3^n-2t))+v_3(2t+1)+1
&=v_3(s(a3^n,a3^n-2t))+v_3(a3^n-2t-1)+1\label{2.4}
\end{align}
since $3\le 2t+1\le a3^n-1$. 
Note that $v_3( \frac{a3^{n}}{2})=n$. 
Hence (\ref{2.1}) follows immediately from (\ref{2.2}) to (\ref{2.4}).
So to finish the proof of Lemma \ref{lem2.2},
it remains to prove (\ref{2.3}).
This will be done in what follows.

We first reduce the estimate for odd indices $i$ to the corresponding 
estimate for even indices. Every integer $i$ with $a3^n-2t+1\le i\le a3^n$ 
can be written in one of the two forms 
\[
i=a3^n-2t+2r\quad \text{or}\quad i=a3^n-2t+2r-1
\]
for some integer $r$ with $1\le r\le t$. By the induction hypothesis, 
for $1\le r\le t$,  
\begin{align}\label{2.5}
v_3(s(a3^n,a3^n-2t+2r-1))
&= v_3(s(a3^n,a3^n-2(t-r)-1))\notag\\
&=v_3(s(a3^n,a3^n-2(t-r)))+v_3(2(t-r)+1)+n.
\end{align}
Hence it is enough to prove that 
\begin{align}
v_3(L_{a3^n-2t+2r})
=v_3(s(a3^n,a3^n-2t+2r))+2rn
&\ge  v_3(s(a3^n,a3^n-2t))+v_3(2t+1)+1 \label{2.6}
\end{align}
for every $1\le r\le t$. Indeed, once (\ref{2.6}) holds, (\ref{2.5}) gives 
\begin{align*}
v_3(L_{a3^n-2t+2r-1})
&=v_3(s(a3^n,a3^n-2t+2r-1))+(2r-1)n\\
&=v_3(s(a3^n,a3^n-2t+2r))+2rn+v_3(2(t-r)+1)\\
&\ge  v_3(s(a3^n,a3^n-2t))+v_3(2t+1)+1.
\end{align*}
Thus (\ref{2.3}) follows from (\ref{2.6}).

We now prove (\ref{2.6}). Since $1\le t\le \frac{a3^n-2}{2}$, we have
$2\le a3^n-2t\le a3^n-2$. Write $a3^n-2t=a3^m-k$, where
$(m, k)\in T_{a, n}$ and $k$ is even. 
Note that $2t\le a3^n-2$ implies $k\le a3^m-2$.
Together with $n\ge m$ and $2\le k\le 2a3^{m-1}$, we obtain 
\[
v_3(2t+1)=v_3(a3^n-a3^m+k+1)=v_3(k+1)\le m.
\]
By the assumed even case of (\ref{1.4}), we get 
\begin{align}\label{2.7}
&v_3(s(a3^n,a3^m-k))
=\frac{a}{2}(3^n-3^m)-(n-m)(a3^m-k)+m-1-v_3(k).
\end{align}
It remains to compare $v_3(L_{a3^n-2t+2r})$ with 
$v_3(s(a3^n,a3^n-2t))+v_3(2t+1)$. 
Note that $a3^n-2t+2\le a3^n-2t+2r\le a3^n$. 
Consider the following two cases.

If $a3^n-2t+2r=a3^n$, then $2r=2t=a3^n-a3^m+k$.
It follows from (\ref{2.7}) that
\begin{align*}
&v_3(L_{a3^n})-v_3(s(a3^n,a3^n-2t))-v_3(2t+1)
=2rn-v_3(s(a3^n, a3^m-k))-v_3(2t+1)\\
&=n(a3^n-a3^m+k)-\frac{a}{2}(3^n-3^m)+(n-m)(a3^m-k)-m+1+v_3(k)-v_3(2t+1)\\
&=\Big(n-\frac{1}{2}\Big)a3^n-\Big(m-\frac{1}{2}\Big)a3^m+(k-1)m+v_3(k)-v_3(2t+1)+1\\
&\ge (k-1)m-v_3(2t+1)+1\ge m-m+1=1.
\end{align*}
Thus (\ref{2.6}) holds in this case.

If $a3^n-2t+2\le a3^n-2t+2r\le a3^n-2$, then we can write
$a3^n-2t+2r=a3^l-j$ for $(l, j)\in T_{a, n}$ with $2\mid j$.
Then $m\le l\le n$, and if $l=m$ then $k-j\ge 2$. 
By (\ref{2.7}) and the assumed even case of (\ref{1.4}), we obtain 
\begin{align*}
&v_3(L_{a3^n-2t+2r})-v_3(s(a3^n,a3^n-2t))-v_3(2t+1)\\
&=v_3(s(a3^n,a3^l-j))+n(a3^l-j-a3^m+k)-v_3(s(a3^n,a3^m-k))-v_3(2t+1)\\
&=\Big(l-\frac{1}{2}\Big)a3^l-\Big(m-\frac{1}{2}\Big)a3^m-lj+l+mk-m+v_3(k)-v_3(j)-v_3(2t+1)=:D_l.
\end{align*} 
It remains to prove $D_l\ge 1$ for $m\le l\le n$. 
If $l=m$, then $2\le j\le 2a3^{m-1}$, and so $v_3(j)\le m-1$. Thus
\[
D_m=m(k-j)+v_3(k)-v_3(j)-v_3(2t+1)\ge 2m-(m-1)-m=1.
\]
If $m+1\le l\le n$, then $2\le j\le 2a3^{l-1}$ and
\begin{align*}
D_l&\ge \Big(l-\frac{1}{2}\Big)a3^l-\Big(m-\frac{1}{2}\Big)a3^m-2la3^{l-1}+l+m-(l-1)-m\\
&=\Big(l-\frac{3}{2}\Big)a3^{l-1}-\Big(m-\frac{1}{2}\Big)a3^m+1\ge 1.
\end{align*}
Thus (\ref{2.6}) holds in all cases. 
Consequently, (\ref{2.3}) is proved, and the induction step follows. 
This completes the proof of Lemma \ref{lem2.2}.
\end{proof}

We next show how Lemma \ref{lem2.2} 
converts the even case of (\ref{1.4}) into the odd case. 
If $k$ is odd, then $k-1$ is even, and the index $a3^m-k$ can be 
written in the form $a3^n-2t-1$. Lemma \ref{lem2.2} then 
relates its valuation to that of the adjacent even index $a3^m-k+1$.

%*******************************Lemma 2.3*******************************

\begin{lem}\label{lem2.3}
Let $a\in\{1,2\}$, and let $n$ be a positive integer.  
Assume that \eqref{1.4} holds for every pair $(m, k)\in T_{a, n}$ with $k$ even. 
Then \eqref{1.4} also holds for every pair $(m, k)\in T_{a, n}$ with $k$ odd.
\end{lem}

\begin{proof}
Let $(m, k)\in T_{a, n}$ with $k$ odd. 
Then $3\le k\le 2a3^{m-1}+1$, $k<a3^m$, and $k-1$ is even. 
Hence $2\le k-1\le 2a3^{m-1}$, and so $(m, k-1)\in T_{a,n}$. 
By the assumed even case of (\ref{1.4}), we have
\begin{align}\label{2.8}
v_3(s(a3^n, a3^m-(k-1)))
=\frac{a}{2}(3^n-3^m)-(n-m)(a3^m-(k-1))+m-1-v_3\Big(\frac{k-1}{2}\Big).
\end{align}
Note that  $a3^m-k=a3^n-2\cdot \frac{a3^n-a3^m+k-1}{2}-1$.
It follows from Lemma \ref{lem2.2} that
\begin{align}\label{2.9}
v_3(s(a3^n, a3^m-k))
&=v_3(s(a3^n, a3^m-k+1))+v_3(a3^n-a3^m+k)+n\notag\\
&=v_3(s(a3^n, a3^m-k+1))+v_3(k)+n
\end{align}
since $n\ge m$ and $3\le k< a3^m$.
By (\ref{2.8}), (\ref{2.9}) and $\left\lfloor\frac{k}{2}\right\rfloor=\frac{k-1}{2}$ one derives that
\begin{align*}
v_3(s(a3^n, a3^m-k))
&=\frac{a}{2}(3^n-3^m)-(n-m)(a3^m-(k-1))+m-1-v_3\Big(\frac{k-1}{2}\Big)+v_3(k)+n\\
&=\frac{a}{2}(3^n-3^m)-(n-m)(a3^m-k)-n+m+m-1-v_3\Big(\Big\lfloor\frac{k}{2}\Big\rfloor\Big)+v_3(k)+n\\
&=\frac{a}{2}(3^n-3^m)-(n-m)(a3^m-k)+m-1-v_3\Big(\Big\lfloor\frac{k}{2}\Big\rfloor\Big)+m+v_3(k).
\end{align*}
Hence (\ref{1.4}) holds for $k$.
This finishes the proof of Lemma \ref{lem2.3}.
\end{proof}

The preceding lemma completes the reduction. 
Lemma \ref{lem2.2} expresses the valuation of the 
odd-indexed term in terms of the adjacent even-indexed 
term, while Lemma \ref{lem2.3} verifies that substituting 
the even formula gives exactly the odd formula. 
Therefore, in the proof of Theorem \ref{thm1}, it remains only to  
establish formula (\ref{1.4}) for even values of $k$.

%*******************************m-Stirling numbers *******************************
\subsection{Auxiliary estimates under the induction hypothesis}

In this subsection, we collect several estimates that will be 
used under the assumption that formula (\ref{1.4}) has already 
been established for the corresponding value of $a$. 
 
We first recall the $m$-th Stirling numbers of the first kind. 
These numbers may be regarded as shifted analogues of 
the ordinary Stirling numbers of the first kind and will be 
used to compare Stirling numbers with shifted arguments 
in the inductive step. They are defined by (see \cite{[QH]})
\begin{align*}
(x+m)(x+m+1)\cdots(x+m+n-1) = \sum_{k=0}^{n} s_m(n, k)x^k
\end{align*}
with $n\ge 1$, $m$ and $k$ being nonnegative integers.

\begin{lem}[\cite{[QH]}]\label{lem2.4}
Let $n\ge 1$, and let $m$ and $k$ be nonnegative integers. Then
\[  
s(m+n, k)=\sum_{i=0}^{k}s(m,i)s_m(n, k-i)\quad\text{and}\quad
s_m(n, k)=\sum_{i=k}^{n}s(n,i)\binom{i}{i-k}m^{i-k}.
\]
\end{lem}

These identities allow us to compare $s_m(n, k)$ with $s(n, k)$.  
For instance, Lemma \ref{lem2.4} gives $s_m(n, k)\equiv s(n, k) \pmod m$ for all $m\ge 1$. 
In the proof of Theorem \ref{thm1}, we need a sharper comparison for  
the $3^n$-th Stirling numbers. The next lemma provides such estimates. 
It is stated in a conditional form and will be used only 
under the relevant induction hypotheses in Sections 3 and 4. 

%*******************************Lemma 2.6*******************************

\begin{lem}\label{lem2.6}
Let $a\in\{1,2\}$, and let $n, t\in \mathbb{Z}^+$ with $t\le a3^n$. 
Assume that \eqref{1.4} holds for this value of $a$. If $2\mid (t-a)$, then
\begin{align}\label{2.10}
v_3(s_{3^n}(a3^n,t))=v_3(s(a3^n,t))\quad 
\text{and}\quad v_3(s_{3^n}(a3^n,t)-s(a3^n,t))\ge v_3(s(a3^n,t))+2.
\end{align}
If $2\nmid (t-a)$, then 
\begin{align}
v_3(s_{3^n}(a3^n,t))\ge v_3(s(a3^n,t+1))+n. \label{2.12}
\end{align}
\end{lem}

\begin{proof}
The proof is based on the expansion in Lemma \ref{lem2.4}. 
We first treat the boundary values $t=a3^n$ and $t=a3^n-1$. 
Then, for $1\le t\le a3^n-2$, we write $t=a3^m-k$ with $(m,k)\in T_{a,n}$, 
and distinguish the two cases according to the parity of $t-a$.  

If $t=a3^n$, then $2\mid (a3^n-a)$ and
$s_{3^n}(a3^n,a3^n)=1=s(a3^n,a3^n)$. Hence (\ref{2.10}) holds. 

If $t=a3^n-1$, then $2\nmid (t-a)$ and, by Lemma \ref{lem2.4} we have
\begin{align*}
s_{3^n}(a3^n,a3^n-1)=s(a3^n,a3^n-1)+s(a3^n,a3^n)3^n\binom{a3^n}{1}=s(a3^n,a3^n-1)+a3^{2n}.
\end{align*}
Thus $v_3(s_{3^n}(a3^n,a3^n-1))=v_3(s(a3^n,a3^n-1))=n=v_3(s(a3^n,a3^n))+n$,
so (\ref{2.12}) holds when $t=a3^n-1$.

It remains to consider the case $1\le t\le a3^n-2$. Then write
$t=a3^m-k$ with $(m, k)\in T_{a, n}$. By Lemma \ref{lem2.4}, 
\begin{align}\label{2.13}
s_{3^n}(a3^n, t)=s(a3^n, t)+\sum_{i=t+1}^{a3^n}\binom{i}{i-t}L_i
\end{align}
with $L_i:=s(a3^n,i)3^{n(i-t)}$. 
We now distinguish two cases according to the parity of $t-a$. 

\smallskip
\noindent\emph{Case 1: $2\mid (t-a)$.}

In this case, we prove (\ref{2.10}). Since $2\mid (t-a)$, 
we have $2\mid k$, and hence $2\le k\le 2a3^{m-1}$. By (\ref{1.4}), 
\begin{align}\label{2.14}
v_3(s(a3^n,a3^m-k))
=\frac{a}{2}(3^n-3^m)-(n-m)(a3^m-k)+m-1-v_3(k).
\end{align}

By (\ref{2.13}), it is enough to show that the summation term in 
(\ref{2.13}) has 3-adic valuation at least $v_3(s(a3^n, t))+2$. 
Equivalently, it suffices to prove that 
\begin{align}\label{2.15}
v_3(L_i)\ge v_3(s(a3^n,t))+2
\end{align}
for every $i$ with $t+1\le i\le a3^n$.  

We now prove (\ref{2.15}). First, let $i\in \{ a3^n-1, a3^n\}$. Then 
\[
v_3(L_i)=n(a3^n-t)=n(a3^n-a3^m+k).
\]
Therefore, by (\ref{2.14}), $1\le m\le n$ and $k\ge 2$, we obtain 
\begin{align*}
v_3(L_{i})-v_3(s(a3^n, t))
&=n(a3^n-a3^m+k)-v_3(s(a3^n, a3^m-k))\\
&=\Big(n-\frac{1}{2}\Big)a3^n-\Big(m-\frac{1}{2}\Big)a3^m+(k-1)m+1+v_3(k)\ge m+1\ge 2.
\end{align*}
Hence (\ref{2.15}) holds for $i\in \{ a3^n-1, a3^n\}$. 

Next, let $t+1\le i\le a3^n-2$. Write $i=a3^l-j$ with $(l, j)\in T_{a, n}$.
Then $l\ge m$ and $2\le j\le k-1$ if $l=m$.
By (\ref{1.4}) and (\ref{2.14}), we get 
\begin{align*}
v_3(L_{i})-v_3(s(a3^n, t))
&=v_3(s(a3^n, a3^l-j))+n(a3^l-j-a3^m+k)-v_3(s(a3^n, a3^m-k))\\
&=\frac{a}{2}(3^n-3^l)-(n-l)(a3^l-j)+l-1
-v_3\Big(\Big\lfloor\frac{j}{2}\Big\rfloor\Big)+(l+v_3(j))\epsilon_j\\
&\ \ +n(a3^l-j-a3^m+k)-\frac{a}{2}(3^n-3^m)+(n-m)(a3^m-k)-m+1+v_3(k)\\
&=\Big(l-\frac{1}{2}\Big)a3^l-\Big(m-\frac{1}{2}\Big)a3^m+l-lj+mk-m+
v_3(k)-v_3\Big(\Big\lfloor\frac{j}{2}\Big\rfloor\Big)+(l+v_3(j))\epsilon_j
=:D_j.
\end{align*}
It remains to prove $D_j\ge 2$ for $2\le j\le 2a3^{l-1}+1$. 

If $j$ is even, then $2\le j\le 2a3^{l-1}$, and hence $v_3(j)\le l-1$. 
In this case, 
\[
D_j=\Big(l-\frac{1}{2}\Big)a3^l-\Big(m-\frac{1}{2}\Big)a3^m-l(j-1)+m(k-1)+v_3(k)-v_3(j).
\]
If $l=m$, then $k-j\ge 2$ since $k-j\ge 1$ and both $k$ and $j$ are even.
Thus
\begin{align*}
D_j=m(k-j)+v_3(k)-v_3(j)\ge 2m-(m-1)=m+1\ge 2.
\end{align*}
If $l\ge m+1$, then 
\begin{align*}
D_j&\ge \Big(l-\frac{1}{2}\Big)a3^l-\Big(m-\frac{1}{2}\Big)a3^m-l(2a3^{l-1}-1)+m-(l-1)\notag\\
&=\Big(l-\frac{3}{2}\Big)a3^{l-1}-\Big(m-\frac{1}{2}\Big)a3^m+m+1\ge m+1\ge 2.
\end{align*}
Therefore, $D_j\ge 2$ holds when $j$ is even. 

If $j$ is odd, then $3\le j\le 2a3^{l-1}+1$ and $j<a3^l$. Note that $j-1$ is even and 
$2\le j-1\le 2a3^{l-1}$. Using the even case just proved, we obtain 
\begin{align*}
D_j
=\Big(l-\frac{1}{2}\Big)a3^l-\Big(m-\frac{1}{2}\Big)a3^m-l(j-1)+m(k-1)+v_3(k)-v_3(j-1)+l+v_3(j)
&=D_{j-1}+v_3(j)\ge 2.
\end{align*}
Thus $D_j\ge 2$ also holds when $j$ is odd.  

Consequently, (\ref{2.15}) holds for all $i$ with $t+1\le i\le a3^n$.  
Hence the summation term in (\ref{2.13}) has 3-adic valuation 
at least $v_3(s(a3^n, t))+2$. This proves (\ref{2.10}).

\smallskip
\noindent\emph{Case 2: $2\nmid (t-a)$.}

In this case, we have $2\nmid k$, $3\le k\le 2a3^{m-1}+1$ and $k<a3^m$.
Since $2\nmid (t-a)$, we have $2\mid (a3^n-t-1)$. Thus
$t=a3^n-2\cdot \frac{a3^n-t-1}{2}-1$. By Lemma \ref{lem2.2},
\[
v_3(s(a3^n, t))=v_3(s(a3^n, t+1))+v_3(a3^n-t)+n.
\]
In particular, 
\begin{align}\label{2.16}
v_3(s(a3^n, t))\ge v_3(s(a3^n, t+1)) +n.
\end{align}

From (\ref{2.13}), we write 
\begin{align}\label{2.17}
s_{3^n}(a3^n, t)
=s(a3^n, t)+(t+1)3^ns(a3^n, t+1)+\sum_{i=t+2}^{a3^n}\binom{i}{i-t}L_i.
\end{align} 
Since $2\mid (t+1-a)$, the estimate proved in Case 1, 
applied to $t+1$, gives 
\begin{align}\label{2.18}
v_3(L_i)=v_3(s(a3^n,i)3^{n(i-t)})
=v_3(s(a3^n, i)3^{n(i-(t+1))})+n\ge v_3(s(a3^n, t+1))+n+2
\end{align} 
for every $t+2\le i\le a3^n$. Therefore,  
(\ref{2.12}) follows from (\ref{2.16}) to (\ref{2.18}).

The proof of Lemma \ref{lem2.6} is complete.
\end{proof}

In Sections 3 and 4, (\ref{2.10}) will be used to replace the 
$3^n$-th Stirling numbers by the corresponding ordinary Stirling 
numbers up to higher-order error terms, while (\ref{2.12}) will 
be used to control the opposite parity cases. 
We shall also use the following adjacent-index estimate under 
the same inductive assumption. 

%*******************************Lemma 2.7*******************************

\begin{lem}\label{lem2.7}
Let $a\in\{1,2\}$ be fixed, and assume that (\ref{1.4}) holds 
for this value of $a$. Let $n, t\in \mathbb{Z}^+$ with $3\le t\le a3^n$.
If $2\mid (t-a)$, then 
\begin{align}\label{2.19} 
v_3(s(a3^n, t)) - v_3(s(a3^n, t-2))\ge -2n+2.
\end{align}
\end{lem}

\begin{proof}
We first consider the boundary case $t=a3^n$. 
By (\ref{1.4}) with $m=n$ and $k=2$, we have 
$v_3(s(a3^n, a3^n-2)) = n-1$. Since $v_3(s(a3^n, a3^n)) = 0$, 
it follows that 
\[
v_3(s(a3^n, t))-v_3(s(a3^n, t-2))=
v_3(s(a3^n, a3^n))- v_3(s(a3^n, a3^n-2)) = -n+1\ge -2n+2.
\] 
Thus (\ref{2.19}) holds in this case.

Now assume that $3\le t\le a3^n-2$. Write $t = a3^l-j$, where
$1\le l\le n$ and $2\le j\le 2a3^{l-1}$. Since $2\mid (t-a)$, 
we have $2\mid j$. Then, by the assumed formula (\ref{1.4}), we obtain  
\begin{align*}
v_3(s(a3^n, t))=v_3(s(a3^n, a3^l-j))
=\frac{a}{2}(3^n-3^l)-(n-l)(a3^l-j)+l-1-v_3(j).
\end{align*}
Since $t-2=a3^l-j-2=a3^l-(j+2)$, we distinguish two cases. 

If $j+2\le 2a3^{l-1}$, then
\begin{align*}
v_3(s(a3^n, t-2))
=v_3(s(a3^n, a3^l-(j+2)))
&=\frac{a}{2}(3^n-3^l)-(n-l)(a3^l-(j+2))+l-1-v_3(j+2).
\end{align*}
Hence
\begin{align*}
v_3(s(a3^n, t))-v_3(s(a3^n, t-2))&=-2(n-l)+v_3(j+2)-v_3(j)
\ge -2n+2l-(l-1)\ge-2n+l+1\ge -2n+2.
\end{align*}
Therefore, (\ref{2.19}) holds in this case.

If $j+2\ge 2a3^{l-1}+2$, then $j=2a3^{l-1}$, and so 
$t-2=a3^{l-1}-2$. Since $t-2\ge 1$, we have $l\ge 2$. Thus 
\begin{align*}
v_3(s(a3^n, t-2))&=v_3(s(a3^n, a3^{l-1}-2))
=\frac{a}{2}(3^n-3^{l-1})-(n-l+1)(a3^{l-1}-2)+l-2
\end{align*}
and
\[
v_3(s(a3^n, t)) = v_3(s(a3^n, a3^l-2a3^{l-1}))=\frac{a}{2}(3^n-3^l)-(n-l)a3^{l-1}.
\]
Therefore
\begin{align*}
v_3(s(a3^n, t))-v_3(s(a3^n, t-2))=-2-2(n-l)-l+2=-2n+l\ge-2n+2.
\end{align*}
Thus (\ref{2.19}) holds in this case. This completes the proof of Lemma \ref{lem2.7}. 
\end{proof}

%*******************************Proof of Lemma*******************************

\section{An Inductive Step: From $v_3(s(3^n, 3^m-k))$ to $v_3(s(2\cdot3^n, 2\cdot3^m-k))$}

In this section, we establish the first key inductive step in the proof of Theorem \ref{thm1}. 
More precisely, we show that formula (\ref{1.4}) for $a=1$ 
implies the corresponding formula for $a=2$.  
For convenience, throughout this paper, define 
$$s^{(1)}_{n,m,k}:=s(3^n,3^m-k), \ \ s^{(1)*}_{n,m,k}:=s_{3^n}(3^n,3^m-k),$$
$$s^{(2)}_{n,m,k}:=s(2\cdot3^n, 2\cdot3^m-k),\ \ \  s^{(2)*}_{n,m,k}:=s_{3^n}(2\cdot3^n, 2\cdot3^m-k).$$

We first outline the proof strategy. The aim of this section is to 
derive the formula for $s^{(2)}_{n,m,k}$ from the corresponding 
formula $s^{(1)}_{n,m,k}$. By Lemma \ref{lem2.3}, it is enough to treat the 
case where $k$ is even. For such $k$, we use the convolution formula 
from Lemma \ref{lem2.4} to decompose $s^{(2)}_{n,m,k}$ into 
three parts, $s^{(2)}_{n,m,k}=s_1+s_2+s_3$. The set defining $s_1$ contains 
the terms expected to give the exact valuation, 
while $s_2$ and $s_3$ are remainder sums.

For each fixed pair $(m, k)$, we denote by $W$ the value predicted by formula (\ref{1.4}); 
this is made explicit in (\ref{3.1}) below. The proof is then reduced to showing 
$v_3(s_1)=W$, $v_3(s_2)\ge W+1$ and $v_3(s_3)\ge W+2$. The principal part 
$s_1$ is handled by replacing the $3^n$-Stirling numbers $s_{3^n}(3^n, \cdot)$ 
with the corresponding ordinary Stirling numbers $s(3^n, \cdot)$. 
By Lemma \ref{lem2.6}, the replacement error has 3-adic valuation at least 
two larger than the predicted leading value, and hence it cannot 
affect the final 3-adic valuation. The sums $s_2$ and $s_3$ are then 
controlled by separating the relevant indices according to parity 
and applying the auxiliary estimates from Section 2. 

We prove the following lemma.

%*******************************Lemma 3.1*******************************

\begin{lem}\label{lem3.1}
If (\ref{1.4}) holds for $a=1$, then (\ref{1.4}) also holds for $a=2$.
\end{lem}

\begin{proof}
Let $n$ be a positive integer.
Suppose that (\ref{1.4}) holds for $a=1$.
Then Lemmas \ref{lem2.2} and \ref{lem2.6} are 
applicable with $a=1$. By Lemma \ref{lem2.3},
it suffices to prove (\ref{1.4}) holds
for all the integers $m$ and even $k$ with $1\le m\le n$ and
$2\le k\le 4\cdot3^{m-1}$.

% \subsection{The convolution decomposition} 

Let $m, k\in \mathbb{Z}$ with $1\le m\le n$,
$2\le k\le 4\cdot3^{m-1}$ and $2\mid k$.
Since $k$ is even, proving (\ref{1.4}) for $a=2$ is 
equivalent to proving the following identity:
\begin{align}\label{3.1}
v_3(s^{(2)}_{n,m,k})
=3^n-3^m-(n-m)(2\cdot3^m-k)+m-1-v_3(k)=:W.
\end{align}
Setting $(m,n,k)\mapsto(3^{n}, 3^{n}, 2\cdot3^m-k)$ in Lemma \ref{lem2.4}
gives us that
\begin{align*}
s^{(2)}_{n,m,k}
=\sum_{i=1}^{2\cdot3^m-k}s(3^n,i)s_{3^n}(3^n,2\cdot3^m-k-i).
\end{align*}
Now let $g_{1,k}(i):=s(3^n,i)s_{3^n}(3^n,2\cdot3^m-k-i)$. Define
$$
s_1:=\sum_{i\in A}g_{1,k}(i),\quad s_2:=\sum_{i\in B}g_{1,k}(i),\quad s_3:=\sum_{i\in C}g_{1,k}(i),
$$
where
\begin{align*}
A:=\left\{\begin{array}
{ll}\{3^m,3^m-k\}, & {\rm if}\ 2 \leq k \leq 2\cdot3^{m-1},\\
\{3^{m-1},2\cdot3^{m}-k-3^{m-1}\}, &
{\rm if}\ m\ge 2\ {\rm and}\ 2\cdot3^{m-1}+2 \leq k \leq 4\cdot3^{m-1}-2,\\
\{3^{m-1}\}, & {\rm if}\ k=4\cdot3^{m-1}
\end{array}\right.
\end{align*}
and
$$
B:=([1,2\cdot3^m-k]\cap (1+2\mathbb{Z}))\setminus A,
\quad C:=[1,2\cdot3^m-k]\cap (2\mathbb{Z}).
$$
Then $$s^{(2)}_{n,m,k}=s_1+s_2+s_3.$$ 
It remains to prove the following three estimates:
$$v_3(s_1)=W, \qquad  v_3(s_2)\ge W+1, \qquad v_3(s_3)\ge W+2.$$ 

%\subsection{The principal contribution $s_1$}
\medskip
\noindent\textbf{Estimate for the principal part $s_1$.}

We prove that $v_3(s_1)=W$. The proof is divided according 
to the three possible forms of $A$. In each case, we replace the terms 
involving $s_{3^n}(3^n, \cdot)$ by the corresponding terms 
involving $s(3^n, \cdot)$ plus error terms. More precisely, 
we write $s_1 = M + E$, where $M$ is the principal 
contribution obtained from the ordinary Stirling numbers $s(3^n, \cdot)$, 
and $E$ is the total error term arising from this replacement.  By Lemma \ref{lem2.6}, 
the error term always satisfies $v_3(E)\ge W+2$. 
Therefore, it remains only to verify that $v_3(M)=W$. 
Then $v_3(s_1)=W$ follows in each case. 

\smallskip
\noindent\emph{Case 1: $2\le k\le 2\cdot3^{m-1}$.} 

In this case, one has $A=\{3^m,3^m-k\}$. Hence
\begin{align*}
s_1=g_{1,k}(3^m)+g_{1,k}(3^m-k)
=s^{(1)}_{n, m, 0}\cdot s^{(1)*}_{n,m,k}+s^{(1)}_{n, m, k}\cdot s^{(1)*}_{n,m,0}.
\end{align*}
By Lemma \ref{lem2.4}, we may
write $s^{(1)*}_{n,m,k}=s^{(1)}_{n, m, k}+L_1$ and
$s^{(1)*}_{n,m,0}=s^{(1)}_{n, m, 0}+L_2$ with $L_1, L_2\in \mathbb{Z}$.
Since Lemma \ref{lem2.6} holds for $a=1$ and $k$ is even,
we have $v_3(L_1)\ge v_3(s^{(1)}_{n, m, k})+2$
and $v_3(L_2)\ge v_3(s^{(1)}_{n, m, 0})+2$.

Thus $s_1 = M + E$, where 
\[
M:=2s^{(1)}_{n, m, 0}\cdot s^{(1)}_{n, m, k},\quad  
E:=s^{(1)}_{n, m, 0}L_1+s^{(1)}_{n, m, k}L_2.
\]
It remains to compute $v_3(M)$. 
Note that $v_3(s^{(1)}_{n, m, 0})=\frac{1}{2}(3^n-3^{m})-(n-m)3^{m}$. 
Indeed, this is clear when $m=n$, and for $1\le m\le n-1$ it 
follows from $v_3(s^{(1)}_{n, m, 0})=v_3(s^{(1)}_{n, m+1, 2\cdot3^m})
=\frac{1}{2}(3^n-3^{m})-(n-m)3^{m}$. 
Therefore, by (\ref{1.4}) for $a=1$, we obtain 
\begin{align*}
v_3(M)&=v_3(2s^{(1)}_{n, m, 0}\cdot s^{(1)}_{n, m, k})
=v_3(s^{(1)}_{n, m, 0})+v_3(s^{(1)}_{n, m, k})
=v_3(g_{1,k}(3^m))=v_3(g_{1,k}(3^m-k))\\
&=\frac{1}{2}(3^n-3^{m})-(n-m)3^{m}+\frac{1}{2}(3^n-3^m)-(n-m)(3^m-k)+m-1-v_3(k)\\
&=3^n-3^m-(n-m)(2\cdot3^m-k)+m-1-v_3(k)=W.
\end{align*}

Since $s_1 = M+E$, $v_3(M)=W$ and $v_3(E)\ge W+2$, we have $v_3(s_1)=W$.

\smallskip
\noindent\emph{Case 2: $m\ge 2$ and $2+2\cdot3^{m-1}\le k\le 4\cdot3^{m-1}-2$.} 

In this case, $A=\{3^{m-1},2\cdot3^m-k-3^{m-1}\}$. Note that
$2\cdot3^m-k-3^{m-1}=3^m-(k-2\cdot3^{m-1})$ where 
$2\le k-2\cdot3^{m-1}\le 2\cdot3^{m-1}-2$, which implies that
$v_3(k-2\cdot3^{m-1})=v_3(k)$. Hence
\begin{align*}
s_1
=g_{1,k}(3^{m-1})+g_{1,k}(2\cdot3^m-k-3^{m-1})
&=s^{(1)}_{n, m-1, 0}\cdot s^{(1)*}_{n, m, k-2\cdot3^{m-1}}
+s^{(1)}_{n, m, k-2\cdot3^{m-1}}\cdot s^{(1)*}_{n, m-1, 0}.
\end{align*}

Again, write
$s^{(1)*}_{n, m, k-2\cdot3^{m-1}}=s^{(1)}_{n, m, k-2\cdot3^{m-1}}+L_3$
and $s^{(1)*}_{n, m-1, 0}=s^{(1)}_{n, m-1, 0}+L_4$ with $L_3, L_4\in \mathbb{Z}$.
Then $v_3(L_3)\ge v_3(s^{(1)}_{n, m, k-2\cdot3^{m-1}})+2$
and $v_3(L_4)\ge v_3(s^{(1)}_{n, m-1, 0})+2$. Thus 
$s_1 = M + E$, where 
$$M:=2s^{(1)}_{n, m-1, 0}\cdot s^{(1)}_{n, m, k-2\cdot3^{m-1}},\ \  
E:=s^{(1)}_{n, m-1, 0}L_3+s^{(1)}_{n, m, k-2\cdot3^{m-1}}L_4.$$ 

By (\ref{1.4}) for $a=1$, we have 
\begin{align*}
v_3(M)& =v_3(2s^{(1)}_{n, m-1, 0}\cdot s^{(1)}_{n, m, k-2\cdot3^{m-1}})
=v_3(s^{(1)}_{n, m, 2\cdot3^{m-1}})+v_3(s^{(1)}_{n, m, k-2\cdot3^{m-1}})\\
&=v_3(g_{1,k}(3^{m-1}))=v_3(g_{1,k}(2\cdot3^m-k-3^{m-1}))\\
&=\frac{1}{2}(3^n-3^{m})-(n-m)(3^{m}-2\cdot3^{m-1})\\
&\ \ \ \ +\frac{1}{2}(3^n-3^m)-(n-m)(3^m-(k-2\cdot3^{m-1}))
+m-1-v_3(k-2\cdot3^{m-1})\\
&=3^n-3^m-(n-m)(2\cdot3^m-k)+m-1-v_3(k)=W.
\end{align*}
Since $s_1 = M+E$, $v_3(M)=W$ and $v_3(E)\ge W+2$, we have $v_3(s_1)=W$.  

\smallskip
\noindent\emph{Case 3: $k=4\cdot3^{m-1}$.} 

In this case, $A=\{3^{m-1}\}$ and
\begin{align*}
s_1=g_{1, k}(3^{m-1})=s^{(1)}_{n, m-1, 0}\cdot s^{(1)*}_{n, m-1, 0}
=(s^{(1)}_{n, m-1, 0})^2+s^{(1)}_{n, m-1, 0}L_4
\end{align*}
with $v_3(L_4)\ge v_3(s^{(1)}_{n, m-1, 0})+2$.
Therefore, we have $s_1 = M + E$, where 
$$M:=(s^{(1)}_{n, m-1, 0})^2,\ \  
E:=s^{(1)}_{n, m-1, 0}L_4$$ and   
\begin{align*}
v_3(M)&=v_3(g_{1, k}(3^{m-1}))
=v_3((s^{(1)}_{n, m-1, 0})^2)
=2v_3(s^{(1)}_{n, m, 2\cdot3^{m-1}})
=3^n-3^m-2(n-m)(3^m-2\cdot3^{m-1})\\
&=3^n-3^m-(n-m)(2\cdot3^m-4\cdot3^{m-1})+m-1-v_3(4\cdot3^{m-1})=W. 
\end{align*}
Since $s_1 = M+E$, $v_3(M)=W$ and $v_3(E)\ge W+2$, we have $v_3(s_1)=W$. 

The three cases above exhaust all possible forms of $A$. Therefore, $v_3(s_1)=W$. 
This proves the estimate for the principal part $s_1$. 

%\subsection{The odd-index remainder}

\medskip
\noindent\textbf{Estimate for the odd-index part $s_2$}

We next estimate the odd-index part $s_2$. 
The purpose of this part is twofold: first, to prove $v_3(s_2)\ge W+1$; 
and second, to identify the possible terms in $s_2$ whose 3-adic valuation is exactly $W+1$. 

If $m=1$, then $B=\varnothing$, hence $s_2=0$. Thus we may 
assume that $2\le m\le n$. For $i\in B$, if either $i>3^n$ or $2\cdot3^m-k-i>3^n$, 
then $g_{1,k}(i)=0$. Therefore, it suffices to consider $B'=B\cap[2\cdot3^m-k-3^n,3^n]$. 
If $B'=\varnothing$, then $s_2=0$, and there is nothing to prove.

Assume henceforth that $B' \ne \varnothing$. We decompose $B'$ as 
$B' := B_{1}\cup B_{2}\cup B_{3}$, where
\begin{align*}
B_{1}:=\Big[1,3^m-\frac{k}{2}-1\Big]\cap B',\quad
B_{2}:=\Big[3^m-\frac{k}{2}+1,2\cdot3^m-k\Big]\cap B',\quad
B_{3}:=\Big\{3^m-\frac{k}{2}\Big\}\cap B'.
\end{align*}
Since at least one of $B_{1},B_{2}$ and $B_{3}$ is nonempty,
one obtains that
$$
s_2=\sum_{i\in B'}g_{1,k}(i)=\Big(\sum_{i\in B_{1}}
+\sum_{i\in B_{2}}+\sum_{i\in B_{3}}\Big)g_{1,k}(i).
$$

\smallskip
\noindent\emph{The central term $B_3$.}

We first estimate the possible central term $B_3$. 
If $B_{3}\ne \varnothing$, then let $i\in B_{3}$. So $i=3^m-\frac{k}{2}$.
Since $i$ is odd, one has $4\mid k$. Thus
\begin{align*}
v_3(g_{1,k}(3^m-\frac{k}{2}))&=v_3(s^{(1)}_{n, m, \frac{k}{2}}\cdot s^{(1)*}_{n, m, \frac{k}{2}})
=v_3((s^{(1)}_{n, m, \frac{k}{2}})^2)\\
&=3^n-3^m-(n-m)(2\cdot3^m-k)+2(m-1-v_3(k))
=W+m-1-v_3(k).
\end{align*}

By $2\le k\le 4\cdot3^{m-1}$ and $4\mid k$, one has
$v_3(k)\le m-1$. Thus $v_3(g_{1,k}(3^m-\frac{k}{2}))\ge W$
with the equality if and only if $k=4\cdot3^{m-1}$.
If $k=4\cdot3^{m-1}$, then $i=3^{m-1}$, this
contradicts the definition of $B$ when $k=4\cdot3^{m-1}$.
Therefore $v_3(g_{1,k}(3^m-\frac{k}{2}))\ge W+1$.
Further, $v_3(g_{1,k}(3^m-\frac{k}{2}))= W+1$ holds
if and only if $k=4\cdot3^{m-2}$ or $k=8\cdot3^{m-2}$, i.e.,
\begin{align}
v_3(g_{1,4\cdot3^{m-2}}(3^m-2\cdot3^{m-2}))
&=v_3(s^{(1)}_{n, m, 2\cdot3^{m-2}}\cdot s^{(1)*}_{n, m, 2\cdot3^{m-2}})
=v_3((s^{(1)}_{n, m, 2\cdot3^{m-2}})^2)=W+1,\label{3.3}\\
v_3(g_{1,8\cdot3^{m-2}}(3^m-4\cdot3^{m-2}))
&=v_3(s^{(1)}_{n, m, 4\cdot3^{m-2}}\cdot s^{(1)*}_{n, m, 4\cdot3^{m-2}})
=v_3((s^{(1)}_{n, m, 4\cdot3^{m-2}})^2)=W+1.\label{3.4}
\end{align}

Thus every term indexed by $B_3$ has 3-adic valuation at least $W+1$. 
Moreover, equality can occur only in the case listed in (\ref{3.3}) and (\ref{3.4}). 
All other terms indexed by $B_3$ have valuation at least $W+2$.

\smallskip
\noindent\emph{The paired terms $B_1$ and $B_2$.}

We now consider the terms indexed by $B_1$ and $B_2$. The two 
parts are symmetric with respect to the involution $i \mapsto 2\cdot 3^m-k-i$. 
Indeed, $i\in B_{1}$ if and only if $2\cdot 3^m-k-i\in B_{2}$.
Moreover, by Lemma \ref{lem2.6}, the two corresponding terms 
have the same 3-adic valuation: 
\begin{align*}
v_3(g_{1,k}(i))&=v_3(s(3^n,i)s_{3^n}(3^n,2\cdot3^{m}-k-i))
=v_3(s(3^n,i)s(3^n,2\cdot3^{m}-k-i))\\
&=v_3(s(3^n,2\cdot3^{m}-k-i)s_{3^n}(3^n,i))
=v_3(g_{1,k}(2\cdot3^{m}-k-i)).
\end{align*}
Hence it is enough to prove $v_3(g_{1,k}(i))\ge W+1$
holds for $i\in B_{1}$. This will be done in what follows.

Let $i\in B_{1}$. Then $1\le i\le 3^m-\frac{k}{2}-1\le 3^m-2$.
Write $i=3^{l_1}-j_1$ with $(l_1, j_1)\in T_{1, m}$ and $2\mid j_1$.
Note that $\frac{k}{2}+1\le j_1\le 2\cdot3^{m-1}$ if $l_1=m$.
Since $1\le l_1\le m\le n$, by (\ref{1.4}) for $a=1$,
\begin{align}\label{3.5}
v_3(s^{(1)}_{n, l_1, j_1})
=\frac{1}{2}(3^n-3^{l_1})-(n-l_1)(3^{l_1}-j_1)+l_1-1-v_3(j_1).
\end{align}
Note that $2\cdot 3^m-k-i \in B_{2}$ and
$3^{m-1}+1\le 3^m-\frac{k}{2}+1\le 2\cdot 3^m-k-i
\le \min\{2\cdot3^m-k-1,3^n\} \le \min\{2\cdot3^m-3, 3^n\}$.

If $2\cdot 3^m-k-i = 3^n$, then we must have $m=n$ and
$i=2\cdot3^m-k-3^n=3^n-k=3^{m}-k$.
It implies that $2+2\cdot3^{m-1}\le k\le 3^{m}-1$
since $i\ge 1$ and $i\ne 3^{m}-k$ when $2\le k\le 2\cdot3^{m-1}$.
Hence $1\le i\le 3^{m-1}-2=3^{n-1}-2$, which implies that $l_1\le n-1$.
Note that $2\le j_1\le 2\cdot3^{l_1-1}\le 2\cdot3^{n-2}$ and
$i=3^{l_1}-j_1=3^n-k$ imply that $v_3(k)=v_3(j_1)$.
From (\ref{3.5}), $s^{(1)*}_{n, n, 0}=1$
and $W=n-1-v_3(k)$ when $m=n$, we deduce that
\begin{align}
v_3(g_{1,k}(i))-W
&=v_3(s(3^n,i))-W=v_3(s^{(1)}_{n, l_1, j_1})-W\notag\\
&=\frac{1}{2}(3^n-3^{l_1})-(n-l_1)3^{l_1}+(n-l_1)(j_1-1)
\ge \frac{1}{2}(3^n-3^{l_1})-(n-l_1)(3^{l_1}-1)\ge 1\notag
\end{align}
with the equality if and only if $l_1=n-1$ and $j_1=2$.
So for $k=2\cdot3^{n-1}+2$, one has
\begin{align}\label{3.6}
v_3(g_{1, k}(3^{n-1}-2))
=v_3(g_{1, k}(3^n))
=v_3(s^{(1)}_{n, n-1, 2}\cdot s^{(1)}_{n, n, 0})=W+1.
\end{align}
For any other integer $k$ in this case, we have $v_3(g_{1,k}(i))\ge W+2$.

If $2\cdot 3^m-k-i \le 3^n-2$, then write $2\cdot3^{m}-k-i=3^{l_2}-j_2$
with $(l_2, j_2)\in T_{1, n}$ and $2\mid j_2$.
Since $3^{m-1}+1\le 2\cdot 3^m-k-i\le \min\{2\cdot3^m-3, 3^n\}$,
we have $m\le l_2\le m+1$. Also note that $2\le j_2\le \frac{k}{2}-1$
if $l_2=m$ and $3^m+3\le j_2\le 2\cdot3^m$ if $l_2=m+1$.
By (\ref{1.4}) for $a=1$,
\begin{align*}
v_3(s^{(1)}_{n, l_2, j_2})
=\frac{1}{2}(3^n-3^{l_2})-(n-l_2)(3^{l_2}-j_2)+l_2-1-v_3(j_2).
\end{align*}
Hence together with (\ref{3.5}) one derives that
\begin{align*}
v_3(g_{1,k}(i))-W
&=v_3(s(3^n,i))+v_3(s(3^n,2\cdot3^{m}-k-i))-W
=v_3(s^{(1)}_{n, l_1, j_1})+v_3(s^{(1)}_{n, l_2, j_2})-W\\
&=3^m-\frac{1}{2}(3^{l_1}+3^{l_2})-(m-l_1)(3^{l_1}-j_1)-(m-l_2)(3^{l_2}-j_2)\\
&\ \ \ \ \ \ \ +l_1+l_2-m-1+v_3(k)-v_3(j_1)-v_3(j_2)=:D_{l_1,l_2}.
\end{align*}
Since $2\cdot3^{m}-k=3^{l_1}-j_1+3^{l_2}-j_2$ with $2\le k\le 4\cdot 3^{m-1}$,
$2\le j_1\le 2\cdot3^{l_1-1}$ and $2\le j_2\le 2\cdot3^{l_2-1}$, we get that
$v_3(k)-v_3(j_1)-v_3(j_2)\ge -\max\{v_3(j_1),v_3(j_2)\}$.
In what follows, we prove that $D_{l_1,l_2}\ge 1$. 
Since $m\le l_2 \le m+1$, there are two cases to consider. 

\smallskip
\noindent\emph{The case $l_2=m$.}

In this case, we have $2\le j_2\le \frac{k}{2}-1\le 2\cdot3^{m-1}-1$.
Since $2\mid j_2$, one derives that $2\le j_2\le 2\cdot3^{m-1}-2$ and
$v_3(j_2)\le m-2$. Note that
\begin{align*}
D_{l_1,m}=\frac{1}{2}(3^{m}-3^{l_1})-(m-l_1)(3^{l_1}-j_1)+l_1-1+v_3(k)-v_3(j_1)-v_3(j_2).
\end{align*}

If $1\le l_1\le m-1$, then $2\le j_1\le 2\cdot3^{m-2}$.
Thus $v_3(j_1)\le m-2$ and one has
\begin{align*}
D_{l_1,m}
\ge \frac{1}{2}(3^{m}-3^{l_1})-(m-l_1)(3^{l_1}-j_1)+l_1-1-(m-2)
&\ge \frac{1}{2}(3^{m}-3^{l_1})-(m-l_1)(3^{l_1}-1)+1\ge 2.
\end{align*}
So $v_3(g_{1,k}(i))\ge W+2$ is proved when $1\le l_1\le m-1$ and $l_2=m$.

If $l_1=m$, then $\frac{k}{2}+1\le j_1\le 2\cdot3^{m-1}$.
Note that $j_1+j_2=k$ and $2\le j_2\le j_1-2$.
This implies that $6\le k\le 4\cdot3^{m-1}-2$.
Since $i\ne 3^{m-1}$ when $m\ge 2$ and $2+2\cdot3^{m-1}\le k\le 4\cdot3^{m-1}-2$,
we have $4\le j_1\le 2\cdot3^{m-1}-2$. Thus $v_3(j_1)\le m-2$ and
$$
D_{m,m}=m-1+v_3(k)-v_3(j_1)-v_3(j_2)\ge m-1-(m-2)=1,
$$
where the equality holds if and only if
$j_1$ (or $j_2$) equals $2\cdot3^{m-2}$ or $4\cdot3^{m-2}$
and $v_3(j_2)=v_3(k)$ (or $v_3(j_1)=v_3(k)$).
Hence $v_3(g_{1,k}(i))\ge W+1$ is proved when $l_1=l_2=m$.
Moreover, one notes that if $m\ge 3$ and $2\cdot3^{m-2}+2\le k\le 4\cdot3^{m-2}-2$, then
\begin{align}\label{3.7}
v_3(g_{1,k}(3^{m}-2\cdot3^{m-2}))
=v_3(g_{1,k}(3^m-(k-2\cdot3^{m-2})))
=v_3(s^{(1)}_{n, m, 2\cdot3^{m-2}}\cdot s^{(1)}_{n, m, k-2\cdot3^{m-2}})=W+1.
\end{align}
If $m\ge 3$ and $4\cdot3^{m-2}+2\le k\le 8\cdot3^{m-2}-2$ with
$k\ne 2\cdot3^{m-1}$, then (\ref{3.7}) holds and
\begin{align}\label{3.8}
v_3(g_{1,k}(3^{m}-4\cdot3^{m-2}))
=v_3(g_{1,k}(3^m-(k-4\cdot3^{m-2})))
=v_3(s^{(1)}_{n, m, 4\cdot3^{m-2}}\cdot s^{(1)}_{n, m, k-4\cdot3^{m-2}})=W+1.
\end{align}
If $m\ge 3$ and $8\cdot3^{m-2}+2\le k\le 10\cdot3^{m-2}-2$, then
(\ref{3.8}) holds.

For any other integer $k$ in this case, we have $v_3(g_{1,k}(i))\ge W+2$.

\smallskip
\noindent\emph{The case $l_2=m+1$.}

In this case, we have $3^m+3\le j_2\le 2\cdot3^m$. Note that
\begin{align*}
&D_{l_1,m+1}=\frac{1}{2}(3^{m}-3^{l_1})
-(m-l_1)(3^{l_1}-j_1)+2\cdot3^m-j_2+l_1+v_3(k)-v_3(j_1)-v_3(j_2).
\end{align*}

If $l_1=m$, then by $2\cdot3^m-k-i=3^{m+1}-j_2\ge 3^m$,
one knows that $i=3^m-j_1\le 3^m-k$. Thus $k\le j_1\le 2\cdot3^{m-1}$.
By the definition of $B$, $i\ne 3^m-k$. Hence $j_2\le 2\cdot3^m-2$ and
\begin{align*}
D_{m,m+1}=2\cdot3^m-j_2+m+v_3(k)-v_3(j_1)-v_3(j_2)\ge2.
\end{align*}
So $v_3(g_{1,k}(i))\ge W+2$ is proved when $l_1=m$ and $l_2=m+1$.

If $1\le l_1\le m-1$, then
\begin{align*}
D_{l_1,m+1}
&\ge \frac{1}{2}(3^{m}-3^{l_1})-(m-l_1)(3^{l_1}-j_1)+2\cdot3^m-2\cdot3^m+l_1-m\\
&=\frac{1}{2}(3^m-3^{l_1})-(m-l_1)(3^{l_1}-j_1+1)
\ge \frac{1}{2}(3^m-3^{l_1})-(m-l_1)(3^{l_1}-1)\ge 1,
\end{align*}
where the equality holds if and only if
$j_2=2\cdot3^m$, $l_1=m-1$, $j_1=2$,
i.e., $i=3^{m}-k=3^{m-1}-2$.
Hence $v_3(g_{1,k}(i))\ge W+1$ is proved when $1\le l_1\le m-1$ and $l_2=m+1$.
Moreover, if $k=2\cdot3^{m-1}+2$, then
\begin{align}\label{3.9}
v_3(g_{1,k}(3^{m}-k))=v_3(g_{1,k}(3^m))
=v_3(s^{(1)}_{n, m-1, 2}\cdot s^{(1)}_{n, m, 0})=W+1.
\end{align}
For any other integer $k$ in this case, we have $v_3(g_{1,k}(i))\ge W+2$. 

Consequently, every term indexed by $B_1 \cup B_2$ has 3-adic valuation 
at least $W+1$. All possible terms with $v_3(g_{1,k}(i))= W+1$ 
are precisely those listed in (\ref{3.6}) to (\ref{3.9}). All other terms indexed by 
$B_1 \cup B_2$ have 3-adic valuation at least $W+2$.

Combining the estimates for $B_3$ and $B_1 \cup B_2$, we obtain 
$v_3(g_{1,k}(i))\ge W+1$ for all $ i \in B' $. 
Moreover, all possible terms with 3-adic valuation exactly $W+1$ 
are listed in (\ref{3.3}) to (\ref{3.9}). Since $g_{1,k}(i)=0$ for $i\in B\setminus B'$, 
it follows that $v_3(s_2)\ge W+1$. This proves the estimate for 
the odd-index part $s_2$, together with the description of the possible $W+1$-terms.  

%\subsection{The even-index remainder}

\medskip
\noindent\textbf{Estimate for the even-index part $s_3$.}

It remains to estimate the contribution from even indices. 
We prove that $v_3(s_3)\ge W+2.$

Let $i\in C$. Then $2\mid i$. If either $i\ge 3^n+1$ or $2\cdot3^m-k-i\ge 3^n+1$,
then $g_{1,k}(i)=0$. Hence it suffices to consider 
$C':=C\cap[2\cdot3^m-k-3^n,3^n]$. Equivalently, $C'=C\cap[2\cdot3^m-k-3^n+1,3^n-1]$, 
since $C$ consists of even integers. Note that $C'\ne \varnothing$, since either 
$2\cdot3^m-k\in C'$ or $3^n-1\in C'$.

For $i\in C'$, both $i$ and $2\cdot3^m-k-i$ are even. 
By Lemmas \ref{lem2.2} and \ref{lem2.6}, we have 
\begin{align}
v_3(s(3^n, i))=v_3(s(3^n, i+1))+v_3(3^n-i)+n \label{3.10}
\end{align}
and
\begin{align}
v_3(s_{3^n}(3^n,2\cdot3^m-k-i))\ge v_3(s(3^n,2\cdot3^m-k-i+1))+n. \label{3.11}
\end{align}
Moreover, since $i$ is even and $3\le i+1\le 3^n$, the induction hypothesis 
and Lemma \ref{lem2.7} imply 
\begin{align}
v_3(s(3^n,i+1))-v_3(s(3^n,i-1))\ge -2n+2. \label{3.12}
\end{align}
Then it follows from (\ref{3.10}) to (\ref{3.12}) that
\begin{align}
v_3(g_{1,k}(i))&=v_3(s(3^n,i))+v_3(s_{3^n}(3^n,2\cdot3^m-k-i))\notag\\
&\ge v_3(s(3^n,i+1))+ v_3(s(3^n,2\cdot3^m-k-i+1))+2n\notag\\
&\ge v_3(s(3^n,i-1))+v_3(s(3^n,2\cdot3^m-k-(i-1)))+2\notag\\
&=v_3(s(3^n,i-1))+v_3(s_{3^n}(3^n,2\cdot3^m-k-(i-1)))+2
=v_3(g_{1,k}(i-1))+2. \label{3.13}
\end{align}
Since $i-1$ is odd and $2\le i\le 2\cdot3^m-k$,
one has $i-1\in A\cup B$. By the estimates for $s_1$ and $s_2$, 
we know that $v_3(g_{1,k}(i-1))\ge W$. It follows from (\ref{3.13}) that 
$$v_3(g_{1,k}(i))\ge v_3(g_{1,k}(i-1))+2\ge W+2.$$

Therefore $v_3(g_{1,k}(i))\ge W+2$ 
for $i\in C'$, and consequently $v_3(s_3)\ge W+2$. 
This proves the estimate for the even-index part $s_3$.

Combining the estimates for $s_1$, $s_2$ and $s_3$, we obtain 
$v_3(s^{(2)}_{n,m,k})=W$. Therefore formula (\ref{1.4}) holds for 
$a=2$ whenever it holds for $a=1$. This completes the proof 
of Lemma \ref{lem3.1}. 
\end{proof}

\begin{rem}\label{rem3.2}
Let $n, m, k\in \mathbb{Z}$ such that
$1\le m\le n$ and $2\le k\le 4\cdot3^{m-1}$ with $2\mid k$.
If Theorem \ref{thm1} holds for $a=1$, then
from the proof of Lemma \ref{lem3.1}, we can conclude that
$$s^{(2)}_{n, m, k}=\bar{s}_1+\bar{s}_2+\bar{s}_3$$
with $\bar{s}_1, \bar{s}_2, \bar{s}_3\in \mathbb{Z}$,
$v_3(\bar{s}_1)=v_3(s^{(2)}_{n, m, k})=W$,
$v_3(\bar{s}_2)\ge W+1$ and $v_3(\bar{s}_3)\ge W+2$, where
\begin{align*}
&\bar{s}_1=
\left\{\begin{array}
{ll}2s^{(1)}_{n, m, 0}\cdot s^{(1)}_{n, m, k}, & {\rm if}\ 2 \leq k \leq 2\cdot3^{m-1},\\
2s^{(1)}_{n, m-1, 0}\cdot s^{(1)}_{n, m, k-2\cdot3^{m-1}}, &
{\rm if}\ k\in[2\cdot3^{m-1}+2, 4\cdot3^{m-1}-2]\ {\rm with}\ m\ge 2,\\
(s^{(1)}_{n, m-1, 0})^2, & {\rm if}\ k=4\cdot3^{m-1}.
\end{array}\right.
\end{align*}

If $m=2$ and $k=2\cdot3^{m-1}+2\cdot3^{m-2}$, then by (\ref{3.4}), (\ref{3.6}) and (\ref{3.9}), one has
\[
\bar{s}_2=2s^{(1)}_{n, m-1, 2}\cdot s^{(1)}_{n, m, 0}+(s^{(1)}_{n,m,4\cdot3^{m-2}})^2.
\]
If $m\ge 2$ and $k=4\cdot3^{m-2}$, then by (\ref{3.3}), 
\[
\bar{s}_2=(s^{(1)}_{n,m,2\cdot3^{m-2}})^2.
\]
Assume now that $m\ge 3$. If $k\in [2\cdot3^{m-2}+2, 4\cdot3^{m-2}-2]$, then by (\ref{3.7}),
\[
\bar{s}_2=2s^{(1)}_{n,m,2\cdot3^{m-2}}\cdot s^{(1)}_{n,m,k-2\cdot3^{m-2}}.
\]
If $k\in [4\cdot3^{m-2}+2, 8\cdot3^{m-2}-2]$
with $k\notin \{2\cdot3^{m-1}, 2\cdot3^{m-1}+2\}$, then by (\ref{3.7}) and (\ref{3.8}),
\[
\bar{s}_2=2s^{(1)}_{n,m,2\cdot3^{m-2}}\cdot s^{(1)}_{n,m,k-2\cdot3^{m-2}}
+2s^{(1)}_{n,m,4\cdot3^{m-2}}\cdot s^{(1)}_{n,m,k-4\cdot3^{m-2}}.
\]
If $k=2\cdot3^{m-1}+2$, then by (\ref{3.6}) to (\ref{3.9}),
\[
\bar{s}_2=2s^{(1)}_{n,m-1,2}\cdot s^{(1)}_{n,m,0}
+2s^{(1)}_{n,m,2\cdot3^{m-2}}\cdot s^{(1)}_{n,m,k-2\cdot3^{m-2}}
+2s^{(1)}_{n,m,4\cdot3^{m-2}}\cdot s^{(1)}_{n,m,k-4\cdot3^{m-2}}.
\]
If $k=8\cdot3^{m-2}$, then by (\ref{3.4}), 
\[
\bar{s}_2=(s^{(1)}_{n,m,4\cdot3^{m-2}})^2.
\]
If $k\in [8\cdot3^{m-2}+2, 10\cdot3^{m-2}-2]$,
then by (\ref{3.8}),
\[
\bar{s}_2=2s^{(1)}_{n,m,4\cdot3^{m-2}}\cdot s^{(1)}_{n,m,k-4\cdot3^{m-2}}.
\]
For any other $k$, one has $\bar{s}_2=0$.
\end{rem}

%*******************************Proof of Theorem 1*******************************

\section{From $2\cdot 3^n$ to $3^{n+1}$: proof of Theorem \ref{thm1}}

%Now we can begin to prove Theorem \ref{thm1}.\\

%{\it Proof of Theorem \ref{thm1}.}
In this section, we complete the proof of Theorem \ref{thm1}. 
By Lemma \ref{lem3.1}, it suffices to prove formula (\ref{1.4}) 
for $a=1$. We shall prove this by induction on $n$.

We briefly describe the structure of the induction step. 
Assuming that Theorem \ref{thm1} holds for $n$, 
Lemma \ref{lem3.1} gives the required formula for the 
family $s^{(2)}_{n, m, k}$. The remaining task is 
to derive the formula for $s^{(1)}_{n+1, m, k}$. We do 
this by applying the convolution formula of Lemma \ref{lem2.4}, 
which expresses $s^{(1)}_{n+1, m, k}$ as a sum of products 
involving $s(3^n, \cdot)$ and $s_{3^n}(2\cdot3^n, \cdot)$. 

After reducing to even $k$, we denote by $W_{m, k}$ the value 
predicted by the formula. The proof then proceeds by estimating 
the summands $g_{2,k}(i)= s(3^n,i)s_{3^n}(2\cdot3^n, 3^m-k-i)$. 
Even-index terms are shown to have 3-adic valuation at least $W_{m, k}+1$, 
and therefore cannot contribute to the leading valuation. 
Odd-index terms may have 3-adic valuation as low as $W_{m, k}-1$, 
so we isolate a special subset $A$ of odd indices containing all possible 
leading contributions. The key point is then to prove 
$v_3(\sum_{i\in A} g_{2,k}(i))=W_{m, k}.$ Once this is established, 
all remaining summands have strictly larger 3-adic valuation, and the desired 
value of $v_3(s^{(1)}_{n+1, m, k})$ follows.

\subsection{Induction setup}

We first set up the induction step. The case $n=1$ is immediate. 
Indeed, then $m=1$ and $k=2$, and 
$v_3(s(3,3-2))=v_3(s(3,1))=v_3(2)=0$. Thus Theorem \ref{thm1} holds 
for $n=1$. 

Assume now that Theorem \ref{thm1} holds for $n$.
We prove it for $n+1$. This is equivalent to showing that, for all integers $m$ and $k$
satisfying $1\le m\le n+1$, $2\le k\le 2\cdot3^{m-1}+1$ and $k<3^m$, one has
\begin{align}\label{4.1}
v_3(s^{(1)}_{n+1, m, k})
=\frac{1}{2}(3^{n+1}-3^m)-(n+1-m)(3^m-k)
+m-1-v_3\Big(\Big\lfloor\frac{k}{2}\Big\rfloor\Big)+(m+v_3(k))\epsilon_k,
\end{align}
where $\epsilon_k=0$ if $k$ is even and $\epsilon_k=1$ if $k$ is odd.

The case $m=1$ is also immediate, since then $k=2$ and
\[
v_3(s^{(1)}_{n+1, 1, 2})=v_3(s(3^{n+1},1))=v_3((3^{n+1}-1)!)
=\frac{1}{2}(3^{n+1}-2n-3).
\]
Hence we may assume that $2\le m \le n+1$. 
Moreover, by Lemma \ref{lem2.3}, it suffices to prove (\ref{4.1}) for even $k$.
Thus, throughout the rest of this section, we assume that 
 $2\le m\le n+1$, $2\le k\le 2\cdot3^{m-1}$ and $k$ is even. 
It remains to prove  
\begin{align}\label{4.2}
&v_3(s^{(1)}_{n+1, m, k})
=\frac{1}{2}(3^{n+1}-3^m)-(n+1-m)(3^m-k)+m-1-v_3(k)=:W_{m,k}.
\end{align}

\subsection{The convolution decomposition and estimates for $g_{2,k}(i)$}

In this subsection, we decompose $s^{(1)}_{n+1, m, k}$ into a sum 
of terms $g_{2,k}(i)$ and establish valuation estimates for these terms
according to the parity of $i$. By Lemma \ref{lem2.4}, with the substitution 
$(m,n,k)\mapsto(3^{n},2\cdot3^{n},3^m-k)$, we have 
\[
s^{(1)}_{n+1, m, k}=\sum_{i=1}^{3^m-k}s(3^n,i)s_{3^n}(2\cdot3^n,3^m-k-i).
\]
Let $i\in \mathbb{Z}$ and $1\le i\le 3^m-k$.
Define $g_{2,k}(i):= s(3^n,i)s_{3^n}(2\cdot3^n,3^m-k-i)$. 
Then 
\begin{align}\label{4.3}
s^{(1)}_{n+1, m, k}=\sum_{i=1}^{3^m-k}g_{2,k}(i).
\end{align}

The proof of (\ref{4.2}) is based on the following strategy. 
We first establish two estimates for the summands $g_{2,k}(i)$: 
\begin{equation}\label{4.4}
\begin{aligned}
v_3(g_{2,k}(i))&\ge W_{m, k}-1\quad  &&\textrm{if}\ 2\nmid i, \\
v_3(g_{2,k}(i))&\ge W_{m, k}+1\quad &&\textrm{if}\ 2\mid i.
\end{aligned}
\end{equation}
The second estimate in (\ref{4.4}) shows that even-index terms cannot affect the 
leading valuation. The first estimate in (\ref{4.4}) shows that possible 
low-valuation terms among odd indices can only have valuation $W_{m,k}-1$ or
$W_{m, k}$. Accordingly, define
\begin{align*}
A_{1}&:=\{i\in [1, 3^m-k ]\cap (1+2\mathbb{Z}): v_3(g_{2,k}(i))=W_{m, k}-1\},\\
A_{2}&:=\{i\in [1, 3^m-k ]\cap (1+2\mathbb{Z}): v_3(g_{2,k}(i))=W_{m, k}\}, 
\quad A:=A_1\cup A_2.
\end{align*}
We also define the two remainder sets 
\[
B:=\bigl([1,3^m-k]\cap(1+2\mathbb Z)\bigr)\setminus A,
\qquad
C:=[1,3^m-k]\cap 2\mathbb Z.
\]
Thus the index set $[1,3^m-k]$ is the disjoint union of $A$, $B$ and $C$. Set
\[
S_A:=\sum_{i\in A}g_{2,k}(i),\qquad
S_B:=\sum_{i\in B}g_{2,k}(i),\qquad
S_C:=\sum_{i\in C}g_{2,k}(i).
\]
Once \eqref{4.4} is proved, the definition of $A$ will imply that
\[
v_3(S_B)\ge W_{m,k}+1,\qquad v_3(S_C)\ge W_{m,k}+1.
\]
Thus the only remaining leading contribution can come from $S_A$. The key
estimate is
\begin{align}\label{4.6}
v_3(S_A)=W_{m,k}.
\end{align}
Therefore, to complete the induction step, it remains to prove
\eqref{4.4} and \eqref{4.6}.

\medskip
\noindent\textbf{Odd indices.}

Now we prove the first estimate in (\ref{4.4}). Let $i$ be an odd integer with $1\le i\le 3^m-k$. 
If $i > 3^n$ or $3^m-k-i > 2\cdot3^n$, then
$g_{2,k}(i)=0$, and there is nothing to prove. 
Hence we may assume that $3^m-k-2\cdot3^n\le i\le 3^n$.

We first consider the boundary case $i=3^m-k$. In this case, 
$3^m-k-i=0$, and therefore 
\begin{align}\label{4.7}
v_3(s_{3^n}(2\cdot 3^n, 3^m-k-i))=v_3(s_{3^n}(2\cdot 3^n, 0))
=v_3\Big(\frac{(3^{n+1}-1)!}{(3^n-1)!}\Big)=3^n-1.
\end{align}
If $i=3^m-k=3^n$, then by $2\le m\le n+1$ and $2\le k\le 2\cdot3^{m-1}$,
we get that $m=n+1$ and $k=2\cdot3^n$. Since $s(3^n,3^n)=1$ and
$W_{m,k}=W_{n+1,2\cdot3^n}=0$, one then deduces that
\begin{align*}
v_3(g_{2,k}(i))=v_3(g_{2, 2\cdot3^n}(3^n))=3^n-1\ge 2=W_{m,k}+2
\end{align*}
as required by (\ref{4.4}). If $i=3^m-k\le 3^n-2$, then $m\le n$.
By the inductive hypothesis and (\ref{4.7}), we obtain 
\begin{align*}
v_3(g_{2,k}(i))-W_{m,k}
&=v_3(s^{(1)}_{n, m, k})+v_3(s_{3^n}(2\cdot 3^n, 0))-W_{m,k}
=3^m-k-1\ge 3^{m-1}-1\ge 2.
\end{align*}
So the first estimate in (\ref{4.4}) holds when $i=3^m-k$. 

It remains to consider the case $1\le i\le 3^m-k-2$. 
Since $i$ is odd and $k$ is even, the integer $3^m-k-i$ is even,  
and $2\le 3^m-k-i < 2\cdot3^m-2$. 
Write 
$$i=3^{l_1}-j_1,\quad 3^m-k-i=2\cdot3^{l_2}-j_2,$$
where $(l_1, j_1)\in T_{1, m}$, $(l_2, j_2)\in T_{2, m}$
with $2\mid j_1$ and $2\mid j_2$. 
Since $2\le k\le 2\cdot3^{m-1}$, we have $3^{m-1}\le 3^m-k\le 3^m-2$. 
In particular, the case $l_1=l_2=m$ cannot occur. Indeed, 
if $l_1=l_2=m$, then 
$3^m-2\ge 3^m-k=i+3^m-k-i=3^m-j_1+2\cdot3^m-j_2\ge
3^m-2\cdot3^{m-1}+2\cdot3^m-4\cdot3^{m-1}=3^m$,
which is a contradiction. 
Since $2\mid (3^m-k-i)$ and $2\le 3^m-k-i\le 2\cdot 3^n$, 
Lemma \ref{lem2.6} also gives us that
\begin{align}
v_3(g_{2,k}(i))\label{4.8}
=v_3(s(3^n,i))+v_3(s_{3^n}(2\cdot3^n,3^m-k-i))
&=v_3(s(3^n,i))+v_3(s(2\cdot3^n,3^m-k-i)).
\end{align}

We first consider the boundary case $i=3^{l_1}-j_1=3^n$. 
Then $l_1=m=n+1$ and $j_1=2\cdot3^{l_1-1}=2\cdot3^n$ since $l_1\le m\le n+1$
and $2\le j_1\le 2\cdot3^{l_1-1}$. Hence $1\le l_2\le m-1=n$. Moreover, 
$2\cdot3^{l_2}-j_2=3^m-k-i=2\cdot3^{n}-k$. Since 
$2\le j_2\le 4\cdot3^{l_2-1}$, we have $2\le k\le 2\cdot3^n-2$, 
and consequently $v_3(j_2)=v_3(k)$. 
By the inductive hypothesis and (\ref{4.8}), we derive that
\begin{align*}
v_3(g_{2,k}(i))-W_{m,k}&=v_3(g_{2,k}(3^n))-W_{n+1,k}
=v_3(s^{(1)}_{n, n, 0})+v_3(s^{(2)}_{n, l_2, j_2})-W_{n+1,k}\notag\\
&=3^n-3^{l_2}-(n-l_2)(2\cdot3^{l_2}-j_2+1)-1=:D_{l_2}.
\end{align*} 
If $l_2=n$, then $2\le j_2=k\le 4\cdot3^{n-1}$ and $D_{n}=-1$.
If $n\ge 2$ and $1\le l_2\le n-1$, then
\begin{align*}
D_{l_2}&\ge 3^n-3^{l_2}-(n-l_2)(2\cdot3^{l_2}-1)-1
\ge 3^n-3^{n-1}-(2\cdot3^{n-1}-1)-1=0,
\end{align*}
where the equality holds if and only if $j_2=2$ and $l_2=n-1$.
Hence the first estimate in (\ref{4.4}) holds when $i=3^n$. 
Moreover, in this boundary case, the possible equality cases are as follows: 
\begin{align*}
&v_3(g_{2,k}(i))=W_{m,k}-1\ \ 
\text{if}\ \ 2\le k\le 4\cdot3^{n-1}=4\cdot3^{m-2}\ \  \text{and}\ \  i=3^n=3^{m-1};\\
&v_3(g_{2,k}(i))=W_{m,k}\ \ \text{if}\ \ 
k=4\cdot3^{n-1}+2=4\cdot3^{m-2}+2\ \  \text{and}\ \  i=3^n=3^{m-1},\  m\ge 3.
\end{align*}
In all other cases with $i=3^n=3^{m-1}$, one has $v_3(g_{2,k}(i))\ge W_{m,k}+1$.

We next consider the other boundary case $3^m-k-i=2\cdot3^{l_2}-j_2=2\cdot3^n$. 
Then $l_2=m=n+1$ and $j_2=4\cdot3^{l_2-1}=4\cdot3^n$. Hence 
$1\le l_1\le m-1=n$. Moreover, $3^{l_1}-j_1=i=3^n-k$. Thus 
$2\le k\le 3^n-1$, and consequently $v_3(j_1)=v_3(k)$. 
By the inductive hypothesis and (\ref{4.8}), we have
\begin{align*}
v_3(g_{2,k}(i))-W_{m,k}
&=v_3(g_{2,k}(3^n-k))-W_{n+1,k}
=v_3(s^{(1)}_{n, l_1, j_1})+v_3(s^{(2)}_{n, n, 0})-W_{n+1,k}\notag\\
&=\frac{1}{2}(3^n-3^{l_1})-(n-l_1)(3^{l_1}-j_1+1)-1=: D_{l_1}.
\end{align*}
If $l_1=n$, then $2\le j_1=k\le 2\cdot3^{n-1}$ and $D_{n}=-1$.
If $n\ge 2$ and $1\le l_1\le n-1$, then
\begin{align*}
D_{l_1}&\ge \frac{1}{2}(3^n-3^{l_1})-(n-l_1)(3^{l_1}-1)-1
\ge \frac{1}{2}(3^n-3^{n-1})-(3^{n-1}-1)-1=0,
\end{align*}
where the equality holds if and only if $j_1=2$ and $l_1=n-1$.
Hence (\ref{4.4}) holds when $3^m-k-i=2\cdot3^n$. 
Moreover, in this boundary case, the possible equality cases are as follows: 
\begin{align*}
&v_3(g_{2,k}(i))=W_{m,k}-1\ \ 
\text{if}\ \ 2\le k\le 2\cdot3^{n-1}=2\cdot3^{m-2}\ \  \text{and}\ \  i=3^n-k=3^{m-1}-k;\\
&v_3(g_{2,k}(i))=W_{m,k}\ \ \text{if}\ \ 
k=2\cdot3^{n-1}+2=2\cdot3^{m-2}+2\ \  \text{and}\ \  i=3^n-k=3^{m-1}-k,\  m\ge 3.
\end{align*}
In all other cases with $i=3^n-k=3^{m-1}-k$, one has $v_3(g_{2,k}(i))\ge W_{m,k}+1$.

It remains to consider the case $i\le 3^n-2$ and $3^m-k-i\le 2\cdot3^n-2$. 
Since $i=3^{l_1}-j_1$ and $3^m-k-i=2\cdot3^{l_2}-j_2$, we have $l_1\le n$ and $l_2\le n$.
By the induction assumption and (\ref{4.8}), we obtain 
\begin{align}
v_3(g_{2,k}(i))-W_{m,k}
&=v_3(s^{(1)}_{n, l_1, j_1})+v_3(s^{(2)}_{n, l_2, j_2})-W_{m,k}\notag\\
&=\frac{1}{2}(3^m-3^{l_1})-3^{l_2}+(3^{l_1}-j_1)(l_1+1-m)+(2\cdot3^{l_2}-j_2)(l_2+1-m)\notag\\
&\quad +l_1+l_2-m-1+v_3(k)-v_3(j_1)-v_3(j_2)=:D_{l_1,l_2}.\label{4.9'}
\end{align}
Therefore, it remains to prove that 
\begin{align} 
D_{l_1, l_2}\ge -1.\label{4.10'}
\end{align}
Note that $3^m-k=i+3^m-k-i=3^{l_1}-j_1+2\cdot3^{l_2}-j_2$
with $2\le k\le 2\cdot 3^{m-1}$, $2\le j_1\le 2\cdot 3^{l_1-1}$ and
$2\le j_2\le 4\cdot 3^{l_2-1}$ imply that
\[
v_3(k)-v_3(j_1)-v_3(j_2)\ge -\max\{v_3(j_1),v_3(j_2)\}.
\]
We prove (\ref{4.10'}) by considering the possible values of $l_1$. 
Since the case $l_1=l_2=m$ has already been excluded and $1\le l_1\le m$, 
it remains to consider the following three cases:
\[
l_1=m,\quad l_1=m-1, \quad 1\le l_1\le m-2.
\]

\smallskip
\noindent\emph{Case 1: $l_1=m$.}

In this case, we have $2\le j_1\le 2\cdot3^{m-1}$ and $1\le l_2\le m-1$.
Hence $v_3(j_1)\le m-1$ and $v_3(j_2)\le l_2-1\le m-2$.
Moreover, 
\begin{align*}
D_{m,l_2}&=3^m-j_1-3^{l_2}+(2\cdot3^{l_2}-j_2)(l_2+1-m)+l_2-1+v_3(k)-v_3(j_1)-v_3(j_2).
\end{align*}
We distinguish two subcases according to whether 
$j_1$ attains its maximal value $2\cdot3^{m-1}$.

\smallskip
\noindent\emph{Subcase 1.1: $j_1=2\cdot3^{m-1}$.}

Then $i=3^{m-1}$. We further distinguish the value of $l_2$. 

If $l_2=m-1$, then $k=j_2\le 4\cdot 3^{m-2}$. A direct calculation gives 
\begin{align*}
D_{m,m-1}&=m-1-1-(m-1)=-1.
\end{align*}

If $1\le l_2\le m-2$, then $m\ge 3$, $i=3^{m-1}$ and 
$3^m-k-i=2\cdot 3^{m-1}-k=2\cdot3^{l_2}-j_2$. Thus 
$4\cdot3^{m-2}+2\le k\le 2\cdot 3^{m-1}-2$, and consequently $v_3(k)=v_3(j_2)$.
It follows that 
\begin{align*}
D_{m,l_2}=3^{m-1}-3^{l_2}(2(m-l_2-1)+1)+(j_2-1)(m-l_2-1)-1
\ge j_2-2\ge 0. 
\end{align*}
Moreover, the equality holds if and only if $l_2=m-2$ and $j_2=2$.

\smallskip
\noindent\emph{Subcase 1.2: $2\le j_1\le 2\cdot3^{m-1}-2$.}

Then $v_3(j_1)\le m-2$, and hence
\begin{align*}
D_{m,l_2}&\ge 3^{m-1}+2-3^{l_2}+(2\cdot3^{l_2}-j_2)(l_2+1-m)+l_2-1-(m-2)\\
&=3^{m-1}-3^{l_2}(2(m-l_2-1)+1)+(j_2-1)(m-l_2-1)+2
\ge 2.
\end{align*}

Combining the two subcases, we have proved (\ref{4.10'}) in the case $l_1=m$.  
Moreover, the equality cases in Case 1 are as follows. 
\begin{align*}
&v_3(g_{2,k}(i))=W_{m,k}-1\ \ \text{if}\ \ 2\le k\le 4\cdot3^{m-2}\ \ \text{and}\ \  i=3^{m-1}; \\
&v_3(g_{2,k}(i))=W_{m,k}\ \ \text{if}\ \ k=4\cdot3^{m-2}+2\ \ \text{and}\ \  i=3^{m-1},\ m\ge 3.  
\end{align*}
In all remaining cases covered by Case 1, we have 
$v_3(g_{2,k}(i))\ge W_{m,k}+1$.

\smallskip
\noindent\emph{Case 2: $l_1=m-1$.}

In this case, we have $2\le j_1\le 2\cdot3^{m-2}$, and hence 
$v_3(j_1)\le m-2$. Moreover, 
\begin{align*}
D_{m-1,l_2}=3^{m-1}-3^{l_2}+(2\cdot3^{l_2}-j_2)(l_2+1-m)+l_2-2+v_3(k)-v_3(j_1)-v_3(j_2).
\end{align*}
We distinguish three subcases according to the value of $l_2$, where $1\le l_2\le m$. 

\smallskip
\noindent\emph{Subcase 2.1: $l_2=m$.}

Then $2\le j_2\le 4\cdot3^{m-1}$ and
\begin{align*}
D_{m-1,m}=4\cdot 3^{m-1}-j_2+m-2+v_3(k)-v_3(j_1)-v_3(j_2).
\end{align*}
If $j_2=4\cdot3^{m-1}$, then $i=3^{m-1}-k$ and so $k=j_1 \le 2\cdot3^{m-2}$.
Hence
\begin{align*}
D_{m-1,m}=m-2-(m-1)=-1.
\end{align*}
If $2\le j_2\le 4\cdot3^{m-1}-2$, then $v_3(j_2)\le m-1$ and
\begin{align*}
D_{m-1,m}\ge 4\cdot 3^{m-1}-(4\cdot 3^{m-1}-2)+m-2-(m-1)=1.
\end{align*}

Hence (\ref{4.10'}) holds in this subcase. 
The only exceptional case occurs when $j_2=4\cdot3^{m-1}$, 
therefore 
\[
v_3(g_{2,k}(i))=W_{m,k}-1\quad \text{if}\quad 2\le k\le 2\cdot3^{m-2}\quad \text{and}\quad  i=3^{m-1}-k. 
\]
In all remaining cases of Subcase 2.1, $D_{m-1,m}\ge 1$, and hence  
$v_3(g_{2,k}(i))\ge W_{m,k}+1$.

\smallskip
\noindent\emph{Subcase 2.2: $l_2=m-1$.}

Then $2\le j_2\le 4\cdot3^{m-2}$. 
Note that in this case $j_1+j_2=k$ and
\[
D_{m-1,m-1}=m-3+v_3(k)-v_3(j_1)-v_3(j_2).
\]

We first identify the cases in which $D_{m-1,m-1}=m-3-(m-2)=-1$. 
This happens only in the following situations:
\begin{align*}
&j_1=2\cdot 3^{m-2},\ \ \ \ v_3(k)=v_3(j_2)=v_3(k-2\cdot 3^{m-2});\\
&j_2=2\cdot 3^{m-2},\ \ \ \ v_3(k)=v_3(j_1)=v_3(k-2\cdot 3^{m-2});\\
&j_2=4\cdot 3^{m-2},\ \ \ \ v_3(k)=v_3(j_1)=v_3(k-4\cdot 3^{m-2}). 
\end{align*}
In each of these cases, 
\[
D_{m-1,m-1}=m-3-(m-2)=-1.
\]
In all remaining cases, we have $D_{m-1,m-1}\ge 0$. The equality 
$D_{m-1,m-1}=0$ occurs only in the following exceptional cases:
\[
j_1=2\cdot3^{m-2},\quad j_2=4\cdot3^{m-2};
\]
or, when $m\ge 3$, in one of the following cases: 
\begin{align*}
&j_1=2\cdot3^{m-3}\ \ \text{or}\ j_2=2\cdot3^{m-3}, \quad v_3(k)=v_3(k-2\cdot3^{m-3});\\
&j_1=4\cdot3^{m-3}\ \ \text{or}\ j_2=4\cdot3^{m-3}, \quad v_3(k)=v_3(k-4\cdot3^{m-3});\\
&j_2=8\cdot3^{m-3}, \quad v_3(k)=v_3(k-8\cdot3^{m-3});\\
&j_2=10\cdot3^{m-3}, \quad v_3(k)=v_3(k-10\cdot3^{m-3}).
\end{align*}
Thus (\ref{4.10'}) is proved in this subcase.

The equality cases above yield the following possibilities for $v_3(g_{2,k}(i))$. 

First, $v_3(g_{2,k}(i))=W_{m,k}-1$ in the following cases: 
\begin{align*}
& 2\cdot 3^{m-2}+2 \le k \le 2\cdot3^{m-1}-2, &&i=3^{m-1}-2\cdot 3^{m-2};\\
& 2\cdot 3^{m-2}+2 \le k \le 4\cdot3^{m-2}, &&i=3^{m-1}-(k-2\cdot3^{m-2});\\
& 4\cdot 3^{m-2}+2 \le k \le 2\cdot3^{m-1}-2, &&i=3^{m-1}-(k-4\cdot3^{m-2}).
\end{align*}
Second, $v_3(g_{2,k}(i))=W_{m,k}$ if 
\[
k=2\cdot3^{m-1}, \quad i=3^{m-2};
\]
or, when $m\ge 3$, in one of the following cases: 
\[
\begin{array}{ll}
2\cdot3^{m-3}+2\le k\le14\cdot3^{m-3}-2,\ 
k\notin\{2\cdot3^{m-2},8\cdot3^{m-3},4\cdot3^{m-2}\},
& i=3^{m-1}-2\cdot3^{m-3};\\
2\cdot3^{m-3}+2\le k\le8\cdot3^{m-3}-2,\ 
k\ne2\cdot3^{m-2},
& i=3^{m-1}-(k-2\cdot3^{m-3});\\
4\cdot3^{m-3}+2\le k\le16\cdot3^{m-3}-2,\ 
k\notin\{2\cdot3^{m-2},10\cdot3^{m-3},4\cdot3^{m-2}\},
& i=3^{m-1}-4\cdot3^{m-3};\\
4\cdot3^{m-3}+2\le k\le10\cdot3^{m-3}-2,\ 
k\ne2\cdot3^{m-2},
& i=3^{m-1}-(k-4\cdot3^{m-3});\\
8\cdot3^{m-3}+2\le k\le14\cdot3^{m-3}-2,\ 
k\ne4\cdot3^{m-2},
& i=3^{m-1}-(k-8\cdot3^{m-3});\\
10\cdot3^{m-3}+2\le k\le16\cdot3^{m-3}-2,\ 
k\ne4\cdot3^{m-2},
& i=3^{m-1}-(k-10\cdot3^{m-3}).
\end{array}
\]
For other $i$ and $k$ in this case, we have $v_3(g_{2,k}(i))\ge W_{m,k}+1$.

\smallskip
\noindent\emph{Subcase 2.3: $m\ge 3$ and $1\le l_2\le m-2$.}

In this case, 
$2\le j_2\le 4\cdot3^{m-3}$ and hence $v_3(j_2)\le m-3$.

If $j_1=2\cdot 3^{m-2}$, then $i=3^{m-2}$. Moreover, 
$3^m-k-i=2\cdot 3^{m-1}-k=2\cdot3^{l_2}-j_2$, which 
implies $v_3(k)=v_3(j_2)$. Therefore
\begin{align*}
&D_{m-1,l_2}=3^{m-1}-3^{l_2}(2(m-l_2-1)+1)+(j_2-1)(m-l_2-1)-1\ge0. 
\end{align*}
The equality in the last inequality can occur only 
when $l_2=m-2$ and $j_2=2$. In that case, 
$k=3^m-3^{m-2}-(2\cdot 3^{m-2}-2)=2\cdot 3^{m-1}+2$,
which contradicts the assumption $2\le k\le 2\cdot 3^{m-1}$.
Hence $D_{m-1,l_2}\ge 1$ when $j_1=2\cdot 3^{m-2}$.

If $2\le j_1\le 2\cdot3^{m-2}-2$, then $v_3(j_1)\le m-3$. Thus 
\begin{align*}
D_{m-1,l_2}
&\ge 3^{m-1}-3^{l_2}(2(m-l_2-1)+1)+j_2(m-l_2-1)+l_2-2-(m-3)\\
&=3^{m-1}-3^{l_2}(2(m-l_2-1)+1)+(j_2-1)(m-l_2-1)
\ge 1.
\end{align*}
Therefore $D_{m-1,l_2}\ge 1$ throughout Subcase 2.3. 
Consequently, $v_3(g_{2,k}(i))\ge W_{m,k}+1$ in this subcase. 

Combining Subcases 2.1 to 2.3, we have $D_{m-1,l_2}\ge -1$ in Case 2. 
The equality cases in Case 2 have been identified in Subcases 2.1 and 2.2. 
In Subcase 2.3, all terms have valuation at least $W_{m,k}+1$.

\smallskip
\noindent\emph{Case 3: $m\ge 3$ and $1\le l_1 \le m-2$.}

We first note that $l_2\ge m-1$. Indeed, if $l_2\le m-2$, then
$3^{m-1}\le 3^m-k\le 3^{m-2}-2+2\cdot3^{m-2}-2=3^{m-1}-4$, 
which is a contradiction. Hence $l_2=m-1$ or $l_2=m$.
Consider the following two subcases.

\smallskip
\noindent\emph{Subcase 3.1: $l_2=m-1$.}

Then $2\le j_2\le 4\cdot3^{m-2}$. We first show that $2\le j_2\le 4\cdot3^{m-2}-2$. 
Indeed, if $j_2=4\cdot3^{m-2}$, 
then $3^{m-1}\le 3^m-k=3^{l_1}-j_1
+2\cdot3^{m-1}-4\cdot3^{m-2}\le3^{m-1}-2$,
which is a contradiction. Hence $2\le j_2\le 4\cdot3^{m-2}-2$ and so 
$v_3(j_2)\le m-2$. Therefore, 
\begin{align*}
D_{l_1,m-1}
&=\frac{1}{2}(3^{m-1}-3^{l_1})+(3^{l_1}-j_1)(l_1+1-m)+l_1-2+v_3(k)-v_3(j_1)-v_3(j_2)\\
&\ \ge \frac{3^{m-1}}{2}-\frac{3^{l_1}}{2}(2(m-l_1-1)+1)+j_1(m-l_1-1)+l_1-2-(m-2)\\
&\ =\frac{3^{m-1}}{2}-\frac{3^{l_1}}{2}(2(m-l_1-1)+1)+(j_1-1)(m-l_1-1)-1
\ge 0. 
\end{align*}
The equality in the last inequality occurs if and only if
$$l_1=m-2,\quad j_1=2, \quad j_2=2\cdot3^{m-2}.$$
Thus $D_{l_1,m-1}\ge 0$ holds in this subcase.

The equality case in Subcase 3.1 gives $v_3(g_{2,k}(i))=W_{m,k}$ 
only when 
$$k=4\cdot3^{m-2}+2, \quad i=3^{m-2}-2=3^{m-1}-(k-2\cdot 3^{m-2}),\quad m\ge 3.$$
In all other cases of Subcase 3.1, we have $v_3(g_{2,k}(i))\ge W_{m,k}+1$.

\smallskip
\noindent\emph{Subcase 3.2: $l_2=m$.}

Then $2\le j_2\le 4\cdot3^{m-1}$ and
\begin{align*}
D_{l_1,m}&=\frac{1}{2}(3^{m+1}-3^{l_1})+(3^{l_1}-j_1)(l_1+1-m)-j_2
+l_1-1+v_3(k)-v_3(j_1)-v_3(j_2)\\
&\ge \frac{1}{2}(3^{m+1}-3^{l_1})+(3^{l_1}-j_1)(l_1+1-m)-4\cdot3^{m-1}+l_1-1-(m-1)\\
&= \frac{1}{2}(3^{m-1}-3^{l_1})+(3^{l_1}-j_1)(l_1+1-m)+l_1-m\\
&=\frac{3^{m-1}}{2}-\frac{3^{l_1}}{2}(2(m-l_1-1)+1)+(j_1-1)(m-l_1-1)-1
\ge 0. 
\end{align*}
The equality in the last inequality occurs if and only if
$$l_1=m-2,\quad j_1=2, \quad j_2=4\cdot3^{m-1}.$$
Thus $D_{l_1,m}\ge 0$ holds in this subcase.

The equality case in Subcase 3.2 gives $v_3(g_{2,k}(i))=W_{m,k}$ 
only when 
$$k=2\cdot3^{m-2}+2, \quad i=3^{m-2}-2=3^{m-1}-k,\quad m\ge 3.$$
In all other cases of Subcase 3.2, we have $v_3(g_{2,k}(i))\ge W_{m,k}+1$.

Combining Subcases 3.1 and 3.2, we obtain $D_{l_1,l_2}\ge 0$ 
throughout Case 3. Hence all terms covered by Case 3 have 3-adic valuation at least $W_{m,k}$.  
The equality cases have been identified in Subcases 3.1 and 3.2; in all other cases, 
the 3-adic valuation of $g_{2,k}(i)$ is at least $W_{m,k}+1$. 

Combining Cases 1 to 3, we have proved that $D_{l_1,l_2}\ge -1$ in all remaining cases. 
Together with (\ref{4.9'}), this gives $v_3(g_{2,k}(i))\ge W_{m,k}-1$ for every odd integer $i$ 
with $1\le i\le 3^m-k$. Hence the first estimate in (\ref{4.4}) is proved.

\medskip
\noindent\textbf{Even indices.}

We now prove the second estimate in (\ref{4.4}). Let $i$ be an even integer with $1\le i\le 3^m-k$.
Then $2\le i\le 3^m-k-1$. If $i> 3^n$ or $3^m-k-i> 2\cdot3^n$, 
then $g_{2,k}(i)=0$, and there is nothing to prove. 
Hence we may assume that $2\le i\le 3^n-1$ and $1\le 3^m-k-i\le 2\cdot 3^n$.

Since $i-1$ is odd and $1\le i-1\le 3^m-k-2$, the first estimate in (\ref{4.4}) gives 
\begin{align}
v_3(g_{2,k}(i-1))\ge W_{m, k}-1.\label{4.11'}
\end{align}
On the other hand, since $i$ is even and $2\le i\le 3^n-1$, and since 
$3^m-k-i$ is odd with $1\le 3^m-k-i\le 2\cdot 3^n$, 
the estimates (\ref{3.10}), Lemma \ref{lem2.6}, (\ref{3.12}) and (\ref{4.8}) imply 
\begin{align}
v_3(g_{2,k}(i))
&=v_3(s(3^n, i))+v_3(s_{3^n}(2\cdot3^n, 3^m-k-i))\notag\\
&\ge v_3(s(3^n,i+1))+n+v_3(s(2\cdot3^n, 3^m-k-i+1))+n\notag\\
&\ge v_3(s(3^n,i-1))+v_3(s(2\cdot3^n,3^m-k-(i-1)))+2
=v_3(g_{2,k}(i-1))+2. \label{4.12'}
\end{align}
Combining (\ref{4.11'}) and (\ref{4.12'}), we obtain $v_3(g_{2,k}(i))\ge W_{m, k}+1.$ 
This proves the second estimate in (\ref{4.4}).

%*******************************Proof of Claim (I)*******************************

\subsection{The leading contribution}

The goal of this subsection is to identify the leading terms and 
to prove that their total contribution has 3-adic valuation $W_{m,k}$. 
More precisely, we shall determine the set $A$ introduced above and prove 
$v_3(S_A)=W_{m, k}$, which is precisely (\ref{4.6}). 

Recall that
\[
A=A_1\cup A_2,\qquad S_A=\sum_{i\in A}g_{2,k}(i),
\]
where $A_1$ and $A_2$ were defined by the conditions
\[
v_3(g_{2,k}(i))=W_{m,k}-1
\quad\text{and}\quad
v_3(g_{2,k}(i))=W_{m,k},
\]
respectively. Thus $A$ consists exactly of the odd-index terms whose
valuations may affect the leading valuation.

We shall prove (\ref{4.6}) in several steps. First, since the sets $A_1$ and 
$A_2$ depend on the value of $k$, we decompose the possible values of $k$ 
and give an explicit description of $A_1$ and $A_2$ in each range. 
Next, we replace $g_{2,k}(i)$ by a corresponding term involving ordinary 
Stirling numbers and reduce the proof of (\ref{4.6}) to an equivalent estimate 
for a modified sum $S_A'$. Finally, we prove this estimate case by case, 
treating the exceptional cases first and then the remaining ranges of $k$. 

\medskip
\noindent\textbf{Decomposition of the parameter range and description of $A_1, A_2$.}

Since the sets $A_1$ and $A_2$ depend on the value of $k$, 
we now decompose the range of $k$. Define
\[
K_1=[2,\,2\cdot 3^{m-2}]\cap (2\mathbb{Z}), \qquad
K_3=\{4\cdot 3^{m-2}\}, \qquad
K_5=\{2\cdot 3^{m-1}\},
\]
and, for $m\ge 3$,
\[
K_2=[2+2\cdot 3^{m-2},\,4\cdot 3^{m-2}-2]\cap (2\mathbb{Z}), \qquad
K_4=[2+4\cdot 3^{m-2},\,2\cdot 3^{m-1}-2]\cap (2\mathbb{Z}).
\]

Moreover, we further decompose the sets $K_1$, $K_2$, and $K_4$ as follows:
\[
K_1=\bigcup_{j=1}^4 K_{1j}, \qquad
K_2=\bigcup_{j=1}^6 K_{2j}, \qquad
K_4=\bigcup_{j=1}^5 K_{4j}.
\]
For $K_1$, we define
\[
\begin{aligned}
K_{11} &= [2+2\cdot 3^{m-3},\, 4\cdot 3^{m-3}-2]\cap (2\mathbb{Z}), \\
K_{12} &= \{4\cdot 3^{m-3}\}, \\
K_{13} &= [2+4\cdot 3^{m-3},\, 2\cdot 3^{m-2}-2]\cap (2\mathbb{Z}),
\end{aligned}
\]
where $K_{12}$ is defined for $m\ge 3$, and $K_{11},K_{13}$ are defined for $m\ge 4$. We set
\[
K_{14}=K_1\setminus \bigcup_{j=1}^3 K_{1j}.
\]
For $K_2$, we define
\[
\begin{aligned}
&K_{21} = \{2+2\cdot 3^{m-2}\},  
&&K_{22} = [4+2\cdot 3^{m-2},\, 8\cdot 3^{m-3}-2]\cap (2\mathbb{Z}), \\
&K_{23} = \{8\cdot 3^{m-3}\}, \quad
&&K_{24} = [2+8\cdot 3^{m-3},\, 10\cdot 3^{m-3}-2]\cap (2\mathbb{Z}), \\
&K_{25} = \{10\cdot 3^{m-3}\}, \quad 
&&K_{26} = [2+10\cdot 3^{m-3},\, 4\cdot 3^{m-2}-2]\cap (2\mathbb{Z}),
\end{aligned}
\]
where $K_{21}$, $K_{22}$, $K_{24}$ and $K_{26}$ are defined for $m\ge 4$.
For $K_4$, we define
\[
\begin{aligned}
&K_{41} = \{2+4\cdot 3^{m-2}\}, 
&&K_{42} = [4+4\cdot 3^{m-2},\, 14\cdot 3^{m-3}-2]\cap (2\mathbb{Z}), \\
&K_{43} = \{14\cdot 3^{m-3}\}, 
&&K_{44} = [2+14\cdot 3^{m-3},\, 16\cdot 3^{m-3}-2]\cap (2\mathbb{Z}),
\end{aligned}
\]
where $K_{41}$, $K_{42}$ and $K_{44}$ are defined for $m\ge 4$. We finally set
\[
K_{45}=K_4\setminus \bigcup_{j=1}^4 K_{4j}.
\]
The following table summarizes the role of this decomposition in the
subsequent leading-term analysis.
\begin{table}[htbp]
\centering
\caption{Roadmap for the decomposition of the even parameter range.}
\label{tab:roadmap-K}
\renewcommand{\arraystretch}{1.25}
\begin{tabular}{c|c|p{7.2cm}}
\hline
Range & Subranges & Role in the leading-term analysis \\
\hline
$K_1$ 
& $K_{11},K_{12},K_{13},K_{14}$ 
& The lower range. Here the possible leading indices are described by
$T_1(k)=\{0,k\}$, while $T_2(k)$ depends on the finer subrange. \\

$K_2$ 
& $K_{21},K_{22},K_{23},K_{24},K_{25},K_{26}$ 
& The middle range. An additional pair of possible leading indices appears,
and $T_1(k)$ contains the terms involving $2\cdot 3^{m-2}$. \\

$K_3$ 
& exceptional singleton 
& Treated separately. In this case $A_1=\{3^{m-2},3^{m-1}\}$ and
$A_2=\varnothing$. \\

$K_4$ 
& $K_{41},K_{42},K_{43},K_{44},K_{45}$ 
& The upper range. The form of $T_1(k)$ changes, and the contribution of
$T_2(k)$ again depends on the finer subrange. \\

$K_5$ 
& exceptional singleton 
& Treated immediately. In this case $A_1=\varnothing$ and
$A_2=\{3^{m-2}\}$. \\
\hline
\end{tabular}
\end{table}

\noindent The table is only intended as a roadmap; the exact descriptions of
$A_1$ and $A_2$ are given below.

We first record the two exceptional cases. If $k\in K_5$, then
\[
A_1=\varnothing\quad \text{and}\quad A_2=\{3^{m-2}\}.
\]
If $k\in K_3$, then
\[
A_1=\{3^{m-2},3^{m-1}\}\quad \text{and}\quad A_2=\varnothing.
\]
For the remaining ranges, it is convenient to describe $A_1$ and $A_2$
by translating finite sets of integers. For any integer $a$ and any finite
set $T$ of integers, write
\[
a-T:=\{a-t\mid t\in T\}.
\]
Then, for $k\in K_1\cup K_2\cup K_4$, the sets $A_1$ and $A_2$ can be
written in the form
\[
A_1=3^{m-1}-T_1(k),\qquad A_2=3^{m-1}-T_2(k),
\]
where the finite sets $T_1(k)$ and $T_2(k)$ are as follows.

If $k\in K_1$, then
\[
T_1(k)=\{0,k\},
\]
and
\[
T_2(k)=
\begin{cases}
\{2\cdot3^{m-3},\,k-2\cdot3^{m-3}\}, & k\in K_{11},\\
\{2\cdot3^{m-3}\}, & k\in K_{12},\\
\{2\cdot3^{m-3},\,k-2\cdot3^{m-3},\,4\cdot3^{m-3},\,k-4\cdot3^{m-3}\}, & k\in K_{13},\\
\varnothing, & k\in K_{14}.
\end{cases}
\]

If $k\in K_2$, then
\[
T_1(k)=\{0,\,2\cdot3^{m-2},\,k-2\cdot3^{m-2}\},
\]
and
\[
T_2(k)=
\begin{cases}
\{k,\,2\cdot3^{m-3},\,k-2\cdot3^{m-3},\,4\cdot3^{m-3},\,k-4\cdot3^{m-3}\}, & k\in K_{21},\\
\{2\cdot3^{m-3},\,k-2\cdot3^{m-3},\,4\cdot3^{m-3},\,k-4\cdot3^{m-3}\}, & k\in K_{22},\\
\{k,\,4\cdot3^{m-3}\}, & m=3,\ k\in K_{23},\\
\{4\cdot3^{m-3}\}, & m\ge4,\ k\in K_{23},\\
\{2\cdot3^{m-3},\,4\cdot3^{m-3},\,k-4\cdot3^{m-3},\,k-8\cdot3^{m-3}\}, & k\in K_{24},\\
\{2\cdot3^{m-3}\}, & k\in K_{25},\\
\{2\cdot3^{m-3},\,4\cdot3^{m-3},\,k-8\cdot3^{m-3},\,k-10\cdot3^{m-3}\}, & k\in K_{26}.
\end{cases}
\]

If $k\in K_4$, then
\[
T_1(k)=\{2\cdot3^{m-2},\,k-4\cdot3^{m-2}\},
\]
and
\[
T_2(k)=
\begin{cases}
\{0,\,k-2\cdot3^{m-2},\,2\cdot3^{m-3},\,4\cdot3^{m-3},\,k-8\cdot3^{m-3},\,k-10\cdot3^{m-3}\}, & k\in K_{41},\\
\{2\cdot3^{m-3},\,4\cdot3^{m-3},\,k-8\cdot3^{m-3},\,k-10\cdot3^{m-3}\}, & k\in K_{42},\\
\{0,\,k-2\cdot3^{m-2},\,4\cdot3^{m-3}\}, & m=3,\ k\in K_{43},\\
\{4\cdot3^{m-3}\}, & m\ge4,\ k\in K_{43},\\
\{4\cdot3^{m-3},\,k-10\cdot3^{m-3}\}, & k\in K_{44},\\
\varnothing, & k\in K_{45}.
\end{cases}
\]

The description above identifies all odd indices that may contribute to the 
leading valuation. We first deal with the immediate case $K_5$.

\medskip
\noindent\textbf{The immediate case $K_5$.}

Suppose that $k\in K_5$. Then $k=2\cdot 3^{m-1}$, 
$A_1=\varnothing$ and $A_2=\{3^{m-2}\}$. Hence 
\begin{align*}
v_3\Big(\sum_{i\in A}g_{2, k}(i)\Big)
=v_3\Big(\sum_{i\in A_2}g_{2, k}(i)\Big)=v_3(g_{2, 2\cdot 3^{m-1}}(3^{m-2}))=W_{m, k}.
\end{align*}
Thus (\ref{4.6}) holds when $k\in K_5$.

\medskip
\noindent\textbf{Reduction to ordinary Stirling numbers.}

It remains to consider $k\in \bigcup_{j=1}^4K_j$. For $i\in A$, we have
$2\le 3^m-k-i\le 2\cdot3^n$. 
By Lemma \ref{lem2.6}, applied with $a=2$, we have
\[
g_{2,k}(i)=g'_{2,k}(i)+O(3^{v_3(g'_{2,k}(i))+2}),
\]
where $O(3^r)$ denotes an integer with 3-adic valuation at least $r$, and
\[
g'_{2,k}(i):=s(3^n,i)s(2\cdot3^n,3^m-k-i).
\]
Since $i\in A=A_1\cup A_2$, this gives
\[
g_{2,k}(i)=g'_{2,k}(i)+O(3^{W_{m,k}+1}).
\]
Consequently,
\[
v_3(g'_{2,k}(i))=v_3(g_{2,k}(i))=
\begin{cases}
W_{m,k}-1, & i\in A_1,\\
W_{m,k}, & i\in A_2.
\end{cases}
\]
Define
\[
S_A':=\sum_{i\in A}g'_{2,k}(i).
\]
Then $S_A=S_A'+O(3^{W_{m,k}+1})$. Therefore, to prove \eqref{4.6}, it suffices to prove
\begin{align}\label{4.9}
v_3(S_A')
=v_3\Big(\sum_{i\in A}s(3^n,i)s(2\cdot3^n,3^m-k-i)\Big)
=W_{m,k}.
\end{align} 

\medskip
\noindent\textbf{The exceptional case $K_3$.} 

We first prove (\ref{4.9}) for $k\in K_3$. In this case, 
$k=4\cdot 3^{m-2}$, $A_1=\{3^{m-2}, 3^{m-1}\}$ and 
$A_2=\varnothing$. Thus
\begin{align}\label{4.10}
S_A'=\sum_{i\in A_1}g_{2, k}'(i)
=g_{2, k}'(3^{m-2})+g_{2, k}'(3^{m-1})
=s_{n, m-2, 0}^{(1)}\cdot s_{n, m-1, 2\cdot 3^{m-2}}^{(2)}
+s_{n, m-1, 0}^{(1)}\cdot s_{n, m-1, 4\cdot 3^{m-2}}^{(2)}.
\end{align}
In Remark \ref{rem3.2}, by setting $(n,m,k)\mapsto(n, m-1,2\cdot3^{m-2})$
and $(n,m,k)\mapsto(n, m-1,4\cdot3^{m-2})$, respectively, we obtain that
\begin{align}
s^{(2)}_{n, m-1, 2\cdot3^{m-2}}=2s^{(1)}_{n, m-1, 0}\cdot s^{(1)}_{n, m-1, 2\cdot3^{m-2}}+L_{1}
=2s^{(1)}_{n, m-1, 0}\cdot s^{(1)}_{n, m-2, 0}+L_{1}\ \ \text{and}\ \  
s^{(2)}_{n, m-1, 4\cdot3^{m-2}}=(s^{(1)}_{n, m-2, 0})^2+L_{2},\label{4.12}
\end{align}
where $v_3(L_{1})\ge v_3(s^{(1)}_{n, m-1, 0}\cdot s^{(1)}_{n, m-2, 0})+2$
and $v_3(L_{2})\ge v_3((s^{(1)}_{n, m-2, 0})^2)+2$.
By (\ref{4.10}) and (\ref{4.12}) we have
\begin{align}\label{4.13}
&S_A'=3s^{(1)}_{n, m-1, 0}(s^{(1)}_{n, m-2, 0})^2
+s^{(1)}_{n, m-2, 0}L_{1}+s^{(1)}_{n, m-1, 0}L_{2}.
\end{align}

By the definition of $A_1$ together with (\ref{4.10}) and (\ref{4.12}),
we deduce that
\begin{align}
&v_3(g_{2, k}'(3^{m-2}))=v_3(g_{2, k}'(3^{m-1}))=W_{m, k}-1
=v_3(s^{(1)}_{n, m-1, 0}(s^{(1)}_{n, m-2, 0})^2)\label{4.14}
\end{align}
and
$v_3(s^{(1)}_{n, m-2, 0}L_{1}+s^{(1)}_{n, m-1, 0}L_{2})
\ge v_3(s^{(1)}_{n, m-1, 0}(s^{(1)}_{n, m-2, 0})^2)+2=W_{m, k}+1$.
Thus, (\ref{4.13}) and (\ref{4.14}) give 
\begin{align*}
v_3(S_A')=v_3(3s^{(1)}_{n, m-1, 0}(s^{(1)}_{n, m-2, 0})^2)=W_{m,k}.
\end{align*}
Hence (\ref{4.9}) is proved when $k\in K_3$.

\medskip
\noindent\textbf{The remaining cases.} 

It remains to prove (\ref{4.9}) for $k\in K_1\cup K_2\cup K_4$. 
In these cases, the terms indexed by $A_1$ produce the principal 
contribution, while the terms indexed by $A_2$ must be combined 
with certain secondary terms arising from the expansion in Remark \ref{rem3.2}. 

To keep the formulas readable, we introduce the following abbreviations. 
For $a_1, a_2, a_3\in \mathbb Z$, let
\begin{align*}
&S_{a_1, a_2}:= s^{(1)}_{n, m-1,a_1}\cdot s^{(1)}_{n, m-1, a_2},
&&S_{a_1, a_2, a_3}:= s^{(1)}_{n, m-1,a_1}
\cdot s^{(1)}_{n, m-1, a_2}\cdot s^{(1)}_{n, m-1, a_3},\\
&S'_{a_1, a_2}:= s^{(1)}_{n, m-2,a_1}\cdot s^{(1)}_{n, m-1, a_2},
&&S'_{a_1, a_2, a_3}:= s^{(1)}_{n, m-2,a_1}
\cdot s^{(1)}_{n, m-1, a_2}\cdot s^{(1)}_{n, m-1, a_3},\\
&S''_{a_1, a_2}:= s^{(1)}_{n, m-2,a_1}\cdot s^{(1)}_{n, m-2, a_2}, 
&&S''_{a_1, a_2, a_3}:= s^{(1)}_{n, m-2,a_1}
\cdot s^{(1)}_{n, m-2, a_2}\cdot s^{(1)}_{n, m-1, a_3}.
\end{align*}
We now treat the cases $k\in K_{1}, k\in K_{2}$ and $ k\in K_{4}$ separately. 
%*******************************Proof of Claim (I)--case 1*******************************

\smallskip
\noindent\emph{Case 1: $k\in K_{1}$.}

Then $2\le k\le 2\cdot3^{m-2}$ and
$A_1=3^{m-1}-\{0, k\}$. Hence 
\begin{align}\label{4.15}
g_{2,k}'(3^{m-1})
=s^{(1)}_{n, m-1, 0}\cdot s^{(2)}_{n, m-1, k},\quad\quad 
g_{2,k}'(3^{m-1}-k)
=s^{(1)}_{n, m-1, k}\cdot s^{(2)}_{n, m-1, 0}.
\end{align}
Setting $(n,m,k)\mapsto(n, m-1,k)$ in Remark \ref{rem3.2}, we get 
\begin{align}\label{4.17}
s^{(2)}_{n, m-1, k}=2s^{(1)}_{n, m-1, 0}\cdot s^{(1)}_{n, m-1, k}+L_{11}+L_{12}
=2S_{0, k}+L_{11}+L_{12},
\end{align}
where $L_{11}, L_{12}\in \mathbb{Z}$,
$v_3(L_{11})\ge v_3(S_{0, k})+2$ and
$v_3(L_{12})\ge v_3(S_{0, k})+1$ with 
\[
L_{12}=
\begin{cases}
2s^{(1)}_{n, m-1, 2\cdot3^{m-3}}\cdot s^{(1)}_{n, m-1, k-2\cdot3^{m-3}}, & k\in K_{11},\\
(s^{(1)}_{n, m-1, 2\cdot3^{m-3}})^2, & k\in K_{12},\\
2s^{(1)}_{n, m-1, 2\cdot3^{m-3}}\cdot s^{(1)}_{n, m-1, k-2\cdot3^{m-3}}
+2s^{(1)}_{n, m-1, 4\cdot3^{m-3}}\cdot s^{(1)}_{n, m-1, k-4\cdot3^{m-3}}, & k\in K_{13},\\
0, & k\in K_{14}.
\end{cases}
\]
Moreover, since $s^{(2)}_{n, m-1, 0}=s^{(1)}_{n, m-1, 0}=1$ if $m=n+1$, and
$s^{(2)}_{n, m-1, 0}=s^{(2)}_{n, m, 4\cdot3^{m-1}}$ if $m\le n$,
Remark \ref{rem3.2} gives 
\begin{align}\label{4.18}
s^{(2)}_{n, m-1, 0}=(s^{(1)}_{n, m-1, 0})^2+L_{13}=S_{0, 0}+L_{13}, 
\end{align}
where $v_3(L_{13})\ge v_3(S_{0, 0})+2$.

Substituting (\ref{4.17}) and (\ref{4.18}) into (\ref{4.15}), we obtain 
\begin{align}\label{4.19}
\sum_{i\in A_1}g_{2,k}'(i)=g_{2,k}'(3^{m-1})+g_{2,k}'(3^{m-1}-k)=S_{11}+S_{12}+S_{13}, 
\end{align}
where 
\[
S_{11}:=3(s^{(1)}_{n, m-1, 0})^2s^{(1)}_{n, m-1, k}=3S_{0, 0, k},\quad
S_{12}:=s^{(1)}_{n, m-1, 0}L_{12},\quad
S_{13}:=s^{(1)}_{n, m-1, 0}L_{11}+s^{(1)}_{n, m-1, k}L_{13}.
\]
By the definition of $A_1$ together with the second identity in (\ref{4.15}) and (\ref{4.18}),
we have 
\begin{align}
v_3(g_{2,k}'(3^{m-1}-k))=W_{m, k}-1=v_3(S_{0, 0, k}). \label{4.20}
\end{align}
So we have
\begin{align}\label{4.21}
v_3(S_{11})= v_3(S_{0, 0, k})+1=W_{m,k}
\ \ \textrm{and}\ \
v_3(S_{13})\ge v_3(S_{0, 0, k})+2=W_{m,k}+1.
\end{align}

It remains to deal with $S_{12}$. We first observe that 
$$S_{12}=0 \quad \Longleftrightarrow \quad k\in K_{14}
\quad \Longleftrightarrow \quad A_2=\varnothing.$$
If $k\in K_{14}$, then $S_{12}=0$ and $A_2=\varnothing$. 
Hence, by (\ref{4.19}) and (\ref{4.21}),  
$$
v_3(S_A')=v_3\Big(\sum_{i\in A}g_{2,k}'(i)\Big)
=v_3\Big(\sum_{i\in A_1}g_{2,k}'(i)\Big)=v_3(S_{11})=W_{m,k}. 
$$
Thus (\ref{4.9}) holds when $k\in K_{14}$.  

It remains to consider $k\in K_{11}\cup K_{12}\cup K_{13}$. In this case, 
$S_{12}\ne 0$ and $A_2\ne \varnothing$. We claim that 
\begin{align}
v_3\Big(S_{12}+\sum_{i\in A_2}g_{2,k}'(i)\Big)\ge W_{m,k}+1.\label{4.22}
\end{align}
Assuming (\ref{4.22}), it follows from (\ref{4.19}) and (\ref{4.21}) that
\begin{align*}
v_3(S_A')&=v_3\Big(\sum_{i\in A}g_{2,k}'(i)\Big)
=v_3\Big(S_{11}+S_{12}+S_{13}+\sum_{i\in A_2}g_{2,k}'(i)\Big)=v_3(S_{11})=W_{m,k}.
\end{align*}
Thus (\ref{4.9}) holds in Case 1, once (\ref{4.22}) is proved. 

We now prove (\ref{4.22}). We do this according to the 
three possible forms of $A_2$, corresponding to $k\in K_{11}$, $k\in K_{12}$, 
and $k\in K_{13}$. In each subcase, we combine $S_{12}$ with the terms indexed 
by $A_2$ and show that the resulting sum has 3-adic valuation at least $W_{m,k}+1$. 

%*******************************Proof of Claim (I)--case 1*******************************

\smallskip
\noindent\emph{Subcase 1.1: $k\in K_{11}$.}

In this subcase,  $m\ge 4$ and $A_2=3^{m-1}-\{2\cdot3^{m-3}, k-2\cdot3^{m-3}\}$. 
Moreover, 
\begin{align}\label{4.23}
S_{12}&=s^{(1)}_{n, m-1,0}L_{12}=2S_{0, 2\cdot3^{m-3}, k-2\cdot3^{m-3}}.
\end{align}
Since $2\le k-2\cdot3^{m-3}\le 2\cdot3^{m-3}-2$, Remark \ref{rem3.2} gives 
\begin{equation}\label{4.24}
\begin{aligned}
g_{2,k}'(3^{m-1}-2\cdot3^{m-3})
&=s^{(1)}_{n, m-1,2\cdot3^{m-3}}\cdot s^{(2)}_{n, m-1,k-2\cdot3^{m-3}}
=2S_{0, 2\cdot3^{m-3}, k-2\cdot3^{m-3}}+L_{111}, \\
g_{2,k}'(3^{m-1}-(k-2\cdot3^{m-3}))
&=s^{(1)}_{n, m-1, k-2\cdot3^{m-3}}\cdot s^{(2)}_{n, m-1, 2\cdot3^{m-3}}
=2S_{0, 2\cdot3^{m-3}, k-2\cdot3^{m-3}}+L_{112},
\end{aligned}
\end{equation}
where $\min\{v_3(L_{111}),v_3(L_{112})\}\ge v_3(S_{0, 2\cdot3^{m-3}, k-2\cdot3^{m-3}})+1$. 

Since $A_2=3^{m-1}-\{2\cdot3^{m-3}, k-2\cdot3^{m-3}\}$ and 
\begin{align}\label{4.26}
v_3(g_{2,k}'(3^{m-1}-2\cdot3^{m-3}))=W_{m,k}=
v_3(S_{0, 2\cdot3^{m-3}, k-2\cdot3^{m-3}}),
\end{align}
from (\ref{4.23}) to (\ref{4.26}), we derive that
\begin{align*}
v_3\Big(S_{12}+\sum_{i\in A_2}g_{2,k}'(i)\Big)
=v_3(6S_{0, 2\cdot3^{m-3}, k-2\cdot3^{m-3}}+L_{111}+L_{112})
\ge v_3(S_{0, 2\cdot3^{m-3}, k-2\cdot3^{m-3}})+1
=W_{m,k}+1.
\end{align*}

Thus (\ref{4.22}) is proved in this subcase.

\smallskip
\noindent\emph{Subcase 1.2: $k\in K_{12}$.}

In this subcase, $m\ge 3$ and $k=4\cdot3^{m-3}$. We have
\begin{align}\label{4.27}
S_{12}=s^{(1)}_{n, m-1, 0}L_{12}
=s^{(1)}_{n, m-1, 0} (s^{(1)}_{n, m-1, 2\cdot3^{m-3}})^2
=S_{0, 2\cdot3^{m-3}, 2\cdot3^{m-3}}
\end{align}
and $A_{2}=\{3^{m-1}-2\cdot3^{m-3}\}$.
It follows from Remark \ref{rem3.2} that
\begin{align}\label{4.28}
&g_{2, k}'(3^{m-1}-2\cdot3^{m-3})
=s^{(1)}_{n, m-1, 2\cdot3^{m-3}}\cdot s^{(2)}_{n, m-1, 2\cdot3^{m-3}}
=2S_{0, 2\cdot3^{m-3}, 2\cdot3^{m-3}}+L_{121}, 
\end{align}
where $v_3(L_{121})\ge v_3(S_{0, 2\cdot3^{m-3}, 2\cdot3^{m-3}})+1$.
Therefore, by (\ref{4.27}) and (\ref{4.28}), we obtain  
\begin{align*}
v_3\Big(S_{12}+\sum_{i\in A_2}g_{2,k}'(i)\Big)
&=v_3(3S_{0, 2\cdot3^{m-3}, 2\cdot3^{m-3}}+L_{121})
\ge v_3(S_{0, 2\cdot3^{m-3}, 2\cdot3^{m-3}})+1\\
&=v_3(g_{2, k}'(3^{m-1}-2\cdot3^{m-3}))+1=W_{m,k}+1.
\end{align*} 
Thus (\ref{4.22}) is proved in this subcase.

\smallskip
\noindent\emph{Subcase 1.3: $k\in K_{13}$.}

In this subcase, $m\ge 4$ and $2+4\cdot3^{m-3}\le k\le 2\cdot3^{m-2}-2$. 
Moreover, 
\begin{align}
S_{12}
&=s^{(1)}_{n, m-1, 0}L_{12}
= 2S_{0, 2\cdot3^{m-3}, k-2\cdot3^{m-3}}
+2S_{0, 4\cdot3^{m-3}, k-4\cdot3^{m-3}}\label{4.29}
\end{align}
and $A_{2}=3^{m-1}-\{2\cdot3^{m-3}, k-2\cdot3^{m-3},
4\cdot3^{m-3}, k-4\cdot3^{m-3}\}$. 
The two terms corresponding to $2\cdot3^{m-3}$ and $k-2\cdot3^{m-3}$ 
are handled as in Subcase 1.1. Since $k-2\cdot3^{m-3}\in K_{1}$,
(\ref{4.24}) and (\ref{4.26}) remain valid. 

For the remaining two terms, Remark \ref{rem3.2} gives 
\begin{align}
g_{2,k}'(3^{m-1}-4\cdot3^{m-3})
&=s^{(1)}_{n, m-1, 4\cdot3^{m-3}}\cdot s^{(2)}_{n, m-1, k-4\cdot3^{m-3}}
=2S_{0, 4\cdot3^{m-3}, k-4\cdot3^{m-3}}+L_{131},\label{4.30}\\  
g_{2,k}'(3^{m-1}-(k-4\cdot3^{m-3}))
&=s^{(1)}_{n, m-1, k-4\cdot3^{m-3}}\cdot s^{(2)}_{n, m-1, 4\cdot3^{m-3}}
= 2S_{0, 4\cdot3^{m-3}, k-4\cdot3^{m-3}}+L_{132},\label{4.31}
\end{align}
where $\min\{v_3(L_{131}), v_3(L_{132})\}\ge v_3(S_{0, 4\cdot3^{m-3}, k-4\cdot3^{m-3}})+1$.
Note that
\begin{align}\label{4.32}
v_3(g_{2,k}'(3^{m-1}-4\cdot3^{m-3}))=W_{m,k}=
v_3(S_{0, 4\cdot3^{m-3}, k-4\cdot3^{m-3}}).
\end{align}
Combining (\ref{4.24}), (\ref{4.26}) and 
(\ref{4.29}) to (\ref{4.32}), we obtain 
\begin{align*}
v_3\Big(S_{12}+\sum_{i\in A_2}g_{2,k}'(i)\Big)
&=v_3(6S_{0, 2\cdot3^{m-3}, k-2\cdot3^{m-3}}+L_{111}+L_{112}
+6S_{0, 4\cdot3^{m-3}, k-4\cdot3^{m-3}}+L_{131}+L_{132})\\
&\ge \min\{v_3(S_{0, 2\cdot3^{m-3}, k-2\cdot3^{m-3}}),
v_3(S_{0, 4\cdot3^{m-3}, k-4\cdot3^{m-3}})\}+1 = W_{m,k}+1.
\end{align*}
Thus (\ref{4.22}) is proved in this subcase. 

Combining Subcases 1.1 to 1.3, we have proved (\ref{4.22}). 
Hence (\ref{4.9}) holds in Case 1.

%*******************************Proof of Claim (I)--case 2*******************************

\smallskip
\noindent\emph{Case 2: $k\in K_{2}$.}

Then $m\ge 3$,
$A_1=3^{m-1}-\{0, 2\cdot3^{m-2}, k-2\cdot3^{m-2}\}$. 
We first record the three terms indexed by $A_1$. 
Recall from the first identity in (\ref{4.15}) that 
$$g_{2,k}'(3^{m-1})=s^{(1)}_{n, m-1,0}\cdot s^{(2)}_{n, m-1, k}. $$
In addition, we have
\begin{align}
&g_{2,k}'(3^{m-1}-2\cdot3^{m-2})=g_{2,k}'(3^{m-2})
=s^{(1)}_{n, m-2,0}\cdot s^{(2)}_{n, m-1, k-2\cdot3^{m-2}}, \label{4.33}\\
&g_{2,k}'(3^{m-1}-(k-2\cdot3^{m-2}))
=s^{(1)}_{n, m-1, k-2\cdot3^{m-2}}\cdot s^{(2)}_{n, m-1, 2\cdot3^{m-2}}.\label{4.34}
\end{align}

We now expand the three $s^{(2)}$-terms appearing above. 
First, setting $(n,m,k)\mapsto(n, m-1, 2\cdot3^{m-2})$ in Remark \ref{rem3.2}, we get 
\begin{align}\label{4.35}
s^{(2)}_{n, m-1, 2\cdot3^{m-2}}
=2s^{(1)}_{n, m-1, 0}\cdot s^{(1)}_{n, m-1, 2\cdot3^{m-2}}+L_{21}
=2s^{(1)}_{n, m-2, 0}\cdot s^{(1)}_{n, m-1, 0}+L_{21}
=2S'_{0,0}+L_{21}, 
\end{align}
where $L_{21}\in \mathbb{Z}$ and $v_3(L_{21}) \ge v_3(S'_{0,0})+2$.

Next, since $m\ge 3$ and $2+2\cdot 3^{m-2}\le k\le 4\cdot 3^{m-2}-2$, 
Remark \ref{rem3.2} gives 
\begin{align}
s^{(2)}_{n, m-1, k}
=2s^{(1)}_{n, m-2, 0}\cdot s^{(1)}_{n, m-1, k-2\cdot3^{m-2}}+L_{22}+L_{23}
=2S'_{0, k-2\cdot3^{m-2}}+L_{22}+L_{23}, \label{4.36}
\end{align}
where $L_{22}, L_{23}\in \mathbb{Z}$,
$v_3(L_{23})\ge v_3(S'_{0, k-2\cdot3^{m-2}})+2$,
$v_3(L_{22})\ge v_3(S'_{0, k-2\cdot3^{m-2}})+1$ with 
\[
L_{22}=
\begin{cases}
2s^{(1)}_{n,m-2,2}s^{(1)}_{n,m-1,0}
+2s^{(1)}_{n,m-1,2\cdot3^{m-3}}s^{(1)}_{n,m-1,k-2\cdot3^{m-3}}
+2s^{(1)}_{n,m-1,4\cdot3^{m-3}}s^{(1)}_{n,m-1,k-4\cdot3^{m-3}},
& k\in K_{21},\\[1mm]
2s^{(1)}_{n,m-1,2\cdot3^{m-3}}s^{(1)}_{n,m-1,k-2\cdot3^{m-3}}
+2s^{(1)}_{n,m-1,4\cdot3^{m-3}}s^{(1)}_{n,m-1,k-4\cdot3^{m-3}},
& k\in K_{22},\\[1mm]
2s^{(1)}_{n,m-2,2}s^{(1)}_{n,m-1,0}
+\left(s^{(1)}_{n,m-1,4\cdot3^{m-3}}\right)^2,
& k\in K_{23},\ m=3,\\[1mm]
\left(s^{(1)}_{n,m-1,4\cdot3^{m-3}}\right)^2,
& k\in K_{23},\ m\ge4,\\[1mm]
2s^{(1)}_{n,m-1,4\cdot3^{m-3}}s^{(1)}_{n,m-1,k-4\cdot3^{m-3}},
& k\in K_{24},\\[1mm]
0, & k\in K_{25}\cup K_{26}.
\end{cases}
\]

Finally, since $k-2\cdot 3^{m-2}\in K_{1}$,
replacing $k$ by $k-2\cdot3^{m-2}$ in (\ref{4.17}) gives 
\begin{align}
s^{(2)}_{n, m-1, k-2\cdot3^{m-2}}
=2s^{(1)}_{n, m-1, 0}\cdot s^{(1)}_{n, m-1, k-2\cdot3^{m-2}}+L_{24}+L_{25}
=2S_{0, k-2\cdot3^{m-2}}+L_{24}+L_{25}, \label{4.37}
\end{align}
where $L_{24}, L_{25}\in \mathbb{Z}$,
$v_3(L_{25})\ge v_3(S_{0, k-2\cdot3^{m-2}})+2$ and
$v_3(L_{24})\ge v_3(S_{0, k-2\cdot3^{m-2}})+1$ with
\[
L_{24}=
\begin{cases}
2s^{(1)}_{n,m-1,2\cdot3^{m-3}}s^{(1)}_{n,m-1,k-8\cdot3^{m-3}},
& k\in K_{24},\\[1mm]
\left(s^{(1)}_{n,m-1,2\cdot3^{m-3}}\right)^2,
& k\in K_{25},\\[1mm]
2s^{(1)}_{n,m-1,2\cdot3^{m-3}}s^{(1)}_{n,m-1,k-8\cdot3^{m-3}}
+2s^{(1)}_{n,m-1,4\cdot3^{m-3}}s^{(1)}_{n,m-1,k-10\cdot3^{m-3}},
& k\in K_{26},\\[1mm]
0, & k\in K_{21}\cup K_{22}\cup K_{23}.
\end{cases}
\]

Substituting (\ref{4.35}) to (\ref{4.37}) into the first identity in (\ref{4.15}), (\ref{4.33}), 
and (\ref{4.34}), we obtain 
\begin{align}\label{4.38}
\sum_{i\in A_1}g_{2,k}'(i)
=g_{2,k}'(3^{m-1})+ g_{2,k}'(3^{m-2})+g_{2,k}'(3^{m-1}-(k-2\cdot3^{m-2}))
=S_{21}+S_{22}+S_{23},
\end{align}
where
\begin{align*}
S_{21}&:=6s^{(1)}_{n, m-2, 0}\cdot s^{(1)}_{n, m-1, 0}\cdot s^{(1)}_{n, m-1, k-2\cdot3^{m-2}}
=6S'_{0, 0, k-2\cdot3^{m-2}},\\
S_{22}&:=s^{(1)}_{n, m-2, 0}L_{24}+s^{(1)}_{n, m-1, 0}L_{22},\quad 
S_{23}:=s^{(1)}_{n, m-2, 0}L_{25}+s^{(1)}_{n, m-1, 0}L_{23}
+s^{(1)}_{n, m-1, k-2\cdot3^{m-2}}L_{21}.
\end{align*}
By the first identity in (\ref{4.15}) and (\ref{4.36}), we have 
\begin{align}
&v_3(g_{2,k}'(3^{m-1}))=W_{m, k}-1
=v_3(S'_{0, 0, k-2\cdot3^{m-2}}).\label{4.39}
\end{align}
Therefore, 
\begin{align}
v_3(S_{21})
&=v_3(S'_{0, 0, k-2\cdot3^{m-2}})+1=W_{m,k}\label{4.40}
\end{align}
and the bounds for $L_{21}, L_{23}, L_{25}$ imply 
\begin{align}\label{4.41}
v_3(S_{23})\ge v_3(S'_{0, 0, k-2\cdot3^{m-2}})+2=W_{m,k}+1.
\end{align}

It remains to control $S_{22}$. We claim that 
\begin{align}
v_3\Big(S_{22}+\sum_{i\in A_2}g_{2,k}'(i)\Big)\ge W_{m,k}+1.\label{4.42}
\end{align}
Assuming (\ref{4.42}), it follows from 
(\ref{4.38}), (\ref{4.40}) and (\ref{4.41}) that
\begin{align*}
v_3(S_A')&=v_3\Big(\sum_{i\in A}g_{2,k}'(i)\Big)
=v_3\Big(\sum_{i\in A_1}g_{2,k}'(i)+\sum_{i\in A_2}g_{2,k}'(i)\Big)\\
&=v_3\Big(S_{21}+S_{22}+S_{23}
+\sum_{i\in A_2}g_{2,k}'(i)\Big)=v_3(S_{21})=W_{m,k}.
\end{align*}
Thus (\ref{4.9}) holds in Case 2 once (\ref{4.42}) is proved. 

We now prove (\ref{4.42}). The argument follows the same pattern as in 
Case 1: the secondary term $S_{22}$ is combined with the contribution from $A_2$. 
Since $A_2$ has six possible forms, corresponding to $k\in K_{21}, \ldots, K_{26}$, 
we treat these cases separately. In each case, we prove that the combined contribution 
has 3-adic valuation at least $W_{m,k}+1$.

%*******************************Proof of Claim (I)--case 2.1 *******************************
\smallskip
\noindent\emph{Subcase 2.1: $k\in K_{21}$.}

In this subcase, $k=2+2\cdot3^{m-2}$, $m\ge 4$, $L_{22}\ne 0, L_{24}=0$,
and 
\[
A_2=3^{m-1}-\{k, 2\cdot3^{m-3}, k-2\cdot3^{m-3},
4\cdot3^{m-3}, k-4\cdot3^{m-3}\}.
\]
Moreover, 
\begin{align}
S_{22}&=s^{(1)}_{n, m-1, 0}L_{22}
=2S'_{2, 0, 0}+2S_{0, 2\cdot3^{m-3}, k-2\cdot3^{m-3}}
+2S_{0, 4\cdot3^{m-3}, k-4\cdot3^{m-3}}.\label{4.43}
\end{align}
Note that $s^{(1)}_{n, m-1, k}=s^{(1)}_{n, m-2, 2}$, 
the second identity in (\ref{4.15}) and (\ref{4.18}) remain valid, 
and we have
\begin{align}\label{4.44}
v_3(g_{2, k}'(3^{m-1}-k))=W_{m,k}=v_3(S'_{2, 0, 0}).
\end{align}

Since $k-2\cdot3^{m-3}, k-4\cdot3^{m-3}\in K_{1}$, the estimates 
(\ref{4.24}), (\ref{4.26}) and (\ref{4.30}) to (\ref{4.32}) remain valid. 
Hence by the second identity in (\ref{4.15}), (\ref{4.18}), 
(\ref{4.24}), (\ref{4.26}), (\ref{4.30}) to (\ref{4.32}),
(\ref{4.43}) and (\ref{4.44}) we obtain 
\begin{align*}
v_3\Big(S_{22}+\sum_{i\in A_2}g_{2,k}'(i)\Big)
&=v_3(3S'_{2, 0, 0}+s^{(1)}_{n, m-2, 2}L_{13}
+6S_{0, 2\cdot3^{m-3}, k-2\cdot3^{m-3}}+L_{111}+L_{112}\\
&\quad \quad \ \ +6S_{0, 4\cdot3^{m-3}, k-4\cdot3^{m-3}}+L_{131}+L_{132})\\
&\ge \min\{v_3(S'_{2, 0, 0}),
v_3(S_{0, 2\cdot3^{m-3}, k-2\cdot3^{m-3}}),
v_3(S_{0, 4\cdot3^{m-3}, k-4\cdot3^{m-3}})\}+1=W_{m,k}+1.
\end{align*}
Thus (\ref{4.42}) is proved in this subcase.

%*******************************Proof of Claim (I)--case 2.2 *******************************

\smallskip
\noindent\emph{Subcase 2.2: $k\in K_{22}$.}

In this subcase, $L_{22}\ne 0$, $L_{24}=0$, and 
$A_2=3^{m-1}-\{2\cdot3^{m-3}, k-2\cdot3^{m-3},
4\cdot3^{m-3}, k-4\cdot3^{m-3}\}$. Moreover, 
\begin{align}
S_{22}&=s^{(1)}_{n, m-1, 0}L_{22}
=2S_{0, 2\cdot3^{m-3}, k-2\cdot3^{m-3}}+2S_{0, 4\cdot3^{m-3}, k-4\cdot3^{m-3}}.\label{4.45}
\end{align}

Since $k-2\cdot3^{m-3}, k-4\cdot3^{m-3}\in K_{1}$,
(\ref{4.24}), (\ref{4.26}) and (\ref{4.30}) to (\ref{4.32}) remain valid.
Combining them with (\ref{4.45}), we get 
\begin{align*}
v_3\Big(S_{22}+\sum_{i\in A_2}g_{2,k}'(i)\Big)
&=v_3(6S_{0, 2\cdot3^{m-3}, k-2\cdot3^{m-3}}+L_{111}+L_{112}
+6S_{0, 4\cdot3^{m-3}, k-4\cdot3^{m-3}}+L_{131}+L_{132})\\
&\ge \min\{v_3(S_{0, 2\cdot3^{m-3}, k-2\cdot3^{m-3}}),
v_3(S_{0, 4\cdot3^{m-3}, k-4\cdot3^{m-3}})\}+1=W_{m,k}+1.
\end{align*}
This proves (\ref{4.42}) in this subcase.

%*******************************Proof of Claim (I)--case 2.3 *******************************

\smallskip
\noindent\emph{Subcase 2.3: $k\in K_{23}$.}

Here $k=8\cdot3^{m-3}$. We consider two cases according to 
whether $m=3$ or $m\ge 4$. 

If $m=3$, then $A_2=3^{m-1}-\{k, 4\cdot3^{m-3}\}$ and
\begin{align}
S_{22}&=s^{(1)}_{n, m-1, 0}L_{22}
=2S'_{2, 0, 0}+S_{0, 4\cdot3^{m-3}, 4\cdot3^{m-3}}. \label{4.46}
\end{align}
Since $s^{(1)}_{n, m-1, k}=s^{(1)}_{n, m-2, 2}$, as in Subcase 2.1, 
the second identity in (\ref{4.15}), (\ref{4.18}) and (\ref{4.44}) hold. Moreover, 
by Remark \ref{rem3.2},
\begin{align}
g_{2,k}'(3^{m-1}-4\cdot3^{m-3})
=s^{(1)}_{n, m-1, 4\cdot3^{m-3}}\cdot s^{(2)}_{n, m-1, 4\cdot3^{m-3}}
=2S_{0, 4\cdot3^{m-3}, 4\cdot3^{m-3}}+L_{231}, \label{4.47}
\end{align}
where $v_3(L_{231})\ge v_3(S_{0, 4\cdot3^{m-3}, 4\cdot3^{m-3}})+1$.
Note that
\begin{align}
&v_3(g_{2,k}'(3^{m-1}-4\cdot3^{m-3}))=W_{m, k}
=v_3(S_{0, 4\cdot3^{m-3}, 4\cdot3^{m-3}}). \label{4.48}
\end{align}
Thus combining the second identity in (\ref{4.15}), 
(\ref{4.18}), (\ref{4.44}), (\ref{4.46}) to (\ref{4.48}), we get 
\begin{align*}
v_3\Big(S_{22}+\sum_{i\in A_2}g_{2,k}'(i)\Big)
&=v_3(3S'_{2, 0, 0} +s^{(1)}_{n, m-2, 2}L_{13}
+3S_{0, 4\cdot3^{m-3}, 4\cdot3^{m-3}}+L_{231})\\
&\ge \min\{v_3(S'_{2, 0, 0}), v_3(S_{0, 4\cdot3^{m-3}, 4\cdot3^{m-3}})\}+1 = W_{m,k}+1.
\end{align*}

If $m\ge 4$, then $A_2=\{3^{m-1}-4\cdot3^{m-3}\}$ and
\begin{align}
S_{22}=s^{(1)}_{n, m-1, 0}L_{22}
=s^{(1)}_{n, m-1, 0} (s^{(1)}_{n, m-1, 4\cdot3^{m-3}})^2
=S_{0, 4\cdot3^{m-3}, 4\cdot3^{m-3}}. \label{4.49}
\end{align}
By (\ref{4.47}) to (\ref{4.49}), we obtain
\begin{align*}
v_3\Big(S_{22}+\sum_{i\in A_2}g_{2,k}'(i)\Big)
=v_3(3S_{0, 4\cdot3^{m-3}, 4\cdot3^{m-3}}+L_{231})
\ge v_3(S_{0, 4\cdot3^{m-3}, 4\cdot3^{m-3}})+1= W_{m,k}+1.
\end{align*}
Thus (\ref{4.42}) is proved in this subcase.

%*******************************Proof of Claim (I)--case 2.4 *******************************

\smallskip
\noindent\emph{Subcase 2.4: $k\in K_{24}$.}

In this subcase, $m\ge 4$, $2+8\cdot3^{m-3}\le k\le 10\cdot3^{m-3}-2$,
and $A_2=3^{m-1}-\{2\cdot3^{m-3}, 4\cdot3^{m-3},  
k-4\cdot3^{m-3}, k-8\cdot3^{m-3}\}$. Moreover,
\begin{align}
S_{22}&=s^{(1)}_{n, m-2, 0}L_{24}+s^{(1)}_{n, m-1, 0}L_{22}
=2S'_{0, 2\cdot3^{m-3}, k-8\cdot3^{m-3}}
+2S_{0, 4\cdot3^{m-3}, k-4\cdot3^{m-3}}.  \label{4.50}
\end{align}

Since $k-4\cdot3^{m-3}\in K_1$ and $k-2\cdot3^{m-3}, 8\cdot3^{m-3}\in K_2$,
(\ref{4.30}) and (\ref{4.31}) hold, and Remark {\ref{rem3.2}} gives 
\begin{equation}\label{4.51}
\begin{aligned}
g_{2,k}'(3^{m-1}-2\cdot3^{m-3})
&=s^{(1)}_{n, m-1, 2\cdot3^{m-3}}\cdot s^{(2)}_{n, m-1, k-2\cdot3^{m-3}}
=2S'_{0, 2\cdot3^{m-3}, k-8\cdot3^{m-3}}+L_{241},\\
g_{2,k}'(3^{m-1}-(k-8\cdot3^{m-3}))
&=s^{(1)}_{n, m-1, k-8\cdot3^{m-3}}\cdot s^{(2)}_{n, m-1, 8\cdot3^{m-3}}
=2S'_{0, 2\cdot3^{m-3}, k-8\cdot3^{m-3}}+L_{242},
\end{aligned}
\end{equation}
where $\min\{v_3(L_{241}), v_3(L_{242})\}\ge v_3(S'_{0, 2\cdot3^{m-3}, k-8\cdot3^{m-3}})+1$.
Note that (\ref{4.32}) holds and
\begin{align}
v_3(g_{2,k}'(3^{m-1}-2\cdot3^{m-3}))=W_{m, k}
=v_3(S'_{0, 2\cdot3^{m-3}, k-8\cdot3^{m-3}}). \label{4.53}
\end{align}

Combining (\ref{4.30}) to (\ref{4.32}) and (\ref{4.50}) to (\ref{4.53}), we obtain 
\begin{align*}
v_3\Big(S_{22}+\sum_{i\in A_2}g_{2,k}'(i)\Big)
&=v_3(6S'_{0, 2\cdot3^{m-3}, k-8\cdot3^{m-3}}+L_{131}+L_{132}
+6S_{0, 4\cdot3^{m-3}, k-4\cdot3^{m-3}}+L_{241}+L_{242})\\
&\ge  \min\{v_3(S'_{0, 2\cdot3^{m-3}, k-8\cdot3^{m-3}}),
v_3(S_{0, 4\cdot3^{m-3}, k-4\cdot3^{m-3}})\}+1 = W_{m,k}+1.
\end{align*}
Thus (\ref{4.42}) is proved in this subcase.

%*******************************Proof of Claim (I)--case 2.5 *******************************
\smallskip
\noindent\emph{Subcase 2.5: $k\in K_{25}$.}

In this subcase, $k=10\cdot3^{m-3}$, $L_{22}=0$, 
and $A_2=\{3^{m-1}-2\cdot3^{m-3}\}$. Moreover, 
\begin{align}
&S_{22}=s^{(1)}_{n, m-2, 0}L_{24}
=s^{(1)}_{n, m-2, 0}(s^{(1)}_{n, m-1, 2\cdot3^{m-3}})^2
=S'_{0, 2\cdot3^{m-3}, 2\cdot3^{m-3}}. \label{4.54}
\end{align}
Since $8\cdot3^{m-3}\in K_2$, Remark {\ref{rem3.2}} gives 
\begin{align}
g_{2,k}'(3^{m-1}-2\cdot3^{m-3})
=s^{(1)}_{n, m-1, 2\cdot3^{m-3}}\cdot s^{(2)}_{n, m-1, 8\cdot3^{m-3}}
=2S'_{0, 2\cdot3^{m-3}, 2\cdot3^{m-3}}+L_{251},\label{4.55}
\end{align}
where $v_3(L_{251})\ge v_3(S'_{0, 2\cdot3^{m-3}, 2\cdot3^{m-3}})+1$.

Note that
\begin{align}
v_3(g_{2,k}'(3^{m-1}-2\cdot3^{m-3}))=W_{m, k}
=v_3(S'_{0, 2\cdot3^{m-3}, 2\cdot3^{m-3}}). \label{4.56}
\end{align}
Then it follows from (\ref{4.54}) to (\ref{4.56}) that
\begin{align*}
v_3\Big(S_{22}+\sum_{i\in A_2}g_{2,k}'(i)\Big)
=v_3(3S'_{0, 2\cdot3^{m-3}, 2\cdot3^{m-3}}+L_{251})
\ge v_3(S'_{0, 2\cdot3^{m-3}, 2\cdot3^{m-3}})+1
= W_{m,k}+1.
\end{align*}
This proves (\ref{4.42}) in this subcase.

%*******************************Proof of Claim (I)--case 2.6 *******************************
\smallskip
\noindent\emph{Subcase 2.6: $k\in K_{26}$.}

In this subcase, $m\ge 4$, $2+10\cdot3^{m-3}\le k\le 4\cdot3^{m-2}-2$, 
$L_{22}=0$, and 
\[
A_2=3^{m-1}-\{2\cdot3^{m-3},
4\cdot3^{m-3}, k-8\cdot3^{m-3}, k-10\cdot3^{m-3}\}. 
\]
Moreover, 
\begin{align}
S_{22}&=s^{(1)}_{n, m-2, 0}L_{24}
=2S'_{0, 2\cdot3^{m-3}, k-8\cdot3^{m-3}}
+2S'_{0, 4\cdot3^{m-3}, k-10\cdot3^{m-3}}.\label{4.57}
\end{align}
Since $k-2\cdot3^{m-3}, k-4\cdot3^{m-3}, 10\cdot3^{m-3}\in K_2$,
(\ref{4.51}) remains valid. By Remark {\ref{rem3.2}}, we have 
\begin{equation}\label{4.58}
\begin{aligned}
g_{2,k}'(3^{m-1}-4\cdot3^{m-3})
&=s^{(1)}_{n, m-1, 4\cdot3^{m-3}}\cdot s^{(2)}_{n, m-1, k-4\cdot3^{m-3}}
=2S'_{0, 4\cdot3^{m-3}, k-10\cdot3^{m-3}}+L_{261},\\
g_{2,k}'(3^{m-1}-(k-10\cdot3^{m-3}))
&=s^{(1)}_{n, m-1, k-10\cdot3^{m-3}}\cdot s^{(2)}_{n, m-1, 10\cdot3^{m-3}}
=2S'_{0, 4\cdot3^{m-3}, k-10\cdot3^{m-3}}+L_{262},
\end{aligned}
\end{equation}
where $\min\{v_3(L_{261}), v_3(L_{262})\}\ge
v_3(S'_{0, 4\cdot3^{m-3}, k-10\cdot3^{m-3}})+1$.

Note that (\ref{4.53}) holds and we have
\begin{align}
v_3(g_{2,k}'(3^{m-1}-4\cdot3^{m-3}))=W_{m, k}
=v_3(S'_{0, 4\cdot3^{m-3}, k-10\cdot3^{m-3}}). \label{4.60}
\end{align}
By (\ref{4.51}), (\ref{4.53}) and (\ref{4.57}) to (\ref{4.60}), we obtain
\begin{align*}
v_3\Big(S_{22}+\sum_{i\in A_2}g_{2,k}'(i)\Big)
&=v_3(6S'_{0, 2\cdot3^{m-3}, k-8\cdot3^{m-3}}+L_{241}+L_{242}
+6S'_{0, 4\cdot3^{m-3}, k-10\cdot3^{m-3}}+L_{261}+L_{262})\\
&\ge \min\{v_3(S'_{0, 2\cdot3^{m-3}, k-8\cdot3^{m-3}}),
v_3(S'_{0, 4\cdot3^{m-3}, k-10\cdot3^{m-3}})\}+1 = W_{m,k}+1.
\end{align*}
Thus (\ref{4.42}) is proved in this subcase.

Combining Subcases 2.1 to 2.6, we have proved (\ref{4.42}). 
Hence (\ref{4.9}) holds in Case 2.

%*******************************Proof of Claim (I)--case 3*******************************

\smallskip
\noindent\emph{Case 3: $k\in K_{4}$.}

Then $m\ge 3$ and $A_1=3^{m-1}-\{2\cdot3^{m-2},k-4\cdot3^{m-2}\}$.
Thus the contribution from $A_1$ consists of two terms. 
Recall from (\ref{4.33}) that 
$g_{2,k}'(3^{m-2})=s^{(1)}_{n, m-2, 0}\cdot s^{(2)}_{n, m-1, k-2\cdot3^{m-2}}.$
Moreover, 
\begin{align}\label{4.61}
&g_{2,k}'(3^{m-1}-(k-4\cdot3^{m-2}))
=s^{(1)}_{n, m-1, k-4\cdot3^{m-2}}\cdot s^{(2)}_{n, m-1, 4\cdot3^{m-2}}.
\end{align}

We now expand the two $s^{(2)}$-terms appearing above. 
Setting $(n,m,k)\mapsto(n, m-1, 4\cdot3^{m-2})$ in Remark \ref{rem3.2}, we get 
\begin{align}\label{4.62}
s^{(2)}_{n, m-1, 4\cdot3^{m-2}}=(s^{(1)}_{n, m-2, 0})^2+L_{41}
=S''_{0, 0}+L_{41},
\end{align}
where $v_3(L_{41})\ge v_3(S''_{0, 0})+2$. 
Since $k-2\cdot3^{m-2}\in K_2$,
replacing $k$ by $k-2\cdot3^{m-2}$ in (\ref{4.36}) gives 
\begin{align}\label{4.63}
s^{(2)}_{n, m-1, k-2\cdot3^{m-2}}
=2s^{(1)}_{n, m-2, 0}\cdot s^{(1)}_{n, m-1, k-4\cdot3^{m-2}}+L_{42}+L_{43}
=2S'_{0, k-4\cdot3^{m-2}}+L_{42}+L_{43},
\end{align}
where $v_3(L_{43})\ge v_3(S'_{0, k-4\cdot3^{m-2}})+2$ and
$v_3(L_{42})\ge v_3(S'_{0, k-4\cdot3^{m-2}})+1$. 
Moreover, 
\[
L_{42}=
\begin{cases}
2s^{(1)}_{n,m-2,2}s^{(1)}_{n,m-1,0}
+2s^{(1)}_{n,m-1,2\cdot3^{m-3}}s^{(1)}_{n,m-1,k-8\cdot3^{m-3}}
+2s^{(1)}_{n,m-1,4\cdot3^{m-3}}s^{(1)}_{n,m-1,k-10\cdot3^{m-3}},
& k\in K_{41},\\[1mm]
2s^{(1)}_{n,m-1,2\cdot3^{m-3}}s^{(1)}_{n,m-1,k-8\cdot3^{m-3}}
+2s^{(1)}_{n,m-1,4\cdot3^{m-3}}s^{(1)}_{n,m-1,k-10\cdot3^{m-3}},
& k\in K_{42},\\[1mm]
2s^{(1)}_{n,m-2,2}s^{(1)}_{n,m-1,0}
+\left(s^{(1)}_{n,m-1,4\cdot3^{m-3}}\right)^2,
& k\in K_{43},\ m=3,\\[1mm]
\left(s^{(1)}_{n,m-1,4\cdot3^{m-3}}\right)^2,
& k\in K_{43},\ m\ge4,\\[1mm]
2s^{(1)}_{n,m-1,4\cdot3^{m-3}}s^{(1)}_{n,m-1,k-10\cdot3^{m-3}},
& k\in K_{44},\\[1mm]
0, & k\in K_{45}.
\end{cases}
\]

Substituting (\ref{4.61}) to (\ref{4.63}) into (\ref{4.33}), we obtain 
\begin{align}\label{4.64}
\sum_{i\in A_1}g_{2,k}'(i)=
g_{2,k}'(3^{m-2})+g_{2,k}'(3^{m-1}-(k-4\cdot3^{m-2}))=S_{41}+S_{42}+S_{43},
\end{align}
where 
\begin{align*}
S_{41}&:=3(s^{(1)}_{n, m-2, 0})^2s^{(1)}_{n, m-1, k-4\cdot3^{m-2}}
=3S''_{0, 0, k-4\cdot3^{m-2}}, \\
S_{42}&:=s^{(1)}_{n, m-2, 0}L_{42}, \quad \quad
S_{43}:=s^{(1)}_{n, m-2, 0}L_{43}+s^{(1)}_{n, m-1, k-4\cdot3^{m-2}}L_{41}.
\end{align*}
By the definition of $A_1$, together with (\ref{4.61}) and (\ref{4.62}), we have 
\begin{align}\label{4.65}
&v_3(g_{2,k}'(3^{m-1}-(k-4\cdot3^{m-2})))= W_{m,k}-1
=v_3(S''_{0, 0, k-4\cdot3^{m-2}}).
\end{align}
Consequently, 
\begin{align}\label{4.66}
&v_3(S_{41})= v_3(S''_{0, 0, k-4\cdot3^{m-2}})+1= W_{m,k}, 
\end{align}
and the bounds for $L_{41}$ and $L_{43}$ imply 
\begin{align}\label{4.67}
&v_3(S_{43})\ge v_3(S''_{0, 0, k-4\cdot3^{m-2}})+2
= W_{m,k}+1.
\end{align}

It remains to deal with $S_{42}$. We first observe that 
$$
S_{42}=0
\quad\Longleftrightarrow\quad
k\in K_{45}
\quad\Longleftrightarrow\quad
A_2=\varnothing.
$$
If $k\in K_{45}$, then $S_{42} = 0$ and $A_2=\varnothing$. 
Hence by (\ref{4.64}), (\ref{4.66}) and (\ref{4.67}), 
\begin{align*}
v_3(S_A')&=v_3\Big(\sum_{i\in A}g_{2,k}'(i)\Big)
=v_3\Big(\sum_{i\in A_1}g_{2,k}'(i)\Big)
=v_3 (S_{41}+S_{43})=v_3(S_{41})=W_{m,k}.
\end{align*}
Thus (\ref{4.9}) holds when $k\in K_{45}$. 

It remains to consider $k\in \bigcup_{j=1}^4 K_{4j}$. 
In this case, $S_{42} \ne 0$ and $A_2\ne \varnothing$. 
We claim that 
\begin{align}\label{4.68}
v_3\Big(S_{42}+\sum_{i\in A_2}g_{2,k}'(i)\Big)\ge W_{m,k}+1.
\end{align}
Assuming (\ref{4.68}), it follows from (\ref{4.64}), (\ref{4.66}) and (\ref{4.67}) that
\begin{align*}
v_3(S_A')&=v_3\Big(\sum_{i\in A}g_{2,k}'(i)\Big)
=v_3\Big(\sum_{i\in A_1}g_{2,k}'(i)+\sum_{i\in A_2}g_{2,k}'(i)\Big)\\
&=v_3\Big(S_{41}+S_{42}+S_{43}+\sum_{i\in A_2}g_{2,k}'(i)\Big)
=v_3(S_{41})=W_{m,k}.
\end{align*}
Thus (\ref{4.9}) holds in Case 3 once (\ref{4.68}) is proved.

We now prove (\ref{4.68}). The argument follows the same pattern as in 
Case 1 and 2: the secondary term $S_{42}$ is combined with the contribution from $A_2$. 
Since $A_2$ has four possible forms, corresponding to $k\in K_{41}, \ldots, K_{44}$, 
we treat these cases separately. In each case, we prove that the combined contribution 
has 3-adic valuation at least $W_{m,k}+1$.

%*******************************Proof of Claim (I)--case 3.1*******************************

\smallskip
\noindent\emph{Subcase 3.1: $k\in K_{41}$.}

In this subcase, $k=2+4\cdot 3^{m-2}$, $m\ge 4$, and 
\[ 
A_2=3^{m-1}-\{0, k-2\cdot3^{m-2}, 2\cdot3^{m-3},
4\cdot3^{m-3}, k-8\cdot3^{m-3}, k-10\cdot3^{m-3}\}.
\]
 Moreover, 
\begin{align}\label{4.69}
S_{42}&=s^{(1)}_{n, m-2, 0}L_{42}
=2S''_{0, 2, 0}+2S'_{0, 2\cdot3^{m-3}, k-8\cdot3^{m-3}}
+2S'_{0, 4\cdot3^{m-3}, k-10\cdot3^{m-3}}.
\end{align}
Since $k=2+4\cdot 3^{m-2}$, we have $s^{(2)}_{n, m-1, k}=s^{(2)}_{n, m-2, 2}$.
Thus Remark \ref{rem3.2} gives 
\begin{equation}\label{4.70}
\begin{aligned}
g_{2,k}'(3^{m-1})
&=s^{(1)}_{n, m-1, 0}\cdot s^{(2)}_{n, m-1, k}
=2S''_{0, 2, 0}+L_{411},\\
g_{2,k}'(3^{m-1}-(k-2\cdot3^{m-2}))
&=g_{2,k}'(3^{m-2}-2)
=s^{(1)}_{n, m-2, 2}\cdot s^{(2)}_{n, m-1, 2\cdot3^{m-2}}
=2S''_{0, 2, 0}+L_{412},
\end{aligned}
\end{equation}
where $\min\{v_3(L_{411}), v_3(L_{412})\}\ge v_3(S''_{0, 2, 0})+1$.

Since $k-2\cdot3^{m-3}, k-4\cdot3^{m-3}\in K_2$,
(\ref{4.51}), (\ref{4.53}), (\ref{4.58}) and (\ref{4.60}) hold, and we have
\begin{align}\label{4.72}
&v_3(g_{2,k}'(3^{m-1}))=W_{m, k}=v_3(S''_{0, 2, 0}).
\end{align}
Combining (\ref{4.69}) to (\ref{4.72}), (\ref{4.51}), (\ref{4.53}), 
(\ref{4.58}) and (\ref{4.60}), we obtain
\begin{align*}
v_3\Big(S_{42}+\sum_{i\in A_2}g_{2,k}'(i)\Big)
&=v_3(6S''_{0, 2, 0}+L_{411}+L_{412}
+6S'_{0, 2\cdot3^{m-3}, k-8\cdot3^{m-3}}+L_{241}+L_{242}\\
&\quad \ \ \ \ +6S'_{0, 4\cdot3^{m-3}, k-10\cdot3^{m-3}}+L_{261}+L_{262})\\
&\ge \min\{v_3(S''_{0, 2, 0}),
v_3(S'_{0, 2\cdot3^{m-3}, k-8\cdot3^{m-3}}),
v_3(S'_{0, 4\cdot3^{m-3}, k-10\cdot3^{m-3}})\}+1 = W_{m,k}+1.
\end{align*}
So (\ref{4.68}) is proved in this subcase.

%*******************************Proof of Claim (I)--case 3.2*******************************

\smallskip
\noindent\emph{Subcase 3.2: $k\in K_{42}$.}

In this subcase, $m\ge 4$, $4+4\cdot3^{m-2}\le k\le 14\cdot3^{m-3}-2$,
and 
\[
A_2=3^{m-1}-\{2\cdot3^{m-3}, 4\cdot3^{m-3},
k-8\cdot3^{m-3}, k-10\cdot3^{m-3}\}.
\]
 Moreover, 
\begin{align}\label{4.73}
S_{42}&=s^{(1)}_{n, m-2, 0}L_{42}
=2S'_{0, 2\cdot3^{m-3}, k-8\cdot3^{m-3}}+2S'_{0, 4\cdot3^{m-3}, k-10\cdot3^{m-3}}.
\end{align} 
Since $k-2\cdot3^{m-3}, k-4\cdot3^{m-3}\in K_2$, the estimates 
(\ref{4.51}), (\ref{4.53}), (\ref{4.58}) and (\ref{4.60}) apply. 
Combining these estimates with (\ref{4.73}) we obtain that
\begin{align*}
v_3\Big(S_{42}+\sum_{i\in A_2}g_{2,k}'(i)\Big)
&=v_3(6S'_{0, 2\cdot3^{m-3}, k-8\cdot3^{m-3}}+L_{241}+L_{242}
+6S'_{0, 4\cdot3^{m-3}, k-10\cdot3^{m-3}}+L_{261}+L_{262})\\
&\ge \min\{v_3(S'_{0, 2\cdot3^{m-3}, k-8\cdot3^{m-3}}),
v_3(S'_{0, 4\cdot3^{m-3}, k-10\cdot3^{m-3}})\}+1 = W_{m,k}+1.
\end{align*}
Thus (\ref{4.68}) is proved in this subcase.

%*******************************Proof of Claim (I)--case 3.3*******************************
\smallskip
\noindent\emph{Subcase 3.3: $k\in K_{43}$.}

In this subcase, $k=14\cdot 3^{m-3}$. We consider two cases 
according to whether $m=3$ or $m\ge 4$. 
If $m=3$, then $A_2=3^{m-1}-\{0, k-2\cdot3^{m-2}, 4\cdot3^{m-3}\}$ and
\begin{align}\label{4.74}
S_{42}
&=s^{(1)}_{n, m-2, 0}L_{42}
=2S''_{0, 2, 0}+S'_{0, 4\cdot3^{m-3}, 4\cdot3^{m-3}}.
\end{align}
Since $m=3$ and $k=14\cdot 3^{m-3}$, one has $s^{(2)}_{n, m-1, k}=s^{(2)}_{n, m-2, 2}$.
Thus (\ref{4.70}) and (\ref{4.72}) hold. It follows from Remark \ref{rem3.2} that
\begin{align}\label{4.75}
g_{2,k}'(3^{m-1}-4\cdot3^{m-3})
=s^{(1)}_{n, m-1, 4\cdot3^{m-3}}\cdot s^{(2)}_{n, m-1, 10\cdot3^{m-3}}
=2S'_{0, 4\cdot3^{m-3}, 4\cdot3^{m-3}}+L_{431}, 
\end{align}
where $v_3(L_{431})\ge v_3(S'_{0, 4\cdot3^{m-3}, 4\cdot3^{m-3}})+1$.
Note that
\begin{align}\label{4.76}
&v_3(g_{2,k}'(3^{m-1}-4\cdot3^{m-3}))=W_{m, k}
=v_3(S'_{0, 4\cdot3^{m-3}, 4\cdot3^{m-3}}).
\end{align}
Combining (\ref{4.74}) to (\ref{4.76}), (\ref{4.70}) and (\ref{4.72}), we get 
\begin{align*}
v_3\Big(S_{42}+\sum_{i\in A_2}g_{2,k}'(i)\Big)
&=v_3(6S''_{0, 2, 0}+L_{411}+L_{412}+3S'_{0, 4\cdot3^{m-3}, 4\cdot3^{m-3}}+L_{431})\notag\\
&\ge \min\{v_3(S''_{0, 2, 0}), v_3(S'_{0, 4\cdot3^{m-3}, 4\cdot3^{m-3}})\}+1 = W_{m,k}+1.
\end{align*}

If $m\ge 4$, then $A_2=3^{m-1}-\{4\cdot3^{m-3}\}$ and
\begin{align}\label{4.77}
S_{42}=s^{(1)}_{n, m-2, 0}L_{42}
=s^{(1)}_{n, m-2, 0}(s^{(1)}_{n, m-1, 4\cdot3^{m-3}})^2
=S'_{0, 4\cdot3^{m-3}, 4\cdot3^{m-3}}.
\end{align}
By (\ref{4.75}) to (\ref{4.77}), we obtain 
\begin{align*}
v_3\Big(S_{42}+\sum_{i\in A_2}g_{2,k}'(i)\Big)
=v_3(3S'_{0, 4\cdot3^{m-3}, 4\cdot3^{m-3}}+L_{431})
\ge v_3(S'_{0, 4\cdot3^{m-3}, 4\cdot3^{m-3}})+1=W_{m,k}+1.
\end{align*}
Thus (\ref{4.68}) is proved in this subcase.

%*******************************Proof of Claim (I)--case 3.4*******************************

\smallskip
\noindent\emph{Subcase 3.4: $k\in K_{44}$.}

In this subcase,  $m\ge 4$ and $A_2=3^{m-1}-\{4\cdot3^{m-3}, k-10\cdot3^{m-3}\}$. 
Moreover, 
\begin{align}\label{4.78}
S_{42}&=s^{(1)}_{n, m-2, 0}L_{42}
=2S'_{0, 4\cdot3^{m-3}, k-10\cdot3^{m-3}}.
\end{align}
Since $m\ge 4$ and $2+14\cdot3^{m-3}\le k\le 16\cdot3^{m-3}-2$,
we have $k-4\cdot3^{m-3}\in K_2$. Hence the estimates 
(\ref{4.58}) and (\ref{4.60}) apply. Combining them with (\ref{4.78}), we get 
\begin{align*}
v_3\Big(S_{42}+\sum_{i\in A_2}g_{2,k}'(i)\Big)
&=v_3(6S'_{0, 4\cdot3^{m-3}, k-10\cdot3^{m-3}}+L_{261}+L_{262})\notag\\
&\ge v_3(S'_{0, 4\cdot3^{m-3}, k-10\cdot3^{m-3}})+1=W_{m,k}+1.
\end{align*}
Thus (\ref{4.68}) is proved in this subcase.

Combining Subcases 3.1 to 3.4, we have proved (\ref{4.68}). 
Hence (\ref{4.9}) holds for $k\in K_4$.

Combining Cases 1 to 3, together with the previously treated cases 
$k\in K_3$ and $k\in K_5$, we obtain (\ref{4.9}) for all admissible $k$. 
Therefore (\ref{4.6}) follows, and the analysis of the leading contribution is complete. 

We now complete the induction step. 
By \eqref{4.3} and the above decomposition, 
\[
s^{(1)}_{n+1, m, k}=S_A+S_B+S_C.
\] 
By (\ref{4.6}), we have $v_3(S_A)=W_{m,k}$. 
Moreover, by (\ref{4.4}) and the definition of $A$, 
we have $v_3(S_B)\ge W_{m,k}+1$ and $v_3(S_C)\ge W_{m,k}+1$. 
It follows that 
\[
v_3(s^{(1)}_{n+1, m, k})=v_3(S_A+S_B+S_C)=W_{m, k}.
\] 
Thus \eqref{4.2} holds. Hence \eqref{1.4} holds for $a=1$ and even $k$ at
level $n+1$. By Lemma \ref{lem2.3}, it also holds for odd $k$. Therefore
\eqref{1.4} holds for $a=1$ at level $n+1$.

By induction, formula \eqref{1.4} holds for $a=1$ for all positive integers
$n$. Finally, Lemma \ref{lem3.1} gives the corresponding formula for $a=2$.
This completes the proof of Theorem \ref{thm1}.

\section{Consequences of the main formula}

In this section we derive three consequences of the main formula.  
The first compares two adjacent orders, namely $a3^n+1$ and $a3^n$. 
The second identifies the maximal 3-adic valuation among the families 
$\{s(3^n, k)\}$ and $\{s(2\cdot 3^n, k)\}$. The final consequence translates 
these bounds into estimates for elementary symmetric functions, giving 
a partial verification of the conjecture of Leonetti and Sanna in the present 3-adic setting.  

\subsection{Proof of Theorem \ref{thm2}}

We first prove Theorem \ref{thm2}, which compares $v_3(s(a3^n+1, k+1))$ 
with the 3-adic valuations of Stirling numbers of order $a3^n$.

Let $a\in\{1,2\}$, and let $n, k\in \mathbb{Z}^+$ with $k\le a3^n$.

We first deal with the two boundary cases. If $k=a3^n$, then 
$2\mid (k-a)$, and 
$v_3(s(a3^n+1,a3^n+1))=v_3(1)=v_3(s(a3^n, a3^n))=0$. 
Thus (\ref{1.7}) holds in this case. If $k=a3^n-1$,
then $2\nmid (k-a)$, and $v_3(s(a3^n+1,a3^n))=v_3(\binom{a3^n+1}{2})
=n$ and $v_3(s(a3^n, a3^n))=0$. Hence (\ref{1.7}) also holds for $k=a3^n-1$.
So Theorem \ref{thm2} holds when $k\in \{a3^n-1, a3^n\}$.

It remains to consider $1\le k\le a3^n-2$. By the recurrence relation for 
Stirling numbers of the first kind (see, for example, \cite{[LC]}), we have 
\begin{align}\label{5.1}
s(a3^n+1, k+1)=a3^ns(a3^n, k+1)+s(a3^n, k).
\end{align}
We distinguish two cases according to the parity of $k-a$. 

If $2\mid(k-a)$, then $2\nmid(k+1-a)$, and so we may write
$k+1=a3^n-2t-1$, where $t\in \mathbb{Z}$ with $0\le t\le \frac{a3^n-3}{2}$.
By Lemma \ref{lem2.2} and Lemma \ref{lem2.7}, we deduce that
\begin{align}\label{5.2}
v_3(a3^ns(a3^n, k+1))&=v_3(s(a3^n, a3^n-2t-1))+n
=v_3(s(a3^n, a3^n-2t))+v_3(2t+1)+2n\notag\\
&= v_3(s(a3^n, k+2))+v_3(k+1)+2n\notag\\
&\ge  v_3(s(a3^n, k))+v_3(k+1)+2 >v_3(s(a3^n, k)).
\end{align}
From (\ref{5.1}) and (\ref{5.2}), we obtain 
\[
v_3(s(a3^n+1, k+1))=v_3(s(a3^n, k)).
\]
This proves (\ref{1.7}) when $2\mid(k-a)$.

If $2\nmid(k-a)$, then we can write $k=a3^n-2t'-1$
with $t'\in \mathbb{Z}$ and $0\le t'\le \frac{a3^n-2}{2}$.
Again, by Lemma \ref{lem2.2}, we have
\begin{align}\label{5.3}
v_3(s(a3^n, k))&=v_3(s(a3^n, a3^n-2t'-1))
=v_3(s(a3^n, a3^n-2t'))+v_3(2t'+1)+n\notag\\
&=v_3(s(a3^n, k+1))+v_3(a3^n-k)+n
\ge  v_3(s(a3^n, k+1))+n.
\end{align}
Hence by (\ref{5.1}) and (\ref{5.3}) we obtain 
\begin{align*}
v_3(s(a3^n+1, k+1))&\ge \min\{v_3(a3^ns(a3^n, k+1)),v_3(s(a3^n, k))\}=v_3(s(a3^n, k+1))+n.
\end{align*}
Thus (\ref{1.7}) holds when $2\nmid(k-a)$.
This finishes the proof of Theorem \ref{thm2}. 

\subsection{Proofs of Theorems \ref{thm3} and \ref{thm4}}

We next prove Theorems \ref{thm3} and \ref{thm4}. 
Both arguments have the same structure: using Theorem \ref{thm1}, 
we compare  $v_3(s(a3^n, k))$ with the 3-adic valuation at $k=1$, 
and show that the latter gives the desired upper bound. 

\medskip
\noindent{\it Proof of Theorem \ref{thm3}.}
For $n=1$, we have $k\in\{1,2,3\}$. Since $s(3,1)=2$,
$s(3,2)=3$ and $s(3,3)=1$, (\ref{1.8}) is immediate.

Assume now that $n\ge 2$. The boundary cases are 
again immediate: $v_3(s(3^n, 3^n))=0$ and
$v_3(s(3^n, 3^n-1))=n$. Hence (\ref{1.8})
holds when $k\in \{3^n-1, 3^n\}$.

It remains to consider $1\le k\le 3^n-2$. Write $k=3^l-j$, 
where $(l,j)\in T_{1,n}$. By Theorem \ref{thm1},
\begin{align}\label{5.4}
v_3(s_{n, l, j}^{(1)})
=\frac{1}{2}(3^n-3^l)-(n-l)(3^l-j)+l-1
-v_3\Big(\Big\lfloor\frac{j}{2}\Big\rfloor\Big)+(l+v_3(j))\epsilon_j.
\end{align}

We first handle the exceptional case $n=2$. 
If $l=1$, then $j=2$ and (\ref{5.4}) gives $v_3(s_{n, l, j}^{(1)})=2<4$. 
If $l=2$, then $2\le j\le 7$, and (\ref{5.4}) gives
\[
v_3(s_{n, l, j}^{(1)})
=1-v_3\Big(\Big\lfloor\frac{j}{2}\Big\rfloor\Big)+(2+v_3(j))\epsilon_j\le 3+v_3(j)\le 4.
\]
Therefore (\ref{1.8}) holds when $n=2$.

We now assume $n\ge3$. Since $1\le l\le n$, $2\le j\le 2\cdot 3^{l-1}+1$ and $j<3^l$, 
we have $v_3(j)\le l-1$. By (\ref{5.4}), we obtain 
\begin{align*}
v_3(s(3^n, 1))-v_3(s(3^n, k))
&=\frac{1}{2}(3^n-2n-1)-v_3(s_{n, l, j}^{(1)})\\
&=\frac{1}{2}(3^l-1)+(n-l)(3^l-j-1)-2l+1
+v_3\Big(\Big\lfloor\frac{j}{2}\Big\rfloor\Big)-(l+v_3(j))\epsilon_j=:D_{l,j}^{(1)}.
\end{align*}
It remains to show that $D_{l,j}^{(1)}\ge 0$. 

If $j$ is even, then $2\le j\le 2\cdot 3^{l-1}$, and
\begin{align*}
D_{l,j}^{(1)}&=\frac{1}{2}(3^l-1)+(n-l)(3^l-j-1)-2l+1+v_3(j)\\
&\ge \frac{1}{2}(3^l-1)+(n-l)(3^{l-1}-1)-2l+1
\ge \frac{1}{2}(3^l-1)-2l+1\ge 0.
\end{align*}

If $j$ is odd, then $l\ge 2$, $3\le j\le 2\cdot 3^{l-1}+1$.  
In this case, 
\begin{align*}
D_{l,j}^{(1)}=\frac{1}{2}(3^l-1)+(n-l)(3^l-j-1)-3l-v_3(j)+v_3(j-1)+1.
\end{align*}
For $l=2$, one checks directly that $D_{2,j}^{(1)}\ge 1$ holds for $j\in \{3, 5, 7\}$.
For $l\ge 3$, we have 
\begin{align*}
D_{l,j}^{(1)}&\ge \frac{1}{2}(3^l-1)+(n-l)(3^{l-1}-2)-3l-(l-1)+1
\ge \frac{1}{2}(3^l-1)-4l+2\ge 3.
\end{align*}
Thus $D_{l,j}^{(1)}\ge 0$ in all cases, and 
hence (\ref{1.8}) holds for $n\ge 3$.
This completes the proof of Theorem \ref{thm3}. \qed

\medskip
The proof of Theorem \ref{thm4} is parallel, using the 
formula in Theorem \ref{thm1} with $a=2$. We give the details. 

\medskip
\noindent{\it Proof of Theorem \ref{thm4}.} 
First, the boundary cases follow from 
$v_3(s(2\cdot3^n, 2\cdot3^n))=0$ and 
$v_3(s(2\cdot3^n, 2\cdot3^n-1))=n$. 
Thus (\ref{1.9}) holds for $k\in \{2\cdot3^n-1, 2\cdot3^n\}$.

Now let $1\le k\le 2\cdot3^n-2$. Write $k=2\cdot3^l-j$ with 
$(l,j)\in T_{2,n}$. By Theorem \ref{thm1}, 
\begin{align}\label{5.5}
v_3(s_{n, l, j}^{(2)})
=3^n-3^l-(n-l)(2\cdot3^l-j)+l-1
-v_3\Big(\Big\lfloor\frac{j}{2}\Big\rfloor\Big)+(l+v_3(j))\epsilon_j.
\end{align}

We first handle the exceptional case $n=1$.  
In this case, we have $l=1$ and $2\le j\le 5$.
By (\ref{5.5}), 
\[
v_3(s(2\cdot3^n, k))=(1+v_3(j))\epsilon_j-v_3\Big(\Big\lfloor\frac{j}{2}\Big\rfloor\Big)\le 2.
\]
Hence (\ref{1.9}) holds when $n=1$.

We now assume $n\ge2$. By (\ref{5.5}), we have 
\begin{align*}
&v_3(s(2\cdot3^n, 1))-v_3(s(2\cdot3^n, k))
=3^n-n-1-v_3(s_{n, l, j}^{(2)})\\
&=3^l+(n-l)(2\cdot3^l-j-1)-2l
+v_3\Big(\Big\lfloor\frac{j}{2}\Big\rfloor\Big)-(l+v_3(j))\epsilon_j=: D_{l,j}^{(2)}.
\end{align*}
It remains to show that $D_{l,j}^{(2)}\ge 0$. 

If $j$ is even, then $2\le j\le 4\cdot3^{l-1}$ and
\begin{align*}
&D_{l,j}^{(2)}=3^l+(n-l)(2\cdot3^l-j-1)-2l
+v_3(j)\ge (n-l)(2\cdot3^{l-1}-1)+3^l-2l\ge 1.
\end{align*}
Thus (\ref{1.9}) holds when $j$ is even.

If $j$ is odd, then $3\le j\le 4\cdot 3^{l-1}+1$ and $v_3(j)\le l$,
where $v_3(j)=l$ holds if and only if $j=3^l$. 
In this case, 
\begin{align*}
&D_{l,j}^{(2)}=3^l+(n-l)(2\cdot3^l-j-1)-3l-v_3(j)
+v_3(j-1).
\end{align*}
For $v_3(j)=l$, we have $j=3^l$ and
\[
D_{l,j}^{(2)}=(n-l)(3^l-1)+3^l-4l\ge 1
\]
since $1\le l\le n$ and $n\ge 2$.
For $v_3(j)\le l-1$, we have
\[
D_{l,j}^{(2)}\ge (n-l)(2\cdot3^{l-1}-2)+3^l-4l+1\ge 0.
\]
Thus $D_{l,j}^{(2)}\ge 0$ in all cases, and 
 (\ref{1.9}) follows for $n\ge 2$.
The proof of Theorem \ref{thm4} is complete. \qed

\subsection{Proof of Corollary \ref{cor2}}

Finally, we prove Corollary \ref{cor2}. The point is to 
translate the valuation estimate for Stirling numbers 
of the first kind into one for the corresponding elementary 
symmetric functions. 

Let $a\in\{1,2\}$. By (\ref{1.1}), we have 
\begin{align}\label{5.6}
H(a3^n, k)=\frac{s(a3^n+1, k+1)}{(a3^n)!}.
\end{align}

Assume that $2\mid(k-a)$, $n\ge 3$ and $1\le k\le a3^n$. 
Then, by (\ref{5.6}) and Theorems \ref{thm2} to \ref{thm4}, we obtain 
\begin{align*}
v_3(H(a3^n, k))&=v_3(s(a3^n+1, k+1))-v_3((a3^n)!)
=v_3(s(a3^n, k))-v_3((a3^n)!)\\
&\le  v_3(s(a3^n, 1))-v_3((a3^n)!)= v_3((a3^n-1)!)-v_3((a3^n)!)=-v_3(a3^n)=-n.
\end{align*}
This proves Corollary \ref{cor2}.

\section{Conclusions}

In this paper we proved the $3$-adic case of Conjecture \ref{cnj2} for
the Stirling numbers of the first kind. More precisely, we obtained an
explicit formula for $v_3(s(a3^n,k))$ 
for $a\in\{1,2\}$ and $1\le k\le a3^n$. As consequences, we derived a
comparison formula for the adjacent family $s(a3^n+1,k+1)$ and determined
the maximal $3$-adic valuations in the ranges considered above.

We also indicate why the present method does not directly extend to
arbitrary odd primes. The case $p=3$ is special because $p-1=2$, so the
residue class of $k$ modulo $p-1$ is completely determined by the parity
of $k$. In particular, every integer $k$ satisfies $k\equiv \epsilon_k \pmod 2$.
Thus the correction term $T_k$ in Conjecture \ref{cnj2} always falls into
its first case.

For primes $p\ge 5$, several nontrivial residue classes modulo $p-1$ occur,
and the Bernoulli-number term in Conjecture \ref{cnj2} can no longer be
absorbed into a parity distinction. Consequently, a corresponding
leading-term analysis would require additional congruence information and
a substantially finer decomposition of the parameter range.

\section*{Data availability}
No data was used for the research described in the article.
AMC
\section*{Declaration of generative AI and AI-assisted technologies in the manuscript preparation process}

During the preparation of this work, the authors used ChatGPT only for language polishing. 
The mathematical results, proofs, and conclusions were developed and verified by the authors. After using
this tool, the authors reviewed and edited the content as needed and take full responsibility
for the content of the published article.


\begin{thebibliography}{999}
\bibitem{[Am]} T. Amdeberhan, D. Manna and V. Moll, The 2-adic valuation of Stirling
numbers, {\it Experiment. Math.} {\bf 17} (2008), 69--82.
\bibitem{[CF]} F. Clarke, Hensel's lemma and the divisibility by primes of Stirling-like numbers,
{\it J. Number Theory} {\bf 52} (1995), 69--84.
\bibitem{[Da]} D.M. Davis, Divisibility by 2 of Stirling-like numbers,
{\it Proc. Amer. Math. Soc.} {\bf 110} (1990), 597--600.
\bibitem{[HZ1]} S.F. Hong, J.R. Zhao and W. Zhao, The 2-adic valuations of
Stirling numbers of the second kind, {\it Int. J. Number Theory} {\bf 8} (2012), 1057--1066.
\bibitem{[TL]} T. Lengyel, On the divisibility by 2 of Stirling numbers of the second kind,
{\it Fibonacci Quart.} {\bf 32} (1994), 194--201.
\bibitem{[TL2]} T. Lengyel, On the 2-adic order of Stirling numbers of the second kind
and their differences, {\it DMTCS Proc. AK} (2009), 561--572.
\bibitem{[TL3]} T. Lengyel, Alternative proofs on the 2-adic order of Stirling numbers
of the second kind, {\it Integers} {\bf 10} (2010), 453--463.
\bibitem{[Lu]} A.T. Lundell, A divisibility property for Stirling numbers,
{\it J. Number Theory} {\bf 10} (1978), 35--54.
\bibitem{[Mi]} P. Miska, On $p$-adic valuations of Stirling numbers,
{\it Acta Arith.} {\bf 186} (2018), 337--348.
\bibitem{[Wan]} S. De Wannermacker, On 2-adic orders of Stirling numbers of the second kind,
{\it Integers} {\bf 5} (2005), A21.
\bibitem{[HZ2]} J.R. Zhao, S.F. Hong and W. Zhao, Divisibility by 2 of Stirling numbers
of the second kind and their differences, {\it J. Number Theory} {\bf 140} (2014), 324--348.
\bibitem{[HZ3]} W. Zhao, J.R. Zhao and S.F. Hong, The 2-adic valuations of differences
of Stirling numbers of the second kind, {\it J. Number Theory} {\bf 153} (2015), 309--320.
\bibitem{[LS]} P. Leonetti and C. Sanna, On the $p$-adic valuation of
Stirling numbers of the first kind, {\it Acta Math. Hungar.} {\bf 151} (2017), 217--231.

\bibitem{[Alt]} C. Altunta\c{s},
On the $p$-adic valuation of generalized harmonic numbers,
{\it Bull. Korean Math. Soc.} {\bf 60} (2023), 933--955.
\bibitem{[Bo]} D.W. Boyd, A $p$-adic study of the partial sums of the harmonic series,
{\it Experiment. Math.} {\bf 3} (1994), 287--302.
\bibitem{[CCG]} L. Carofiglio, G. Cherubini and A. Gambini,
On Eswarathasan--Levine and Boyd's conjectures for harmonic numbers,
{\it Bull. Aust. Math. Soc.} {\bf 113} (2026), 15--25.
\bibitem{[CDFG]} L. Carofiglio, L. De Filpo and A. Gambini,
$p$-adic valuation of harmonic sums and their connections with Wolstenholme primes,
{\it Indian J. Pure Appl. Math.} {\bf 55} (2024), 555--566.
\bibitem{[CT]} Y.G. Chen and M. Tang, On the elementary symmetric functions of
$1, 1/2, ...,$ $1/n$, {\it Amer. Math. Monthly} {\bf 119} (2012), 862--867.
\bibitem{[EN]} P. Erd\H{o}s and I. Niven, Some properties of partial sums of
the harmonic series, {\it Bull. Amer. Math. Soc.} {\bf 52} (1946), 248--251.
\bibitem{[Es]} A. Eswarathasan and E. Levine, $p$-Integral harmonic sums,
{\it Discrete Math.} {\bf 91} (1991), 249--257.
\bibitem{[HW]} S.F. Hong and C.L. Wang, The elementary symmetric functions of
reciprocal arithmetic progressions, {\it  Acta Math. Hungar.} {\bf 144} (2014), 196--211.
\bibitem{[Ka]} K. Kamano, On 3-adic valuations of generalized harmonic numbers,
{\it Integers} {\bf 12} (2012), 311--319.

\bibitem{[LHQW]} Y.Y. Luo, S.F. Hong, G.Y. Qian and C.L. Wang,
The elementary symmetric functions of a reciprocal polynomial sequence,
{\it C. R. Math. Acad. Sci. Paris} {\bf 352} (2014), 269--272.
\bibitem{[N]} T. Nagell, Eine Eigenschaft gewisser Summen,
{\it Skr. Norske Vid. Akad. Kristiania} {\bf 13} (1923), 10--15.
\bibitem{[Sa]} C. Sanna, On the $p$-adic valuation of harmonic numbers,
{\it J. Number Theory} {\bf 166} (2016), 41--46.
\bibitem{[T]} L. Theisinger, Bemerkung \"uber die harmonische Reihe,
{\it Monatsh. Math. Phys.} {\bf 26} (1915), 132--134.
\bibitem{[WH]} C.L. Wang and S.F. Hong, On the integrality of the elementary
symmetric functions of $1, 1/3, ...,1/(2n-1)$,
{\it Math. Slovaca} {\bf 65} (2015), 957--962.
\bibitem{[WC]} B.L. Wu and Y.G. Chen, On certain properties of harmonic numbers,
{\it J. Number Theory} {\bf 175} (2017), 66--86.

\bibitem{[L]} T. Lengyel, On $p$-adic properties of the Stirling numbers of the first kind,
{\it J. Number Theory} {\bf 148} (2015), 73--94.
\bibitem{[KY]} T. Komatsu and P. Young, Exact $p$-adic valuations of
Stirling numbers of the first kind, {\it J. Number Theory} {\bf 177} (2017), 20--27.

\bibitem{[QH]} M. Qiu and S.F. Hong, 2-Adic valuations of Stirling numbers
of the first kind, {\it Int. J. Number Theory} {\bf 15}(2019), 1827--1855.
\bibitem{[HQ]} S.F. Hong and M. Qiu, On the $p$-adic properties of
Stirling numbers of the first kind, {\it Acta Math. Hungar.} {\bf 161}(2020), 366--395.

\bibitem{[LC]} L. Comtet, {\it Advanced combinatorics: The art of finite and infinite expansions,
Revised and Enlarged Edition}, D. Reidel Publishing Co., Dordrecht and Boston, 1974.

\bibitem{[VA]} V. Adamchik, On Stirling numbers and Euler sums,
{\it J. Comput. Appl. Math.} {\bf 79} (1997), 119--130.

\bibitem{[K]} N. Koblitz, {\it $p$-Adic numbers, $p$-adic analysis
and zeta-functions}, 2nd ed., GTM {\bf 58}, Springer-Verlag, New York, 1984.

\end{thebibliography}
\end{document}